\documentclass[10pt]{article}
\usepackage{graphicx} 
\usepackage[margin=1in]{geometry}

\usepackage[normalem]{ulem}
\usepackage{amsmath}
\usepackage{amssymb}
\usepackage{amsfonts}
\usepackage{mathtools}
\numberwithin{equation}{section}
\usepackage{bm}
\usepackage{amsthm}
\usepackage{mathrsfs}
\newtheorem{theorem}{Theorem}[section]
\newtheorem{proposition}[theorem]{Proposition}
\newtheorem{lemma}[theorem]{Lemma}

\newcommand{\neqm}{n_{\neq}}
\newcommand{\ceqm}{c_{\neq}}
\newcommand{\Lin}{\mathcal L}
\newcommand{\Non}{\mathcal N}
\usepackage{amsthm}
\usepackage{xcolor}
\theoremstyle{definition}
\newtheorem{definition}{Definition}[section]
\newcommand{\cred}{\color{red}}
\theoremstyle{remark}
\usepackage[colorlinks=true, citecolor=blue]{hyperref}

\title{Suppression of Blow-up in Two-dimensional Keller–Segel Systems by Stochastic Couette Flows
}
\author{Fanze Kong\footnote{Department of Applied Mathematics, University of Washington, Seattle, 98195, WA, USA, email address: {fzkong@uw.edu}.}~  and Krutika Tawri\footnote{Department of Applied Mathematics, University of Washington, Seattle, 98195, WA, USA, email address: {ktawri@uw.edu}.} }

\date{}

\begin{document}

\maketitle

 \begin{abstract}
We study Keller--Segel systems on ${\mathbb{T}\times\mathbb{R}}$ subject to a stochastic Couette transport flow.
We prove that sufficiently strong mixing suppresses chemotactic finite-time
blow-up with high probability; thus yielding a unique global-in-time mild solution for initial data of arbitrary mass. Moreover, we obtain a quantitative estimate for the probability of blow-up, showing that it decays exponentially with increasing strength of the stochastic shear.

Our analysis is based on the enhanced dissipation generated by the stochastic shear flow.
To prove our main result, we first construct a maximal local mild solution through pathwise estimates for the random solution operator of passive
scalar transport equations and a fixed-point argument.
To obtain global control, we decompose the evolution into time blocks and exploit mixing-induced decay of the nonzero spatial Fourier modes.
A nonlinear bootstrap argument then yields global existence above an almost-sure finite {random} threshold for the shear strength which is obtained by an application of the Borel-Cantelli theorem.
\end{abstract}

\section{Introduction}
In this paper, we study the following two-dimensional Keller--Segel model driven by stochastic Couette noise:
\begin{align}\label{eq1}
\left\{
\begin{array}{ll}
dn+\sqrt{A}y\partial_x n\circ dW_t
=
\Delta n\,dt
-\nabla\cdot(n\nabla c)\,dt,&(x,y)\in\Omega,t>0,
\\
-\Delta c=n,&(x,y)\in\Omega,t>0,\\
n(x,0)=n_{\rm in}(x)\ge ,\not\equiv 0,
& (x,y)\in\Omega.
\end{array}
\right.
\end{align}
Here \(n\) denotes the cell density and \(c\) is the concentration of the chemical signal. We choose the domain $\Omega=\mathbb{T}\times\mathbb{R},
$ where \(\mathbb{T}\) represents the one-dimensional torus  defined by $
\mathbb{T}=\mathbb{R}/(2\pi\mathbb{Z}).$  Here, \(n_{\rm in}\) is the initial data, \(A>0\) characterizes the strength of the Couette noise which is defined on a complete filtered probability space \((\mathfrak O,\mathcal F,(\mathcal F_t)_{t\ge0},\mathbb P)\) supporting a standard real-valued Brownian motion \(W_t\) adapted to \((\mathcal F_t)_{t\ge0}\). The stochastic term $\partial_x n \circ dW_t$ is interpreted in the Stratonovich sense.

System \eqref{eq1} can be viewed as an extension of the classical Patlak–Keller–Segel (PKS) model \cite{patlak1953random,Keller1970,Keller1971}. When the fluid advection is absent, the equations for the cell density and the chemoattractant reduce to the classical PKS system. The latter serves as a paradigm for the mathematical description of chemotaxis, a process ubiquitous in various physiological and pathological phenomena, such as wound healing, tumor growth  and morphogenesis. The analysis of the PKS equation has been extensively developed; see, for instance, the surveys \cite{hillen2009user,horstmann2003,horstmann2004}.

A fundamental feature of the two-dimensional PKS system is the critical-mass phenomenon. Let $$\displaystyle M:=\int_{\mathbb{R}^{2}} n(x,0)\,dx$$ denote the initial mass of cells. It is well established that, if \(M<8\pi\) then free-energy solutions exist globally in time; see e.g. \cite{blanchet2006two,dolbeault2004optimal}. In contrast, in the supercritical regime \(M>8\pi\), solutions may develop finite-time blow-up \cite{nanjundiah1973,childress1981,herrero1996,senba2000some,wang2002steady}. For critical initial mass \(M=8\pi\), global existence was first established for radially symmetric solutions in \cite{biler20068pi} and later extended to general solutions in \cite{velazquez2004point}. Despite global existence, solutions at the critical mass may still concentrate asymptotically: infinite-time blow-up with finite second moment was constructed in \cite{blanchet2008infinite}, while the structure and stability of the corresponding blow-up profiles were further investigated in \cite{davila2020existence}.

When chemotactic cell motion takes place in a fluid environment, fluid advection can significantly influence the aggregation dynamics. This has motivated extensive study of chemotaxis models coupled with fluid flow; see, for example, \cite{winkler2014stabilization,duan2010global,SimingHe,kong2024global}. In particular, sufficiently strong mixing flows, including shear flows, can enhance the dissipation of chemotactic aggregation and, in appropriate regimes, suppress finite-time blow-up in Keller--Segel and related chemotaxis-fluid systems \cite{KiselevXu2016,BedrossianHe2017,he2023enhanced,FengShiWang2022,ShiWang2024,ZengZhangZi2021,LiXiangXu2025,CuiWangWang2026}. More recently, stochastic transport has emerged as another mechanism capable of modifying aggregation dynamics. In stochastic Keller--Segel models, transport noise has been shown to delay blow-up with high probability and to yield quantitative estimates for the probability of blow-up \cite{FlandoliGaleatiLuo2021,FlandoliGaleatiLuo2024}.

In realistic biological and environmental settings, chemotactic cell motion is inevitably subject to random fluctuations arising from unresolved environmental effects, fluctuations in chemical signals, and variability in cell motion. It is therefore physically relevant to incorporate stochastic forcing into chemotaxis models such as \eqref{eq1}. Moreover, when chemotactic cells move in a fluctuating fluid environment, the ambient flow itself may be subject to random perturbations, naturally leading to stochastic chemotaxis-fluid systems. Such systems have recently attracted increasing attention; see, for instance, \cite{zhai20202d,zhang2025keller,kong2026global,ChenZhaiZhang2025}.

In this paper, we study the stochastic Keller–-Segel models \eqref{eq1} subject to a stochastic Couette flow.  We first present a formal derivation of  \eqref{eq1}  from a rapidly varying random Couette flow. Let $\xi:\mathbb{R}_{+}\times\mathfrak O\rightarrow\mathbb{R}$  be a stationary, mean-zero and  strongly mixing stochastic process. Thus, $\mathbb{E}\xi(t)=0$,
$t\ge0.$   
   For $\epsilon>0,$ define $W_t^\varepsilon 
:=\frac{1}{\sqrt{\varepsilon}}
\int_0^t
\xi\!\left(\frac{s}{\varepsilon}\right)\,\mathrm ds=
\sqrt{\varepsilon}\int_0^{t/\varepsilon}\xi(r)\,dr.$  Then $\frac{d}{dt}W_t^\varepsilon
=
\frac{1}{\sqrt{\varepsilon}}
\xi\left(\frac{t}{\varepsilon}\right).$  Now, we consider the following Keller--Segel system with advection by a random Couette flow:
\begin{equation}\label{eq:fast-random-shear}
\begin{cases}
\displaystyle
\partial_t n^\varepsilon
+
\sqrt{\frac{A}{\varepsilon}}\,
\xi\left(\frac{t}{\varepsilon}\right)
y\partial_x n^\varepsilon
=
\Delta n^\varepsilon
-\nabla\cdot(n^\varepsilon\nabla c^\varepsilon),
& (x,y)\in\Omega, t>0,
\\[2mm]
-\Delta c^\varepsilon=n^\varepsilon,
& (x,y)\in\Omega, t>0,\\
n(x,0)=n_{\rm in}(x),&(x,y)\in\Omega.
\end{cases}
\end{equation}
We assume that the initial data $n_{\rm in}$ is non-negative and not identically equal to zero.
Suppose that \(\xi\) satisfies the functional central limit theorem  $W^\varepsilon
\Rightarrow
\sigma W$ in $C([0,T])$ as \(\varepsilon\to0\) for every \(T<\infty\) with \(W\) being a standard Brownian motion and \(\sigma>0\).   
We have for the stationary process $\xi$ with integrable covariance, that the effective variance is formally given by  $\sigma^2
=
2\int_0^\infty
\operatorname{Cov}(\xi(0),\xi(s))\,ds.$
After a  normalization of \(\xi\), we take \(\sigma=1\).  Noting that \(W^\varepsilon\) has absolutely continuous sample paths, we invoke the Wong--Zakai principle \cite{WongZakai1965Convergence,WongZakai1965Relation} to formally arrive at (\ref{eq1}) in the limit \(\varepsilon\to0\).

 As discussed above, deterministic shear flows can suppress blow-up in Keller--Segel systems with fluid advection. This naturally raises the question of whether the enhanced-dissipation mechanism persists under stochastic advection, such as that induced by the stochastic Couette flow in \eqref{eq1}. In this paper, we answer this question positively on a set of high-probability.

{To address this question and facilitate the subsequent analysis, we introduce the rescaled time \(\tilde t:=t/A\) and Brownian motion \(\tilde W_t:=\sqrt{A}\,W_{t/A}\).} Dropping the tildes for simplicity, we obtain the following rescaled form of \eqref{eq1} in \(\Omega\):
\begin{align}\label{eq:KS-local}
\left\{
\begin{array}{ll}
dn+y\partial_x n\circ dW_t
=
\nu\Delta n\,dt
-\nu\nabla\cdot(n\nabla c)\,dt,&(x,y)\in\Omega,t>0,\
\\
-\Delta c=n,&(x,y)\in\Omega,t>0,\\
n(x,0)=n_{\rm in}(x)
& (x,y)\in\Omega,
\end{array}
\right.
\end{align}
where $\nu:=\frac{1}{A}$. Throughout this manuscript we assume that $n_{\rm in}$ satisfies  
\begin{align}\label{initial}
    n_{\rm in} \geq 0,\quad\text{ and } \quad n_{\rm in} \not\equiv 0.
\end{align}
Our goal is to prove that sufficiently strong shear (corresponding to small \(\nu\)) can suppress finite-time blow-up and ensure global well-posedness of \eqref{eq:KS-local} in two dimensions with high probability.

\subsection{Main results}
In this section, we state our main definitions and result concerning the global existence of mild solutions to \eqref{eq:KS-local}. To introduce our definition of a mild solution, we first introduce the random solution operator $S_\nu(t,s)$
associated with the linear equation,
\begin{equation}\label{eq:passive-local}
 \,d f+y\partial_xf\circ\,d W_t=\nu\Delta f\,\,d t,\qquad f(x,y,s)=f_s.
\end{equation}
Note that $S_\nu(t,r)S_\nu(r,s)=S_\nu(t,s)$, $s\le r\le t$ holds pathwise. 

Let $f_0$ and $f_{\neq}$ denote the orthogonal projections of $f \in L^2(\Omega)$
onto the zero and nonzero Fourier modes in $x$, respectively:
\begin{equation}\label{f0-def}
  f_0(y,t):=\frac{1}{2\pi}\int_{-\pi}^{\pi}f(x,y,t)\,dx,
  \qquad
  f_{\neq}(x,y,t):=f(x,y,t)-f_0(y,t).
\end{equation}
Noting that the zero Fourier mode in $x$ of the Poisson equation
$-\Delta c=n$ on $\Omega=\mathbb T\times\mathbb R$ satisfies
$-\partial_y^2c_0=n_0$, we impose the normalization
\begin{equation}\label{eq:poisson-normalization-local}
  \partial_y c_0(y,t)
  =\frac12\int_{\mathbb R} n_0(z,t)\,dz
  -\int_{-\infty}^y n_0(z,t)\,dz,
\end{equation}
so that the term $\nabla c$ in \eqref{eq:KS-local}$_1$ is well-defined.
For convenience we also use the following notation: 
\begin{equation}\label{eq:X-definition}
 X=L^1(\Omega)\cap L^p(\Omega),\qquad
 \Vert{f}\Vert_{X}=\Vert{f}\Vert_{L^1(\Omega)}+\Vert{f}\Vert_{L^p(\Omega)},
\end{equation}
where an appropriate $p>2$ will be chosen later.

With the normalization and notations shown above, we now define a local mild solution to \eqref{eq:KS-local}, in the definition below. 

\begin{definition}[Local mild solution]\label{def:local-mild}
Let $\tau>0$ be an ${\{\mathcal{F}_t\}_{t\geq0}}$-stopping time and $n$ be an ${ \{\mathcal{F}_t\}_{t\geq0}}$-adapted process
with paths in $C([0,\tau);X)$.
The pair $(n,\tau)$ is called a local mild solution to
\eqref{eq:KS-local} if there exist stopping times $\rho_j\uparrow\tau$,
with $\rho_j<\tau$ on $\{\tau<\infty\}$, such that
{the following equation
\begin{equation}\label{eq:untruncated-mild-local}
 n(\cdot,t)=S_\nu(t,0)n_{\rm in}
 -\nu\int_0^tS_\nu(t,s)\nabla\cdot(n(\cdot,s)\nabla c(\cdot,s))\,\,d s,
 \qquad-\Delta c(\cdot,s)=n(\cdot,s).
\end{equation}
holds almost surely} on $[0,\rho_j]$ and,
for every $T<\infty$ and $q\ge2$, we have
\[
 \mathbb E\sup_{0\le t\le T}\Vert{n(t\wedge\rho_j)}\Vert_{X}^q<\infty,
\]
where $X$ is given in \eqref{eq:X-definition}.
\end{definition}
Starting from the local mild solution defined up to the stopping time $\tau$, we extend the solution whenever the corresponding continuation criterion remains satisfied. We define a maximal local mild solution as follows:
\begin{definition}[Maximal local mild solution]
\label{def:maximal-local}
A local mild solution $(n,\tau)$ to \eqref{eq:KS-local} given in Definition \ref{def:local-mild} is called a maximal local mild solution if
\begin{equation}\label{eq:X-blowup-definition}
 \limsup_{t\uparrow\tau}\Vert {n}\Vert_{X_t}
 =
 \infty
 \qquad
 \text{on }\{\tau<\infty\},
 \quad \mathbb P\text{-almost surely},
\end{equation}
where $X_t:=C([0,t];X)$ and $
\Vert{u}\Vert_{X_t}
:=
\sup_{0\le s\le t}\Vert{u(s)}\Vert_{X}$ with $X$ shown in \eqref{eq:X-definition}.
\end{definition}
We first establish the existence of a maximal local mild solution to \eqref{eq:KS-local} in the sense of Definition~\ref{def:maximal-local}. More precisely, we have the following result:
\begin{theorem}[Local existence]
\label{thm:local-complete}
Assume that $n_{\rm in}$ is deterministic and satisfies
\begin{equation}\label{eq:initial-local}
 n_{\rm in}\ge0,~
 n_{\rm in}\in H^2(\Omega)\cap L^\infty(\Omega)\cap L^1_2(\Omega),~ L^1_2(\Omega)=\{f:(1+y^2)f\in L^1(\Omega)\}.
\end{equation}
For any $0<\nu\le1$, system \eqref{eq:KS-local} has a unique maximal local mild solution
$(n,\tau_*)$ in the sense of Definition \ref{def:maximal-local}. The solution $n$ is nonnegative and for $t<\tau_*$, it satisfies $\Vert{n}\Vert_{L^1(\Omega)}=\Vert{n_{\rm in}}\Vert_{L^1(\Omega)}=:M.$
Moreover, for {almost every $\omega\in \mathfrak{D}$} and for every $T<\tau_*$, we have
\begin{equation}\label{eq:regularity-conclusion}
 n\in C([0,T];H^1(\Omega)\cap L^1(\Omega))\cap L^2(0,T;H^2(\Omega)),\qquad
 \sup_{0\le t\le T}\Vert{n(t)}\Vert_{L^\infty(\Omega)}<\infty.
\end{equation}
Furthermore,
\begin{equation}\label{eq:Linfty-blowup-local}
 \{\tau_*<\infty\}\subseteq
 \{\limsup_{t\uparrow\tau_*}\Vert{n(t)}\Vert_{L^\infty(\Omega)}=\infty\}.
\end{equation}
\end{theorem}
{ 
We remark that, under an appropriate rescaling, the systems \eqref{eq1} and \eqref{eq:KS-local} are equivalent and thus Theorem \ref{thm:local-complete} also provides us with a mild solution to \eqref{eq:KS-local}. To observe this fact, for any fixed integer $A$, define
\[
\nu_A:=A^{-1},\qquad
\widetilde W_s^{(A)}:=\sqrt A\,W_{s/A},
\qquad
\widetilde{\mathcal F}_s^{(A)}:=\mathcal F_{s/A}.
\]
Then, we recall that $\widetilde W^{(A)}$ is a standard Brownian motion
with respect to
$(\widetilde{\mathcal F}_s^{(A)})_{s\ge0}$.
Now, assume for some fixed $\nu=\frac1A$ that $(\tilde{n},\tilde{\tau}_*)$ satisfies \eqref{eq:KS-local} with respect to the stochastic basis $(\mathfrak{D},\mathcal{F},(\widetilde{\mathcal{F}}^{(A)}_t)_{t\geq 0},\tilde W^{(A)})$. Then, $(n^A,\tau_*^{A})$ defined as follows,
\[
n^A(t,x,y)
:=
\widetilde n(At,x,y),
\qquad
c^A(t,x,y)
:=
\widetilde c(At,x,y),
\qquad \tau_*^{A} =\frac{\widetilde\tau_*}{A},
\]
is a maximal local mild solution to \eqref{eq1}.
}

Our next result concerns global well-posedness of \eqref{eq:KS-local} and thus of \eqref{eq1} (recall that we notate $A=\frac1\nu$). More precisely, in our main result stated below, we prove the existence of a random, almost surely finite threshold \(A^*\) such that, on every path where \(A>A^*(\omega)\), the local mild solution to \eqref{eq:KS-local} exists globally in time and that the probability of finite-time blow decreases exponentially with increasing $A$.  
\begin{theorem}
\label{thm:global-original}

Assume that the conditions stated in Theorem \ref{thm:local-complete} hold and that, for any $A>0$, $(n^A,\tau^A_*)$ is the maximal local mild solution to \eqref{eq1} corresponding to the solution obtained in Theorem \ref{thm:local-complete}. Then, we have, that 
\begin{equation*}
\mathbb P\left( \{\omega\in\mathfrak{D}:
\exists\,A_*(\omega)<\infty
\ \text{such that}\
\tau_*^A(\omega)=\infty
\ \text{for every integer }A\ge A_*(\omega) \}
\right)=1.
\end{equation*}
Moreover, 
for some constants $C,c>0$ and $\beta\in(0,\frac14]$ we have
$$\mathbb P\bigl(\tau_*^A=\infty\bigr)\ge 1-Ce^{-cA^\beta}.$$
\end{theorem}
Theorem~\ref{thm:global-original} shows that sufficiently strong
stochastic shear suppresses blow-up with high probability in \eqref{eq1} and, yields global solutions for all integer amplitudes
above a finite random threshold.  
We argue that this result is strict in the sense that randomness of this threshold
cannot, in general, be removed.
In 
Appendix~\ref{app:blowup} we indeed demonstrate this by constructing a single
deterministic, nonnegative, smooth, compactly supported supercritical initial datum
of mass $16\pi$, for which finite-time blow-up
occurs with positive probability for every finite deterministic
shear amplitude $A\ge1$. Consequently, even for this fixed initial
datum, there is no finite deterministic threshold for shear amplitude above which almost-sure global
well-posedness is guaranteed. Our methods rely on an enhanced dissipation mechanism. We remark that the recent work of Coti Zelati, Hairer and Villringer \cite{CotiZelatiHairerVillringer2025} discusses enhanced dissipation of the passive scalar equation subject to stochastic shear flow.  Our contribution is to identify an analogous mechanism for the suppression of finite-time blow-up in the Keller--Segel system, which may be viewed as a nonlinear analogue of the passive scalar transport equation. Here, the mixing induced by the stochastic shear suppresses the nonzero $x$-modes, while the zero mode is effectively one-dimensional and can be controlled by means of one-dimensional Nash inequalities. The interplay of these two mechanisms prevents two-dimensional chemotactic concentration from developing into finite-time blow-up.  

 {  {\bf Organization of the paper:} We start our analysis by collecting all the necessary preliminary results in Section \ref{sec:prelims} including estimates for the random solution operator $S_\nu$ in Section \ref{sec:Sestimates}. In Section \ref{sec:existence}, we prove Theorem \ref{thm:local-complete}. The proof is divided into three steps: In Section \ref{sec:approx}, we construct a sequence of global approximate solutions using a cutoff argument; In Section \ref{sec:local}, we prove, by restricting the aforementioned solution to suitable stopping-times, the existence of a unique local mild solution and then extend it to a maximal local mild solution of \eqref{eq:KS-local}. In Section \ref{sec:global}, we prove Theorem \ref{thm:global-original}. We first establish the existence of an enhanced dissipation mechanism for the linearized nonzero Fourier modes in Section \ref{sec:linear}, and then estimate the nonlinear counterpart in Section \ref{sec:nonlinear}. Finally, we combine these estimates and apply a bootstrap argument in Section \ref{sec:bootstrap} to obtain the desired global result.}
 
\section{Preliminary results}\label{sec:prelims}
In this section, we give some preliminary results including the a-priori estimates of Poisson equations in $\Omega=\mathbb T\times \mathbb R$ and the random solution operator $S_{\nu}(t,s)$.  
We remark that throughout the paper, \(C\) denotes a generic positive constant whose value may change from line to line.

We begin by introducing some notation. We define the Fourier transform
in the $x$ variable and its inverse by
\begin{align}\label{Fourier_mode_x_f}
f_k(y,t)
:=
\frac{1}{2\pi}
\int_{-\pi}^{\pi} f(x,y,t)e^{-ikx}\,dx,
\qquad
\check g(x,y,t)
:=
\sum_{k\in\mathbb Z} g_k(y,t)e^{ikx}.
\end{align}
We further define the unitary Fourier transform in the $y$ variable by
\begin{equation}\label{Fourier_transform_y}
\widehat f_k(\eta,t)
:=
\frac{1}{\sqrt{2\pi}}
\int_{\mathbb R} f_k(y,t)e^{-iy\eta}\,dy.
\end{equation}
For $\Omega=\mathbb T\times\mathbb R$, Parseval's identity
and the unitarity of the Fourier transform in $y$ yield
\begin{equation}\label{eq:cylinder-Plancherel}
\|f(\cdot,\cdot,t)\|_{L^2(\Omega)}^2
=
2\pi\sum_{k\in\mathbb Z}
\int_{\mathbb R}|\widehat f_k(\eta,t)|^2\,d\eta.
\end{equation}
We denote by $P_0$ and $P_{\neq}$ the projections onto the zero
and nonzero Fourier modes in the $x$-variable, respectively:
\begin{align}\label{nonzero-mode-projection}
P_0f(y):=\frac{1}{2\pi}\int_{\mathbb T}f(x,y)\,\mathrm dx,
\qquad
P_{\neq}f(x,y):=f(x,y)-P_0f(y).
\end{align}
For brevity, we write $f_0:=P_0f$ and $f_{\neq}:=P_{\neq}f$ (see \eqref{f0-def}).

We next collect several estimates for the Poisson equation in the following lemma; a consequence of which is that any initial datum $n_{\rm in}$ satisfying \eqref{eq:initial-local} further satisfies 
$$n_{\rm in}\in X.$$
\begin{lemma}[Poisson estimates]\label{lem:X-elliptic} For the solution to the Poisson equation $-\Delta\phi=f$  satisfying the normalization
\eqref{eq:poisson-normalization-local}, there exists a constant
$C_p>0$ with $p>2$ such that, for all $f,g\in X$ given by \eqref{eq:X-definition}, we have
\begin{align}
\Vert {\nabla(-\Delta)^{-1}f}\Vert_{L^\infty(\Omega)}
 &\le C_p\Vert{f}\Vert_{X},\label{eq:X-Poisson}\\
\Vert {f\nabla(-\Delta)^{-1}f}\Vert_{X}
 &\le C_p\Vert{f}\Vert_{X}^2,\label{eq:X-quadratic}\\
 \Vert {f\nabla(-\Delta)^{-1}f-g\nabla(-\Delta)^{-1}g}\Vert_X
 &\le C_p(\Vert {f}\Vert_X+\Vert {g}\Vert_{X})\Vert {f-g}\Vert_X.
 \label{eq:X-difference}
\end{align}
 Moreover, if $f_0=0$ and $-\Delta\phi=f$ with $
\phi_0=0$, then
\begin{equation}\label{eq:elliptic-L2}
\|\nabla\phi\|_{L^2(\Omega)}+\|\nabla^2\phi\|_{L^2(\Omega)}
       \le C\|f\|_{L^2(\Omega)},
~~
 \|\partial_y\phi\|_{L^\infty_yL^2_x}\le C\|f\|_{L^2(\Omega)}.
\end{equation}
\end{lemma}
\begin{proof}
The estimates follow from the Green kernel representation on
$\mathbb T\times\mathbb R$; see \cite[Section~2.1]{CastroLear2023}.
For completeness, we give a proof here.  As shown in \eqref{eq:poisson-normalization-local}, we have
\begin{align}\label{phi0modeprime}
\phi_0'(y)
=
\frac{1}{2}\int_{\mathbb R} f_0(z)\,\mathrm{d}z
-
\int_{-\infty}^{y} f_0(z)\,\mathrm{d}z,
\end{align}
where $\phi_0$ is the zero-mode defined in \eqref{f0-def}.  In addition, we rewrite $\phi=(-\Delta)^{-1}f$ and  $\phi(z)=\int_{\Omega}G(z-\zeta)f(\zeta)\,\,d\zeta,$
where \(G\) is the normalized Green function on
\(\Omega=\mathbb T\times\mathbb R\). Hence,
\[
\nabla\phi(z)
=
\int_{\Omega}\nabla G(z-\zeta)f(\zeta)\,\,d\zeta.
\]
For fixed \(z\in\Omega\), we decompose $\nabla\phi(z)=I_{\mathrm{near}}(z)+I_{\mathrm{far}}(z),$
where
\begin{align*}
I_{\mathrm{near}}(z)
=
\int_{\{d_\Omega(z,\zeta)\le1\}}
\nabla G(z-\zeta)f(\zeta)\,\,d\zeta,\quad
I_{\mathrm{far}}(z)
=
\int_{\{d_\Omega(z,\zeta)>1\}}
\nabla G(z-\zeta)f(\zeta)\,\,d\zeta.
\end{align*}
For the near-field term, the local singularity of the cylinder Green
function satisfies
\[
|\nabla G(z-\zeta)|
\le
C\,d_\Omega(z,\zeta)^{-1}
\qquad\text{when }d_\Omega(z,\zeta)\le1.
\]
  Let $p'=p/(p-1)$ be the H\"older conjugate of $p$.
Since $p>2$, we have $1<p'<2$. H\"older's inequality then yields
\begin{align}
|I_{\mathrm{near}}(z)|
&\le
C\left(\frac{2\pi}{2-p'}\right)^{1/p'}
\left(
\int_{\{d_\Omega(z,\zeta)\le1\}}
|f(\zeta)|^p\,\,d\zeta
\right)^{1/p}\le
C_p\Vert{f}\Vert_{L^p(\Omega)}.
\label{eq:near-field-Poisson}
\end{align}
For the far-field term, the normalized cylinder Green function satisfies
\[
\sup_{\{d_\Omega(z,\zeta)>1\}}
|\nabla G(z-\zeta)|\le C.
\]
It then follows that
\begin{align}
|I_{\mathrm{far}}(z)|
&\le
\int_{\{d_\Omega(z,\zeta)>1\}}
|\nabla G(z-\zeta)|\,|f(\zeta)|\,d\zeta \le
C\int_{\{d_\Omega(z,\zeta)>1\}}
|f(\zeta)|\,d\zeta\le
C\Vert {f}\Vert_{L^1(\Omega)}.
\label{eq:far-field-Poisson}
\end{align}
Combining \eqref{eq:near-field-Poisson} and
\eqref{eq:far-field-Poisson}, we obtain that
\[
\Vert {\nabla(-\Delta)^{-1}f}\Vert_{L^\infty(\Omega)}
\le
C_p\left(
\Vert {f}\Vert_{L^p(\Omega)}
+
\Vert {f}\Vert_{L^1(\Omega)}
\right)
=
C_p\Vert{f}\Vert_{X},
\]
which proves \eqref{eq:X-Poisson}. For $r=1,p$, thanks to H\"older's inequality, we obtain that
\[
\Vert{f\nabla(-\Delta)^{-1}f}\Vert_{L^r(\Omega)}\le
\Vert {f}\Vert_{L^r(\Omega)}
\Vert {\nabla(-\Delta)^{-1}f}\Vert_{L^\infty(\Omega)}
 \le C_p\Vert {f}\Vert_{L^r(\Omega)}\Vert {f}\Vert_{X},
\]
which proves \eqref{eq:X-quadratic}.   \\
Using the identity
\[
\begin{aligned}
&f\nabla(-\Delta)^{-1}f
-g\nabla(-\Delta)^{-1}g
=(f-g)\nabla(-\Delta)^{-1}f
+g\nabla(-\Delta)^{-1}(f-g),
\end{aligned}
\]
we obtain, for \(r=1,p\), that
\[
\begin{aligned}
&\Vert {
f\nabla(-\Delta)^{-1}f
-g\nabla(-\Delta)^{-1}g
}\Vert_{L^r(\Omega)}\le
\Vert {f-g}\Vert_{L^r(\Omega)}
\Vert {\nabla(-\Delta)^{-1}f}\Vert_{L^\infty(\Omega)}
+
\Vert {g}\Vert_{L^r(\Omega)}
\Vert{\nabla(-\Delta)^{-1}(f-g)}\Vert_{L^\infty(\Omega)}.
\end{aligned}
\]
Applying \eqref{eq:X-Poisson} first to \(f\) and then to \(f-g\)
gives us the following inequality:
\[
\begin{aligned}
&\Vert {
f\nabla(-\Delta)^{-1}f
-g\nabla(-\Delta)^{-1}g
}\Vert_{L^r(\Omega)}\le
C_p\Vert{f-g}\Vert_{L^r(\Omega)} \Vert{f}\Vert_{X}
+
C_p\Vert {g}\Vert_{L^r(\Omega)}\Vert {f-g}\Vert_{X}.
\end{aligned}
\]
Summing this inequality for \(r=1\) and \(r=p\), we write
\[
\begin{aligned}
\Vert {
f\nabla(-\Delta)^{-1}f
-g\nabla(-\Delta)^{-1}g
}\Vert_{X}&\le
C_p
\bigl(
\Vert {f-g}\Vert_{L^1(\Omega)}+\Vert {f-g}\Vert_{L^p(\Omega)}
\bigr)\Vert {f}\Vert_{X}
+
C_p
\bigl(
\Vert {g}\Vert_{L^1(\Omega)}+\Vert {g}\Vert_{L^p(\Omega)}
\bigr)\Vert {f-g}\Vert_{X}\\
&\quad=
C_p\Vert {f-g}\Vert_{X}\Vert {f}\Vert_{X}
+
C_p\Vert {g}\Vert_{X}\Vert {f-g}\Vert_{X}\\
&\quad=
C_p\bigl(\Vert {f}\Vert_{X}+\Vert {g}\Vert_{X}\bigr)
\Vert {f-g}\Vert_{X}.
\end{aligned}
\]
This proves \eqref{eq:X-difference}.  In addition, we remark that $\Vert {n_{\rm in}}\Vert_{L^p(\Omega)}^p\le
\Vert {n_{\rm in}}\Vert_{L^\infty(\Omega)}^{p-1}\Vert {n_{\rm in}}\Vert_{L^1(\Omega)}$, which implies $n_{\text{in}}\in  X.$

  For every \(k\ne0\), by using (\ref{Fourier_mode_x_f}), we have the \(k\)-th Fourier mode in $x$ of
\(-\Delta \phi =f\) satisfies $\left(k^2-\partial_y^2\right)\phi_k=f_k.$
Taking the Fourier transform in \(y\), we obtain from (\ref{Fourier_transform_y}) that $(k^2+\eta^2)\hat\phi_k(\eta)
=
\widehat f_k(\eta),$
and hence
\begin{equation}\label{eq:psi-Fourier-multiplier}
\widehat\phi_k(\eta)
=
\frac{\widehat f_k(\eta)}{k^2+\eta^2},
\qquad k\ne0.
\end{equation}
Since \(k\ne0\), we have \(k^2+\eta^2\ge1\). Using
\eqref{eq:cylinder-Plancherel} and
\eqref{eq:psi-Fourier-multiplier}, we further find
\begin{align*}
\|\nabla\phi\|_{L^2(\Omega)}^2
&=
2\pi\sum_{k\ne0}\int_{\mathbb R}
(k^2+\eta^2)|\widehat\phi_k(\eta)|^2\,d\eta
=
2\pi\sum_{k\ne0}\int_{\mathbb R}
\frac{|\widehat f_k(\eta)|^2}{k^2+\eta^2}\,d\eta
\le
2\pi\sum_{k\ne0}\int_{\mathbb R}
|\widehat f_k(\eta)|^2\,d\eta
=
\|f\|_{L^2(\Omega)}^2.
\end{align*}

Similarly,
\begin{align*}
\|\phi_{xx}\|_{L^2(\Omega)}^2
+2\|\phi_{xy}\|_{L^2(\Omega)}^2
+\|\phi_{yy}\|_{L^2(\Omega)}^2
&=
2\pi\sum_{k\ne0}\int_{\mathbb R}
(k^4+2k^2\eta^2+\eta^4)
|\widehat\phi_k(\eta)|^2\,d\eta
\\
&=
2\pi\sum_{k\ne0}\int_{\mathbb R}
(k^2+\eta^2)^2
|\widehat\phi_k(\eta)|^2\,d\eta
\\
&=
2\pi\sum_{k\ne0}\int_{\mathbb R}
|\widehat f_k(\eta)|^2\,d\eta=
\|f\|_{L^2(\Omega)}^2.
\end{align*}
Consequently,
\begin{equation}\label{eq:elliptic-L2-detailed}
\|\nabla\phi\|_{L^2(\Omega)}
+
\|\nabla^2\phi\|_{L^2(\Omega)}
\le
C\|f\|_{L^2(\Omega)}.
\end{equation}

To prove the second estimate in \eqref{eq:elliptic-L2}, set $F(y):=\partial_y\phi(\cdot,y)\in L^2(\mathbb T).$
By using \eqref{eq:elliptic-L2-detailed}, we have
$F\in H^1\bigl(\mathbb R;L^2(\mathbb T)\bigr)$ and $F'(y)=\partial_y^2\phi(\cdot,y).$
The one-dimensional Gagliardo--Nirenberg inequality  gives
\[
\|F\|_{L^\infty(\mathbb R;L^2(\mathbb T))}^2
\le
2
\|F\|_{L^2(\mathbb R;L^2(\mathbb T))}
\|F'\|_{L^2(\mathbb R;L^2(\mathbb T))}.
\]
Consequently,
\[
\|\partial_y\phi\|_{L^\infty_yL^2_x}^2
\le
2
\|\partial_y\phi\|_{L^2(\Omega)}
\|\partial_y^2\phi\|_{L^2(\Omega)}
\le
C\|f\|_{L^2(\Omega)}^2,
\]
where the last inequality follows from
\eqref{eq:elliptic-L2-detailed}. Taking square roots yields $\|\partial_y\phi\|_{L^\infty_yL^2_x}
\le
C\|f\|_{L^2(\Omega)},$
which finishes the proof of \eqref{eq:elliptic-L2}.
\end{proof}
We next collect several important Gagliardo--Nirenberg inequalities in the following lemma:
 \begin{lemma} \label{lem:interpolation}
There exists a constant $C>0$, such that for any $f\in H^1(\Omega)\cap L^1(\Omega),$ we have
\begin{align}
 \|f\|_{L^4(\Omega)}^2
 &\le C\|f\|_{L^2(\Omega)}
                   \bigl(\|\nabla f\|_{L^2(\Omega)}+\|f\|_{L^2(\Omega)}\bigr),
                   \label{eq:Ladyzhenskaya}\\
 \|f\|_{L^2(\Omega)}^2
 &\le C\|f\|_{L^1(\Omega)}\|\nabla f\|_{L^2(\Omega)}
                  +C\|f\|_{L^1(\Omega)}^2.\label{eq:inhom-Nash}
\end{align}
and for any $a\in L^1(\mathbb R)\cap W^{1,2}(\mathbb R)$, we have 
\begin{equation}\label{eq:one-D-Nash}
 \|a\|_{L^2(\mathbb R)}^6
 \le C\|a\|_{L^1(\mathbb R)}^4\|a'\|_{L^2(\mathbb R)}^2.
\end{equation}
\end{lemma}
\begin{proof}
The inequality \eqref{eq:Ladyzhenskaya} is a consequence of 
 (A-5) in \cite{TerraciniTzvetkovVisciglia2014}
(obtained by taking $n=k=1$ and $M^1=\mathbb T$) and,
 \eqref{eq:inhom-Nash} follows from
Lemma~3.1 in \cite{GallaySlijepcevic2015} and Young's inequality.
Finally, \eqref{eq:one-D-Nash} is obtained by taking $n=1$ in inequality
(3.2) in \cite{GallaySlijepcevic2015}.
\end{proof}
\subsection{A-priori estimates of  random solution operators in passive scalar equations}\label{sec:Sestimates}
The goal of this section is to establish suitable estimates for the random solution operator \(S_{\nu}(t,s)\) for \eqref{eq:passive-local} that we require in our proof of local existence. We begin with the following continuity estimates.
\begingroup
\begin{proposition}\label{prop:pathwise-kernel}
For any $f,F\in L^r(\Omega)$, $0\le s<t$ and $r\in[1,\infty]$,  we have for some constant $C>0$ that 
\begin{align}
\Vert {S_\nu(t,s)f}\Vert_{L^r(\Omega)}&\le\Vert {f}\Vert_{L^r(\Omega)},\label{eq:Lr-contraction}\\
\Vert {S_\nu(t,s)\nabla\cdot F}\Vert_{L^r(\Omega)}
 &\le\frac{C\Gamma_{t,s}}{\sqrt{\nu(t-s)}}\Vert {F}\Vert_{L^r(\Omega)},
 \label{eq:pathwise-divergence-bound}\\
  \|S_\nu(t,s)\nabla\cdot F\|_{L^\infty(\Omega)}
 &\le C\Lambda_T[\nu(t-s)]^{-\frac12-\frac1p}
           \|F\|_{L^p(\Omega)},\label{eq:pathwise-Lp-Linfty-divergence}
\end{align}
where, 
\begin{equation}\label{eq:Gamma-definition}
 \Gamma_{t,s}
 =\left(2+\frac1{t-s}\int_s^t|W_a-W_s|^2\,\,d a\right)^{1/2},  \quad\Lambda_T:=2(1+\sup_{0\le r\le T}|W_r|^2).
\end{equation}
Moreover,  for $r<\infty$ and $f\in L^r(\Omega),$ 
\begin{align}\label{strong_continuity_semigroup}
\lim_{t\downarrow s}
\Vert {S_\nu(t,s)f-f}\Vert_{L^r(\Omega)} =0.
\end{align} 
\end{proposition}
\begin{proof}
Fix $s\geq 0$ and define 
\begin{equation}\label{eq:shear-matrix-definitions}
 B_a^s=W_a-W_s,\quad
 A_a^s=\begin{pmatrix}1+(B_a^s)^2&-B_a^s\\-B_a^s&1\end{pmatrix},
 \quad Q_{t,s}=\int_s^tA_a^s\,d a.
\end{equation}
Define $h(a,x,y)
=
f(a,x+yB_a^s,y),$
where $f$ satisfies $d f
=
-y\partial_xf\circ d W_a
+\nu\Delta f\,d a.$
In addition, since \(B_s^s=0\), we have $h(s,x,y)=f(s,x,y).$ 
Since $d B_a^s=d W_a,$
the Stratonovich chain rule gives
 \(h\) satisfies the pathwise parabolic equation
\begin{align}\label{transformed_eq}
\partial_ah
=
\nu\nabla\cdot(A_a^s\nabla h),
\qquad
h(s)=f(s).
\end{align}

We next derive the fundamental solution of (\ref{transformed_eq}) in $\mathbb R^2$.
 Taking the Fourier transform, we obtain
\[
\partial_a\widehat h(a,\xi)
=
-\nu\,\xi^TA_a^s\xi\,\widehat h(a,\xi),
\qquad
\xi\in\mathbb R^2.
\]
For each fixed \(\xi\),
integrating it from \(s\) to \(t\), we obtain
\begin{align*}
\widehat h(t,\xi)
&=
\exp\left(
-\nu\int_s^t\xi^TA_a^s\xi\,d a
\right)\widehat h(s,\xi)=
\exp\left(
-\nu\xi^T
\left(\int_s^tA_a^s\,d a\right)\xi
\right)\widehat f(\xi)=
\exp\left(
-\nu\xi^TQ_{t,s}\xi
\right)\widehat f(\xi).
\end{align*}
Thus, on \(\mathbb R^2\), the Fourier transform of the fundamental
solution is $\widehat K_{t,s}^{\mathbb R^2}(\xi)
=
\exp\left(-\nu\xi^TQ_{t,s}\xi\right).$

Next, note that \(Q_{t,s}\) is positive definite.  Indeed, for every
\(v=(v_1,v_2)^T\in\mathbb R^2\), we have
\begin{align*}
v^TA_a^sv
&=
\bigl(1+(B_a^s)^2\bigr)v_1^2
-2B_a^sv_1v_2+v_2^2=
v_1^2+\bigl(v_2-B_a^sv_1\bigr)^2\geq 0.
\end{align*}
The inverse transform then implies that
\begin{equation}\label{eq:whole-space-Gaussian}
K_{t,s}^{\mathbb R^2}(x,y)
=
\frac{1}{4\pi\nu\sqrt{\det Q_{t,s}}}
\exp\left[
-\frac{(x,y)^TQ_{t,s}^{-1}(x,y)}{4\nu}
\right],
\end{equation}
where the right-hand side is the centered Gaussian density with
covariance matrix $2\nu Q_{t,s},$ since $\frac{1}{2\pi\sqrt{\det(2\nu Q_{t,s})}}
=
\frac{1}{4\pi\nu\sqrt{\det Q_{t,s}}}$
and $-\frac12(x,y)^T(2\nu Q_{t,s})^{-1}(x,y)
=
-\frac{(x,y)^TQ_{t,s}^{-1}(x,y)}{4\nu}.$
Recall that the kernel is
\(2\pi\)-periodic in the \(x\)-variable. Hence, by periodizing
\eqref{eq:whole-space-Gaussian} in \(x\), we define the fundamental solution of (\ref{transformed_eq}) in $\Omega$ as follows
\begin{equation}\label{eq:periodized-Gaussian}
K_{t,s}(x,y)
=
\sum_{m\in\mathbb Z}
\frac{
\exp\left[
-\dfrac{((x+2\pi m),y)^T
Q_{t,s}^{-1}
((x+2\pi m),y)}
{4\nu}
\right]
}
{4\pi\nu\sqrt{\det Q_{t,s}}}.
\end{equation}
Indeed,
\begin{align*}
K_{t,s}(x+2\pi,y)
&=
\sum_{m\in\mathbb Z}
K_{t,s}^{\mathbb R^2}(x+2\pi+2\pi m,y)=
\sum_{m\in\mathbb Z}
K_{t,s}^{\mathbb R^2}(x+2\pi(m+1),y)=
K_{t,s}(x,y),
\end{align*}
which implies that $K_{t,s}$ is periodic in $x.$  Moreover, note that \(K_{t,s}\) is nonnegative and its total mass is $1$. Indeed,
 Tonelli's theorem and the change of variables
\(x'=x+2\pi m\), yield
\begin{align*}
\int_{\mathbb T\times\mathbb R}
K_{t,s}(x,y)\,d xd y
&=
\sum_{m\in\mathbb Z}
\int_0^{2\pi}\int_{\mathbb R}
K_{t,s}^{\mathbb R^2}(x+2\pi m,y)
\,d yd x=
\int_{\mathbb R^2}
K_{t,s}^{\mathbb R^2}(x,y)\,d xd y=1.
\end{align*}
Consequently, the solution of (\ref{transformed_eq}) in $\Omega$ is given by
\begin{equation}\label{eq:transformed-kernel-representation}
h(t,x,y)
=
(K_{t,s}*f)(x,y).
\end{equation}
Recall that $h(t,x,y)=f(t,x+yB_t^s,y).$
By replacing \(x\) by \(x-yB_t^s\) in
\eqref{eq:transformed-kernel-representation}, we obtain the original random solution operator representation $S_\nu(t,s)f(x,y)
=
(K_{t,s}*f)(x-yB_t^s,y).$
The determinant of the Jacobian of the map $(x,y)\mapsto(x-yB_t^s,y)$
is one. Therefore, it preserves every
\(L^r(\Omega)\) norm. By Young's convolution inequality and the identity
\(\Vert {K_{t,s}}\Vert_{L^1(\Omega)}=1\), we conclude, for
\(1\le r\le\infty\), that
\begin{align*}
\Vert {S_\nu(t,s)f}\Vert_{L^r(\Omega)}
&=
\Vert {K_{t,s}*f}\Vert_{L^r(\Omega)}\le
\Vert {K_{t,s}}\Vert_{L^1(\Omega)}\Vert {f}\Vert_{L^r(\Omega)}=
\Vert {f}\Vert_{L^r(\Omega)}.
\end{align*}
This proves \eqref{eq:Lr-contraction}.

Next, for
$\delta:=t-s$, $I_1:=\int_s^tB_a^s\,d a$, and
$I_2:=\int_s^t(B_a^s)^2\,d a$ we write,
\begin{equation}\label{eq:Q-explicit}
 Q_{t,s}=\begin{pmatrix}\delta+I_2&-I_1\\-I_1&\delta\end{pmatrix},
 \quad
 \operatorname{tr}Q_{t,s}=2\delta+I_2,
 \quad
 \det Q_{t,s}=\delta^2+\delta I_2-I_1^2.
\end{equation}
Thanks to the Cauchy--Schwarz inequality we have $I_1^2\le\delta I_2$. This implies that
\begin{equation}\label{eq:Q-inverse-covariance}
 \det Q_{t,s}\ge\delta^2,\qquad
 \lambda_{\min}(Q_{t,s})^{-1}
 \le\frac{\operatorname{tr}Q_{t,s}}{\det Q_{t,s}}
 \le\frac1\delta\left(2+\frac{I_2}{\delta}\right).
\end{equation}
Next, we estimate the gradient of the whole-space kernel.  
Since \(Q_{t,s}^{-1}\) is symmetric, $\nabla_z(z^TQ_{t,s}^{-1}z)=2Q_{t,s}^{-1}z.$
It follows from \eqref{eq:whole-space-Gaussian} that $\nabla K_{t,s}^{\mathbb R^2}(z)
=
-\frac{1}{2\nu}Q_{t,s}^{-1}z
K_{t,s}^{\mathbb R^2}(z).$
Consequently,
\[
\Vert{\nabla K_{t,s}^{\mathbb R^2}}\Vert _{L^1(\mathbb R^2)}
=
\frac{1}{2\nu}
\int_{\mathbb R^2}
|Q_{t,s}^{-1}z|K_{t,s}^{\mathbb R^2}(z)\,d z.
\]
We will next apply the change of variables $z=2\sqrt\nu\,Q_{t,s}^{1/2}w,$ so that we have $z^TQ_{t,s}^{-1}z=4\nu|w|^2$
and $d z
=
4\nu\sqrt{\det Q_{t,s}}\,d w.$
Hence, $K_{t,s}^{\mathbb R^2}(z)\,d z
=
\frac1\pi e^{-|w|^2}\,d w.$
Moreover, $Q_{t,s}^{-1}z
=
2\sqrt\nu\,Q_{t,s}^{-1/2}w.$
Therefore,
\begin{align*}
\Vert {\nabla K_{t,s}^{\mathbb R^2}}\Vert_{L^1(\mathbb R^2)}
&=
\frac{1}{\sqrt\nu}
\frac1\pi
\int_{\mathbb R^2}
|Q_{t,s}^{-1/2}w|e^{-|w|^2}\,d w\\
&\le
\frac{1}{\sqrt\nu}
\lambda_{\min}(Q_{t,s})^{-1/2}
\frac1\pi
\int_{\mathbb R^2}
|w|e^{-|w|^2}\,d w\\
&\le
C\nu^{-1/2}\lambda_{\min}(Q_{t,s})^{-1/2}.
\end{align*}

Since the Gaussian and its derivatives decay rapidly, the
periodized kernel can differentiated term by term $\nabla K_{t,s}(x,y)
=
\sum_{m\in\mathbb Z}
\nabla K_{t,s}^{\mathbb R^2}(x+2\pi m,y).$
Therefore, by the triangle inequality and Tonelli's theorem, we obtain
\begin{align*}
\Vert {\nabla K_{t,s}}\Vert_{L^1(\mathbb T\times\mathbb R)}
&=
\int_{\mathbb R}\int_0^{2\pi}
\left|
\sum_{m\in\mathbb Z}
\nabla K_{t,s}^{\mathbb R^2}(x+2\pi m,y)
\right|
\,d xd y\le
\sum_{m\in\mathbb Z}
\int_{\mathbb R}\int_0^{2\pi}
\left|
\nabla K_{t,s}^{\mathbb R^2}(x+2\pi m,y)
\right|
\,d xd y.
\end{align*}
For each \(m\), we apply the change of variables \(x'=x+2\pi m\).
Since the intervals $[2\pi m,2\pi(m+1))$,
$m\in\mathbb Z$
form a disjoint partition of \(\mathbb R\), we obtain
\begin{align*}
\Vert {\nabla K_{t,s}}\Vert_{L^1(\mathbb T\times\mathbb R)}
&\le
\sum_{m\in\mathbb Z}
\int_{\mathbb R}
\int_{2\pi m}^{2\pi(m+1)}
\left|
\nabla K_{t,s}^{\mathbb R^2}(x',y)
\right|
\,d x'd y=
\Vert {
\nabla K_{t,s}^{\mathbb R^2}
}\Vert_{L^1(\mathbb R^2)}.
\end{align*}
Thus, periodization does not increase the \(L^1\) norm of the gradient. Then, by
\eqref{eq:Q-inverse-covariance}, we have
\begin{equation}\label{eq:gradient-kernel-final}
\Vert {\nabla K_{t,s}}\Vert_{L^1(\Omega)}
 \le\frac{C\Gamma_{t,s}}{\sqrt{\nu(t-s)}},
\end{equation}
where $\Gamma_{t,s}$ is given in (\ref{eq:Gamma-definition}).

We next prove the estimate
\eqref{eq:pathwise-divergence-bound}. Let $F=(F_1,F_2)\in L^r(\Omega;\mathbb R^2)$, $1\le r\le\infty.$  Recall that   $\Phi_a^s(x,y)
=
\bigl(x+yB_a^s,y\bigr)$ and $B_a^s=W_a-W_s$.
Since $B_s^s=0,$
we have that $\Phi_s^s(x,y)=(x,y)$ and
\[
D\Phi_s^s
=
\begin{pmatrix}
1&B_s^s\\
0&1
\end{pmatrix}
=
I.
\]
Consequently, an initial condition of the form $\nabla\cdot F
$
is unchanged under the coordinate transformation at time \(s\). Moreover,
\begin{align}
&S_\nu(t,s)\nabla\cdot F(x,y)=
\left[
K_{t,s}*
\bigl(\partial_xF_1+\partial_yF_2\bigr)
\right]
(x-yB_t^s,y).
\label{eq:divergence-kernel-before-ibp}
\end{align}
 For smooth
periodic vector fields $F$ that are compactly supported in \(y\), we next integrate by parts and obtain
\begin{align}\label{convlution_gradient_estimate}
&\left[
K_{t,s}*
\bigl(\partial_xF_1+\partial_yF_2\bigr)
\right](x,y)=
\int_{\mathbb T}\int_{\mathbb R}
K_{t,s}(x-x',y-y')
\left[
\partial_{x'}F_1(x',y')
+
\partial_{y'}F_2(x',y')
\right]
\,d y'd x'.
\end{align}
Since $\partial_{x'}K_{t,s}(x-x',y-y')
=
-\partial_xK_{t,s}(x-x',y-y'),$ for the first term in the right hand side of (\ref{convlution_gradient_estimate}), periodicity in \(x'\) yields
\begin{align*}
\int_{\mathbb T}\int_{\mathbb R}
K_{t,s}(x-x',y-y')
\partial_{x'}F_1(x',y')
\,d y'd x'
&=
-\int_{\mathbb T}\int_{\mathbb R}
\partial_{x'}
K_{t,s}(x-x',y-y')
F_1(x',y')
\,d y'd x'\\
&\quad=
\int_{\mathbb T}\int_{\mathbb R}
\partial_xK_{t,s}(x-x',y-y')
F_1(x',y')
\,d y'd x',
\end{align*}
Similarly,
\begin{align*}
&\int_{\mathbb T}\int_{\mathbb R}
K_{t,s}(x-x',y-y')
\partial_{y'}F_2(x',y')
\,d y'd x'=
\int_{\mathbb T}\int_{\mathbb R}
\partial_yK_{t,s}(x-x',y-y')
F_2(x',y')
\,d y'd x'.
\end{align*}
Therefore,
\begin{align}
K_{t,s}*(\nabla\cdot F)
&=
(\partial_xK_{t,s})*F_1
+
(\partial_yK_{t,s})*F_2=
\nabla K_{t,s}*F,
\label{eq:kernel-divergence-ibp}
\end{align}
For \(1\le r\le\infty\), Jacobian determinant of transform $(x,y)\mapsto(x-yB_t^s,y)$ is $1$  and therefore preserves every
\(L^r(\Omega)\)-norm.  Thus,
combining \eqref{eq:divergence-kernel-before-ibp} and
\eqref{eq:kernel-divergence-ibp}, we obtain that
\begin{align*}
\Vert {S_\nu(t,s)\nabla\cdot F}\Vert_{L^r(\Omega)}
&=
\Vert {\nabla K_{t,s}*F}\Vert_{L^r(\Omega)}.
\end{align*}
Young's convolution inequality and then \eqref{eq:gradient-kernel-final} gives us,
\begin{align*}
\Vert {S_\nu(t,s)\nabla\cdot F}\Vert_{L^r(\Omega)}
&\le
\Vert {\nabla K_{t,s}}\Vert_{L^1(\Omega)}
\Vert {F}\Vert_{L^r(\Omega)}\le
\frac{C\Gamma_{t,s}}{\sqrt{\nu(t-s)}}
\Vert {F}\Vert_{L^r(\Omega)}.
\end{align*}
This proves
\eqref{eq:pathwise-divergence-bound}.

Finally, we extend the preceding identity to an arbitrary vector field $F=(F_1,F_2)\in L^r(\Omega;\mathbb R^2).$ \(\nabla\cdot F\) is defined by
\[
\langle\nabla\cdot F,\varphi\rangle
=
-\int_\Omega F(z)\cdot\nabla\varphi(z)\,d z,
\qquad
\varphi\in C_c^\infty(\Omega),
\]
where we denote by \(C_c^\infty(\Omega)\) the set of smooth functions that are periodic
in \(x\) and compactly supported in \(y\).
  We first consider \(1\le r<\infty\). Choose a sequence $F^{(n)}\in C_c^\infty(\Omega;\mathbb R^2)$ of smooth vector fields, periodic in \(x\), such that $\Vert{F^{(n)}-F}\Vert_{L^r(\Omega)}\rightarrow0.$
Then, for every \(n\), integrating by parts gives us the following identity \begin{align}\label{smooth_vector_field_Fn}
    K_{t,s}*(\nabla\cdot F^{(n)})
=
\nabla K_{t,s}*F^{(n)}.
\end{align}
Moreover, Young's convolution inequality indicates that
\begin{align*}
\Vert {
\nabla K_{t,s}*F^{(n)}
-
\nabla K_{t,s}*F
}\Vert_{L^r(\Omega)}
&=
\Vert {
\nabla K_{t,s}*(F^{(n)}-F)
}\Vert_{L^r(\Omega)}\le
\Vert {\nabla K_{t,s}}\Vert_{L^1(\Omega)}
\Vert {F^{(n)}-F}\Vert_{L^r(\Omega)}\rightarrow0.
\end{align*}
Thus, $\nabla K_{t,s}*F^{(n)}
\rightarrow
\nabla K_{t,s}*F$ in $L^r(\Omega)$.  On the other hand, $\nabla\cdot F^{(n)}
\rightarrow
\nabla\cdot F$ in $\mathcal D'(\Omega)$.  Indeed, for every \(\varphi\in C_c^\infty(\Omega)\), we have
\begin{align*}
\left|
\left\langle
\nabla\cdot(F^{(n)}-F),\varphi
\right\rangle
\right|
&=
\left|
\int_\Omega
(F^{(n)}-F)\cdot\nabla\varphi\,d z
\right|\le
\Vert {F^{(n)}-F}\Vert_{L^r(\Omega)}
\Vert {\nabla\varphi}\Vert_{L^{r'}(\Omega)}
\rightarrow0,
\end{align*}
where \(r'=r/(r-1)\) with \(r'=\infty\) when \(r=1\).
Passing to the limit in the distributional identity (\ref{smooth_vector_field_Fn}), we obtain that $K_{t,s}*(\nabla\cdot F)
=
\nabla K_{t,s}*F
\quad\text{in }\mathcal D'(\Omega).$
Now, since \(\nabla K_{t,s}*F\in L^r(\Omega)\), we further deduce that,
\begin{equation*}
\Vert {K_{t,s}*(\nabla\cdot F)}\Vert_{L^r(\Omega)}
\le
\Vert {\nabla K_{t,s}}\Vert_{L^1(\Omega)}
\Vert{F}\Vert_{L^r(\Omega)} .
\end{equation*}
For the case of \(r=\infty\), since smooth functions are not dense in \(L^\infty\), we
define the convolution directly by
\begin{equation*}
\bigl[K_{t,s}*(\nabla\cdot F)\bigr](z)
:=
\int_\Omega
\nabla K_{t,s}(z-z')\cdot F(z')\,d z'.
\end{equation*}
Note that,
\begin{align*}
\left|
\int_\Omega
\nabla K_{t,s}(z-z')\cdot F(z')\,d z'
\right|
&\le
\Vert{F}\Vert_{L^\infty(\Omega)}
\int_\Omega
|\nabla K_{t,s}(z-z')|\,d z'=
\Vert {\nabla K_{t,s}}\Vert_{L^1(\Omega)}
\Vert{F}\Vert_{L^\infty(\Omega)},
\end{align*}
and thus,
\begin{equation*}
\Vert{K_{t,s}*(\nabla\cdot F)}\Vert_{L^\infty(\Omega)}
\le
\Vert {\nabla K_{t,s}}\Vert_{L^1(\Omega)}
\Vert {F}\Vert_{L^\infty(\Omega)}.
\end{equation*}
Finally, recall that the inverse transform $(x,y)\mapsto(x-yB_t^s,y)$ preserves every \(L^r\) norm. Hence, we have for all \(1\le r\le\infty\), that
\begin{align*}
\Vert {S_\nu(t,s)\nabla\cdot F}\Vert_{L^r(\Omega)}
&=
\Vert{K_{t,s}*(\nabla\cdot F)}\Vert_{L^r(\Omega)}=
\Vert{\nabla K_{t,s}*F}\Vert_{L^r(\Omega)}\\
&\le
\Vert {\nabla K_{t,s}}\Vert_{L^1(\Omega)}
\Vert {F}\Vert_{L^r(\Omega)}\le
\frac{C\Gamma_{t,s}}{\sqrt{\nu(t-s)}}
\Vert {F}\Vert_{L^r(\Omega)}.
\end{align*} 

Next, we will prove (\ref{eq:pathwise-Lp-Linfty-divergence}). For $Q_{t,s}$ defined in (\ref{eq:shear-matrix-definitions}), and $\Gamma_{t,s}$ and $\Lambda_{T}$ defined in (\ref{eq:Gamma-definition}), we have that 
\[
 \frac{t-s}{\Gamma_{t,s}^2}I\le Q_{t,s}
       \le(t-s)\Gamma_{t,s}^2I,
 \qquad
 \Gamma_{t,s}^2\le2\Lambda_T.
\]
Note that
$\det Q_{t,s}\ge(t-s)^2$ then implies,for some constant $C>0$, that
\[
 |\nabla K^{\mathbb R^2}_{t,s}(z)|
 \le C\Lambda_T^{1/2}[\nu(t-s)]^{-3/2}
             \exp\!\left(-\frac{|z|^2}{C\Lambda_T\nu(t-s)}\right).
\]

After periodization, we obtain
\[
 \|\nabla K_{t,s}\|_{L^\infty(\Omega)}
 \le C\Lambda_T^{1/2}[\nu(t-s)]^{-3/2}(1+\sqrt{\Lambda_T\nu(t-s)}),
 \qquad
 \|\nabla K_{t,s}\|_{L^1(\Omega)}
 \le C\Lambda_T^{1/2}[\nu(t-s)]^{-1/2}.
\]
Next, interpolation gives us for $p'=p/(p-1)$ and some constant $C_p>0$ that
\[
 \|\nabla K_{t,s}\|_{L^{p'}(\Omega)}
 \le C_p\Lambda_T^{\frac12+\frac1{2p}}
           (1+\sqrt{\nu T})^{1/p}
           [\nu(t-s)]^{-\frac12-\frac1p}.
\]
Thus, (\ref{eq:pathwise-Lp-Linfty-divergence}) follows from an application of Young's inequality since $\frac12+\frac1{2p}<1$, and the fact that $\Lambda_T\ge1$.

Finally, we will prove the strong continuity result (\ref{strong_continuity_semigroup}). We fix $\omega\in\mathfrak{D}$, an
initial time \(s\ge0\) and \(1\le r<\infty\) and first prove that
\begin{equation}\label{eq:strong-continuity-smooth}
\lim_{t\downarrow s}
\Vert {S_\nu(t,s)\varphi-\varphi}\Vert_{L^r(\Omega)}=0~
\text{for every }\varphi\in C_c^\infty(\Omega),
\end{equation}
where \(C_c^\infty(\Omega)\) denotes the set of smooth functions that are periodic
in \(x\) and compactly supported in \(y\).  We next utilize the representation $S_\nu(t,s)\varphi(x,y)
=
(K_{t,s}*\varphi)(x-yB_t^s,y)$ and the fact that the inverse transform preserves all \(L^r\) norms, to arrive at
\begin{align}
\Vert {S_\nu(t,s)\varphi-\varphi}\Vert_{L^r(\Omega)}
&\le
\Vert {
(K_{t,s}*\varphi)(x-yB_t^s,y)
-
\varphi(x-yB_t^s,y)
}\Vert_{L^r(\Omega)}\notag\\
&\qquad+
\Vert {
\varphi(x-yB_t^s,y)-\varphi(x,y)
}\Vert_{L^r(\Omega)}\notag\\
&=
\Vert {K_{t,s}*\varphi-\varphi}\Vert_{L^r(\Omega)}+
\Vert {
\varphi(x-yB_t^s,y)-\varphi(x,y)
}\Vert_{L^r(\Omega)}.
\label{eq:strong-continuity-decomposition}
\end{align}
We will bound the right hand side terms of \eqref{eq:strong-continuity-decomposition} starting with the convolution term.  Recall that the covariance matrix  of
\(K_{t,s}^{\mathbb R^2}\) is \(2\nu Q_{t,s}\), and thus 
\[
\int_{\mathbb R^2}
|z|^2K_{t,s}^{\mathbb R^2}(z)\,d z
=
2\nu\operatorname{tr}Q_{t,s}.
\]
We use the notation \(\delta=t-s\) again, which gives us
\begin{align*}
\operatorname{tr}Q_{t,s}
&=
2\delta+\int_s^t|W_a-W_s|^2\,d a\le
\delta
\left(
2+\sup_{s\le a\le t}|W_a-W_s|^2
\right).
\end{align*}
Thanks to the continuity of the process $W_t$, we find that $\operatorname{tr}Q_{t,s}\rightarrow0$
as $t\downarrow s.$
Therefore, for every \(\varepsilon>0\), Chebyshev's inequality gives us
\begin{align}
\int_{\{|z|>\varepsilon\}}
K_{t,s}^{\mathbb R^2}(z)\,d z
&\le
\frac{1}{\varepsilon^2}
\int_{\mathbb R^2}
|z|^2K_{t,s}^{\mathbb R^2}(z)\,d z=
\frac{2\nu}{\varepsilon^2}
\operatorname{tr}Q_{t,s}
\rightarrow0.
\label{eq:Gaussian-tail-to-zero}
\end{align}
Since \(K_{t,s}\) is obtained by periodizing
\(K_{t,s}^{\mathbb R^2}\), we obtain the same estimate for it i.e.  $\int_{\{d_\Omega(z,0)>\varepsilon\}}
K_{t,s}(z)\,d z
\rightarrow0$.
Now, we use Minkowski's integral inequality which leads to
\begin{align*}
\Vert {K_{t,s}*\varphi-\varphi}\Vert_{L^r(\Omega)}
&=
\Vert {
\int_\Omega
K_{t,s}(z')
\bigl[\varphi(\,\cdot-z')-\varphi(\,\cdot\,)\bigr]
\,dz'
}\Vert_{L^r(\Omega)}\le
\int_\Omega
K_{t,s}(z')
\Vert {
\varphi(\,\cdot-z')-\varphi(\,\cdot\,)
}\Vert_{L^r(\Omega)}
\,d z'.
\end{align*}
For any \(\eta>0\), by using the strong continuity property of translations on
\(L^r(\Omega)\), we obtain for some \(\varepsilon>0\),  that
\[
\Vert {
\varphi(\,\cdot-z')-\varphi(\,\cdot\,)\Vert_{L^r(\Omega)}
}
<\eta
\qquad
\text{if }d_\Omega(z',0)<\varepsilon.
\]
We further use
\(\Vert {K_{t,s}}\Vert_{L^1(\Omega)}=1\) to obtain
\begin{align*}
\Vert{K_{t,s}*\varphi-\varphi}\Vert_{L^r(\Omega)}
&\le
\eta
+
2\Vert {\varphi}\Vert_{L^r(\Omega)}
\int_{\{d_\Omega(z',0)\ge\varepsilon\}}
K_{t,s}(z')\,d z'.
\end{align*}
In light of \eqref{eq:Gaussian-tail-to-zero}, for any $\eta>0$, we arrive at the fact that $\limsup_{t\downarrow s}
\Vert {K_{t,s}*\varphi-\varphi}\Vert_{L^r(\Omega)}
\le\eta,$ which means that we have
\begin{equation}\label{eq:kernel-approximate-identity}
\lim_{t\downarrow s}
\Vert {K_{t,s}*\varphi-\varphi}\Vert_{L^r(\Omega)}=0.
\end{equation}
We next treat the second term of the right side in (\ref{eq:strong-continuity-decomposition}) by applying the fundamental theorem of calculus, as follows. We write,
\begin{align*}
&\varphi(x-yB_t^s,y)-\varphi(x,y)=
-yB_t^s
\int_0^1
\partial_x\varphi(x-\theta yB_t^s,y)\,d\theta.
\end{align*}
By using Minkowski's inequality and again the fact that translations in \(x\)
preserve \(L^r\) norms, we obtain
\begin{align*}
\Vert {
\varphi(x-yB_t^s,y)-\varphi(x,y)
}\Vert_{L^r(\Omega)}
&\le
|B_t^s|
\int_0^1
\Vert{
y\partial_x\varphi(x-\theta yB_t^s,y)
}\Vert_{L^r(\Omega)}
\,d\theta\\
&=
|B_t^s|
\Vert{y\partial_x\varphi}\Vert_{L^r(\Omega)}.
\end{align*}
The continuity of Brownian motion continuity implies $B_t^s=W_t-W_s\rightarrow0$ as $t\downarrow s.$
Hence,
\begin{equation}\label{eq:shear-to-identity}
\lim_{t\downarrow s}
\Vert{
\varphi(x-yB_t^s,y)-\varphi(x,y)\Vert_{L^r(\Omega)}
}
=0.
\end{equation}
Collecting \eqref{eq:strong-continuity-decomposition},
\eqref{eq:kernel-approximate-identity} and
\eqref{eq:shear-to-identity} proves
\eqref{eq:strong-continuity-smooth}.

Finally, consider any \(f\in L^r(\Omega)\). Since for \(r<\infty\), the space
\(C_c^\infty(\Omega)\) is dense in \(L^r(\Omega)\), for any
\(\varepsilon>0\), we can choose
\(\varphi\in C_c^\infty(\Omega)\) such that $\Vert{f-\varphi}\Vert_{L^r(\Omega)}<\varepsilon.$
Thanks to (\ref{eq:Lr-contraction}), we obtain
\begin{align*}
\Vert{S_\nu(t,s)f-f}\Vert_{L^r(\Omega)}
&\le
\Vert{S_\nu(t,s)(f-\varphi)}\Vert_{L^r(\Omega)}
+
\Vert{S_\nu(t,s)\varphi-\varphi}\Vert_{L^r(\Omega)}
+
\Vert{\varphi-f}\Vert_{L^r(\Omega)}\\
&\le
2\varepsilon
+
\Vert{S_\nu(t,s)\varphi-\varphi}\Vert_{L^r(\Omega)}.
\end{align*}
By applying the limit \(t\downarrow s\) we get that $
\limsup_{t\downarrow s}
\Vert {S_\nu(t,s)f-f}\Vert_{L^r(\Omega)}
\le
2\varepsilon$ for any
 \(\varepsilon>0\). In other words, we have $\lim_{t\downarrow s}
\Vert{S_\nu(t,s)f-f}\Vert_{L^r(\Omega)}=0.$
Thus, (\ref{strong_continuity_semigroup}) holds for $f\in L^r(\Omega).$

\end{proof}
We next establish moment estimates for the random solution operator
\(S_\nu(t,s)\) to (\ref{eq:passive-local}).
\begin{proposition}
\label{prop:expected-maximal}
Let $F$ be a process adapted to $\{\mathcal{F}_t\}_{t\geq0}$ satisfying
$\mathbb E\sup_{a\le s\le T}\Vert {F(s)}\Vert_{X}^q<\infty$ for any $q>2$. Then we have,
\begin{equation}\label{eq:expected-duhamel}
 \begin{aligned}
 &\left(\mathbb E\sup_{a\le t\le T}
 \Vert {\nu\int_a^tS_\nu(t,s)\nabla\cdot F(s)\,d s}\Vert_{X}^q\right)^{1/q}\le C_{p,q}(1+\sqrt{T-a})\sqrt{\nu(T-a)}
 \left(\mathbb E\sup_{a\le s\le T}\Vert {F(s)}\Vert_{X}^q\right)^{1/q},
 \end{aligned}
\end{equation}
where the space $X$ is defined in \eqref{eq:X-definition}.

Moreover, $\int_a^tS_\nu(t,s)\nabla\cdot F(s)\,d s$ has continuous paths in $X$.
 Similarly, if for some $a>0$, $f$ is
$\mathcal F_a$-measurable, then
\begin{equation}\label{eq:expected-homogeneous}
 \left(\mathbb E\sup_{a\le t\le T}\Vert{S_\nu(t,a)f}\Vert_{X}^q\right)^{1/q}
 \le\Vert{f}\Vert_{L^q(\mathfrak O;X)}.
\end{equation}
\end{proposition}
\begin{proof}
For $a\le s<T$, set  $\Gamma^*_{s,T}=\sup_{s<t\le T}\Gamma_{t,s}.$
By applying the inequality \eqref{eq:pathwise-divergence-bound} in $L^1$ and $L^p$, we obtain
\begin{equation}\label{eq:pathwise-duhamel-integral}
\bigg\Vert {\nu\int_a^tS_\nu(t,s)\nabla\cdot F(s)\,d s}\bigg\Vert_{X}
 \le C\sqrt\nu\int_a^t(t-s)^{-1/2}
 \Gamma^*_{s,T}\Vert {F(s)}\Vert_{X}\,d s.
\end{equation}
For $q'=q/(q-1)<2$, we apply H\"older's inequality in time to obtain
\begin{align}
 &\int_a^t(t-s)^{-1/2}\Gamma^*_{s,T}
\Vert{F(s)}\Vert_{X}\,d s \le C_q(T-a)^{1/2-1/q}
 \left(\int_a^T(\Gamma^*_{s,T})^q
\Vert{F(s)}\Vert_{X}^q\,d s\right)^{1/q}.
 \label{eq:time-holder-maximal}
\end{align}
After raising both sides to $q$-th power, we have
\begin{align}
 &\mathbb E\sup_{a\le t\le T}
\Vert{\nu\int_a^tS_\nu(t,s)\nabla\cdot F(s)\,d s}\Vert_{X}^q\le C_q\nu^{q/2}(T-a)^{q/2-1}
 \int_a^T\mathbb E[(\Gamma^*_{s,T})^q\Vert{F(s)}\Vert_{X}^q]\,\,d s.
 \label{eq:before-independence}
\end{align}
Note that $F(s)$ is $\mathcal F_s$-measurable, whereas
$\Gamma^*_{s,T}$ is a functional only of the increments after $s$.
Therefore,
\begin{equation}\label{eq:conditional-factorization}
 \mathbb E[(\Gamma^*_{s,T})^q\Vert{F(s)}\Vert_{X}^q]
 =\mathbb E[(\Gamma^*_{s,T})^q]\mathbb E[\Vert{F(s)}\Vert_{X}^q].
\end{equation}
In light of  \eqref{eq:Gamma-definition}, we have
\begin{align}\label{Gammastar_estimate}
 (\Gamma^*_{s,T})^2
 \le2+\sup_{s\le r\le T}|W_r-W_s|^2.
\end{align}
By using Doob's maximal inequality and rescaling of Brownian motion, we find
\begin{equation}\label{eq:Gamma-star-moment}
 \sup_{a\le s<T}\mathbb E[(\Gamma^*_{s,T})^q]
 \le C_q\bigl(1+(T-a)^{q/2}\bigr).
\end{equation}
Indeed, we have from Doob's \(L^q\) maximal inequality that
\begin{align*}
\mathbb E\left[
\sup_{s\le r\le T}|W_r-W_s|^q
\right]
&\leq
\left(\frac{q}{q-1}\right)^q
\mathbb E|W_T-W_s|^q.
\end{align*}
Note that, by rescaling the Brownian motion, we have $W_T-W_s
\overset{\mathrm{d}}=
\sqrt{T-s}\,Z$ where $Z\sim N(0,1)$. Therefore,
\[
\mathbb E|W_T-W_s|^q
=
(T-s)^{q/2}\mathbb E|Z|^q
\le
C_q(T-s)^{q/2},
\]
for some constant $C_q>0$.
Consequently, we have
\begin{equation}\label{eq:Brownian-maximal-moment}
\mathbb E\left[
\sup_{s\le u\le T}|W_u-W_s|^q
\right]
\le
C_q(T-s)^{q/2}.
\end{equation}
In addition, we have from (\ref{Gammastar_estimate}) that
\begin{align}\label{q-moment-Gammastar}
(\Gamma^*_{s,T})^q
\le
C_q
\left(
1+\sup_{s\le r\le T}|W_r-W_s|^q
\right).
\end{align}
By taking expectations in (\ref{q-moment-Gammastar}) and using
\eqref{eq:Brownian-maximal-moment}, we obtain
\[
\mathbb E[(\Gamma^*_{s,T})^q]
\le
C_q\left(1+(T-s)^{q/2}\right),
\]
which further implies that,
\[
\sup_{a\le s<T}
\mathbb E[(\Gamma^*_{s,T})^q]
\le
C_q\left(1+(T-a)^{q/2}\right).
\]
Substituting \eqref{eq:conditional-factorization} and
\eqref{eq:Gamma-star-moment} into \eqref{eq:before-independence}, we obtain that
\begin{align*}
\mathbb E\sup_{a\le t\le T}
\Vert{
\nu\int_a^t
S_\nu(t,s)\nabla\cdot F(s)\,d s
}\Vert_{X}^q&\le
C_q\nu^{q/2}(T-a)^{q/2-1}
\int_a^T
\mathbb E[(\Gamma^*_{s,T})^q]\,
\mathbb E\Vert{F(s)}\Vert_X^q\,d s\\
&\quad\le
C_q\nu^{q/2}(T-a)^{q/2-1}
\left(1+(T-a)^{q/2}\right)
\int_a^T
\mathbb E\Vert{F(s)}\Vert_X^q\,d s.
\end{align*}
Since $\mathbb E\Vert{F(s)}\Vert_X^q
\le
\mathbb E\sup_{a\le r\le T}\Vert{F(r)}\Vert_{X}^q,$
we have
\begin{align*}
\int_a^T\mathbb E\Vert{F(s)}\Vert_{X}^q\,d s
&\le
(T-a)
\mathbb E\sup_{a\le r\le T}\Vert{F(r)}\Vert_{X}^q.
\end{align*}
Hence, we obtain
\begin{align*}
&\mathbb E\sup_{a\le t\le T}
\bigg\Vert{
\nu\int_a^t
S_\nu(t,s)\nabla\cdot F(s)\,d s
}\bigg\Vert_{X}^q \le
C_q\nu^{q/2}(T-a)^{q/2}
\left(1+(T-a)^{q/2}\right)
\mathbb E\sup_{a\le r\le T}\Vert{F(r)}\Vert_{X}^q.
\end{align*}
By taking the \(q\)-th root on both sides and using $\left(1+(T-a)^{q/2}\right)^{1/q}
\le
1+\sqrt{T-a},$ we arrive at
\begin{align*}
&\left(
\mathbb E\sup_{a\le t\le T}
\Vert{
\nu\int_a^t
S_\nu(t,s)\nabla\cdot F(s)\,d s
}\Vert_{X}^q
\right)^{1/q}\le
C_{p,q}(1+\sqrt{T-a})\sqrt{\nu(T-a)}
\left(
\mathbb E\sup_{a\le s\le T}\Vert{F(s)}\Vert_{X}^q
\right)^{1/q},
\end{align*}
which proves \eqref{eq:expected-duhamel}.

Next we consider the Duhamel term which we denote by,
\[
D(t)
:=
\nu\int_a^tS_\nu(t,s)\nabla\cdot F(s)\,\,d s,
\qquad a\le t\le T.
\]
We will next prove that \(D\) has continuous paths in \(X\).  We start by first showing right continuity. Let \(h>0\) be such that \(t+h\le T\), then we obtain
\begin{align*}
D(t+h)
&=
\nu\int_a^t
S_\nu(t+h,s)\nabla\cdot F(s)\,d s
+
\nu\int_t^{t+h}
S_\nu(t+h,s)\nabla\cdot F(s)\,d s.
\end{align*}
For \(a\le s\le t\), the fact $S_\nu(t+h,s)
=
S_\nu(t+h,t)S_\nu(t,s)$ implies that
\begin{align}
D(t+h)
&=
S_\nu(t+h,t)D(t)+
\nu\int_t^{t+h}
S_\nu(t+h,s)\nabla\cdot F(s)\,d s,
\label{eq:Duhamel-right-evolution}
\end{align}
which implies that 
\begin{align}
D(t+h)-D(t)
&=
\bigl(S_\nu(t+h,t)-I\bigr)D(t)+
\nu\int_t^{t+h}
S_\nu(t+h,s)\nabla\cdot F(s)\,d s.
\label{eq:Duhamel-right-difference}
\end{align}
 As shown in Proposition \ref{prop:pathwise-kernel}, we have that \(S_\nu\) is continuous on both
\(L^1(\Omega)\) and \(L^p(\Omega)\).  Moreover, since \(D(t)\in X\), we obtain
\[
\Vert{
\bigl(S_\nu(t+h,t)-I\bigr)D(t)
}\Vert_{X}
\rightarrow0
~\text{as }h\downarrow0.
\]
In addition, for the second term in the right hand side of \eqref{eq:Duhamel-right-difference}, one finds from  (\ref{eq:pathwise-divergence-bound}) that
\begin{align*}
\Vert{
\nu\int_t^{t+h}
S_\nu(t+h,s)\nabla\cdot F(s)\,d s
}\Vert_{X} &\le
C\sqrt\nu
\int_t^{t+h}
(t+h-s)^{-1/2}
\Gamma_{t+h,s}
\Vert{F(s)}\Vert_{X}\,d s\\
&\le
C\sqrt\nu
\left(
\sup_{t\le s<u\le t+h}\Gamma_{u,s}
\right)
\left(
\sup_{t\le s\le t+h}\Vert{F(s)}\Vert_{X}
\right)
\int_t^{t+h}(t+h-s)^{-1/2}\,d s\\
&=
2C\sqrt{\nu h}
\left(
\sup_{t\le s<u\le t+h}\Gamma_{u,s}
\right)
\left(
\sup_{t\le s\le t+h}\Vert{F(s)}\Vert_{X}
\right).
\end{align*}
For \(t\le s<u\le t+h\),  we conclude from the definition of \(\Gamma_{u,s}\) shown in (\ref{eq:Gamma-definition}) that
\begin{align*}
\Gamma_{u,s}^2
&=
2+\frac1{u-s}
\int_s^u|W_r-W_s|^2\,d r\le
2+\sup_{s\le r\le u}|W_r-W_s|^2\le
2+\sup_{t\le r_1,r_2\le t+h}
|W_{r_1}-W_{r_2}|^2.
\end{align*}
Consequently,
\[
\sup_{t\le s<u\le t+h}\Gamma_{u,s}
\le
\left(
2+\sup_{t\le r_1,r_2\le t+h}
|W_{r_1}-W_{r_2}|^2
\right)^{1/2}.
\]
By continuity of the Brownian path, we obtain $\sup_{t\le r_1,r_2\le t+h}
|W_{r_1}-W_{r_2}|
\rightarrow0$ almost surely as $h\downarrow0.$
Moreover, the assumption $\mathbb E\sup_{a\le s\le T}\Vert{F(s)}\Vert_{X}^q<\infty$
implies that $\sup_{a\le s\le T}\Vert{F(s)}\Vert_{X}<\infty$
\text{almost surely}.  It follows that
\[
\Vert{
\nu\int_t^{t+h}
S_\nu(t+h,s)\nabla\cdot F(s)\,d s
}\Vert_{X}
\rightarrow0
\qquad\text{almost surely}.
\]
Together with \eqref{eq:Duhamel-right-difference}, this proves the right continuity result: $\Vert{D(t+h)-D(t)}\Vert_{X}
\rightarrow0$\text{ as }$h\downarrow0$. \\
Next, we will prove left continuity. Fix \(t\in(a,T]\) and
\(0<\varepsilon<t-a\). For \(0<h<\varepsilon/2\), we calculate
\begin{align}
D(t)-D(t-h)
&=
\nu\int_a^{t-\varepsilon}
\left[
S_\nu(t,s)-S_\nu(t-h,s)
\right]
\nabla\cdot F(s)\,d s\notag\\
&\qquad+
\nu\int_{t-\varepsilon}^{t}
S_\nu(t,s)\nabla\cdot F(s)\,d s-
\nu\int_{t-\varepsilon}^{t-h}
S_\nu(t-h,s)\nabla\cdot F(s)\,d s.
\label{eq:Duhamel-left-decomposition}
\end{align}
We use (\ref{prop:pathwise-kernel}) to obtain
\begin{align}
\Vert{
\nu\int_{t-\varepsilon}^{t}
S_\nu(t,s)\nabla\cdot F(s)\,d s
}\Vert_{X}&+
\Vert{
\nu\int_{t-\varepsilon}^{t-h}
S_\nu(t-h,s)\nabla\cdot F(s)\,d s
}\Vert_{X}\notag\\
&\le
4C\sqrt{\nu\varepsilon}
\left(
\sup_{t-\varepsilon\le s<u\le t}\Gamma_{u,s}
\right)
\left(
\sup_{a\le s\le T}\Vert{F(s)}\Vert_{X}
\right).
\label{eq:Duhamel-left-near}
\end{align}
For a fixed path, we know that $\sup_{t-\varepsilon\le s<u\le t}\Gamma_{u,s}$
remains bounded as \(\varepsilon\downarrow0\).   Then, we consider the first term on the right hand side of
\eqref{eq:Duhamel-left-decomposition}. For every fixed
\(s\le t-\varepsilon\), we obtain  $S_\nu(r,s)\nabla\cdot F(s)
=
\bigl(\nabla K_{r,s}*F(s)\bigr)
(x-yB_r^s,y).$
As \(r\to t\), $B_r^s\rightarrow B_t^s$
by the continuity of Brownian motion  and $Q_{r,s}\rightarrow Q_{t,s}$
by the definition of \(Q_{r,s}\) given in (\ref{eq:Q-explicit}). Since \(t-s\ge\varepsilon\),
the covariance matrices remain uniformly positive definite for \(r\)
sufficiently close to \(t\). Invoking the  dominated
convergence theorem, we obtain $\Vert{
\nabla K_{r,s}-\nabla K_{t,s}
}\Vert_{L^1(\Omega)}
\rightarrow0.$
The strong continuity of the transform on \(L^1\) and \(L^p\), together with
Young's inequality, then implies as $r\rightarrow t,$ $\Vert{
S_\nu(r,s)\nabla\cdot F(s)
-
S_\nu(t,s)\nabla\cdot F(s)
}\Vert_{X}
\rightarrow0$
for every fixed \(s\le t-\varepsilon\).  Moreover, for \(r\) sufficiently close to \(t\) and
\(s\le t-\varepsilon\), thanks to estimate (\ref{eq:pathwise-divergence-bound}), we find 
\[
\Vert{S_\nu(r,s)\nabla\cdot F(s)}\Vert_{X}
\le
C_{\nu,T,\varepsilon}\Vert{F(s)}\Vert_{X}.
\]
Since $\sup_{a\le s\le T}\Vert{F(s)}\Vert_{X}<\infty$ \text{almost surely}, the dominated convergence theorem yields
\begin{align*}
&\Vert{
\nu\int_a^{t-\varepsilon}
\left[
S_\nu(t,s)-S_\nu(t-h,s)
\right]
\nabla\cdot F(s)\,d s
}\Vert_{X}\rightarrow0
\text{ as }h\downarrow0.
\end{align*}
After collecting the above estimates, we let
\(h\downarrow0\) with \(\varepsilon\) fixed, and then let
\(\varepsilon\downarrow0\) in the right hand side of \eqref{eq:Duhamel-left-decomposition}. This gives us that,
\[
\Vert{D(t-h)-D(t)}\Vert_{X}
\rightarrow0
\text{ as }h\downarrow0.
\]
Thus, this proves that $D$ has left-continuous paths in $X$.  Combining the two results above, in summary, we have proved that \(D\) has continuous paths in \(X\).


Next, we will show that \eqref{eq:expected-homogeneous} holds.
For that purpose, let \(f\in L^q(\mathfrak O;X)\) be
\(\mathcal F_a\)-measurable. Then for every
\(t\in[a,T]\),  thanks to \eqref{eq:Lr-contraction}, we have the following estimates:
\[
\Vert{S_\nu(t,a)f}\Vert_{L^1(\Omega)}
\le
\Vert{f}\Vert_{L^1(\Omega)},\quad\text{ and }\quad~\Vert{S_\nu(t,a)f}\Vert_{L^p(\Omega)}
\le\Vert{f}\Vert_{L^p(\Omega)}.
\]
Adding these two inequalities, we obtain
\begin{align*}
\Vert{S_\nu(t,a)f}\Vert_{X}
&=
\Vert{S_\nu(t,a)f}\Vert_{L^1}
+
\Vert{S_\nu(t,a)f}\Vert_{L^p(\Omega)}\le
\Vert{f}\Vert_{L^1}+\Vert{f}\Vert_{L^p(\Omega)}=
 \Vert{f}\Vert_{X}.
\end{align*}
Since this estimate holds for every \(t\in[a,T]\), it follows 
that $\sup_{a\le t\le T}
\Vert{S_\nu(t,a)f}\Vert_{X}^q
\le
 \Vert{f}\Vert_{X}^q.$
Taking expectation on both sides gives $\mathbb E\sup_{a\le t\le T}
\Vert{S_\nu(t,a)f}\Vert_{X}^q
\le
\mathbb E \Vert{f}\Vert_{X}^q.$
Finally, after taking \(q\)-th root of both sides, we conclude that
\[
\left(
\mathbb E\sup_{a\le t\le T}
\Vert{S_\nu(t,a)f}\Vert_{X}^q
\right)^{1/q}
\le
\left(
\mathbb E\Vert{f}\Vert_{X}^q
\right)^{1/q}
=\Vert{f}\Vert_{L^q(\mathfrak O;X)}.
\]
This finally proves \eqref{eq:expected-homogeneous}.
Thus we conclude the proof of Proposition \ref{prop:expected-maximal}.
\end{proof}

\section{Local-in-time existence}\label{sec:existence}
This section is devoted to the existence of a maximal local mild solution to (\ref{eq:KS-local}). In other words we will present a proof for Theorem \ref{thm:local-complete}. First of all, we consider an approximation system constructed using a smooth cutoff function and show the existence of the unique global-in-time to this approximate system via the Banach fixed point theorem. This existence result for the approximation system is the focus on the next subsection.
\subsection{Global existence of approximation systems}\label{sec:approx}
We choose $\theta\in C^1([0,\infty);[0,1])$ satisfying
\begin{equation*}
 \theta=1\text{ on }[0,1],\qquad\theta=0\text{ on }[2,\infty),
 \qquad\Vert{\theta'}\Vert_{L^\infty}\le C.
\end{equation*}
and set $\theta_m(r)=\theta(r/m)$. Using this cutoff function, we consider the following mild representation.
\begin{equation}\label{eq:cutoff-mild-global}
 n^m(t)=S_\nu(t,0)n_{\rm in}
 -\nu\int_0^tS_\nu(t,s)\nabla\cdot[
 \theta_m(\Vert{n^m}\Vert_{X_s})n^m(s)\nabla(-\Delta)^{-1}n^m(s)]\,d s,
\end{equation}
The aim of this section is to prove the existence of a global-in-time solution $n^m$ to \eqref{eq:cutoff-mild-global} (cf. Definition \ref{def:local-mild}).

Before diving into the proof, we present some preliminary results involving the cutoff function $\theta_m$. For ease of notation, we denote by
\begin{equation}\label{eq:history-norm}
\Vert{u}\Vert_{X_t}:=\sup_{0\le s\le t}\Vert{u(s)}\Vert_{X}.
\end{equation}
\begin{lemma}\label{lem:cutoff-global-lipschitz}
For $u,v\in X$, we have for any $t>0$ that
\begin{align}
\Vert{\theta_m(\Vert{u}\Vert_{X_t})u(t)\nabla(-\Delta)^{-1}u(t)}\Vert_{X}
 \le4C_pm^2&,\label{eq:cutoff-global-bound}
 \end{align}
 and
 \begin{align}
&\Vert{\theta_m(\Vert{u}\Vert_{X_t})u(t)\nabla(-\Delta)^{-1}u(t)
 -\theta_m(\Vert{v}\Vert_{X_t})v(t)\nabla(-\Delta)^{-1}v(t)}\Vert_{X}
 \le C_pm\sup_{0
\le s\le t}\Vert{u(s)-v(s)}\Vert_{X},
 \label{eq:cutoff-global-lipschitz}
\end{align}
where the norm $\|\cdot\|_{X_t}$ is given by \eqref{eq:history-norm}.
\end{lemma}
\begin{proof}
 Since \(0\le\theta_m\le1\), invoking H\"older's inequality and
\eqref{eq:X-Poisson}, we obtain for \(r=1,p\) that,
\begin{align*}
\Vert{
\theta_m(\Vert{u}\Vert_{X_t})
u(t)\nabla(-\Delta)^{-1}u(t)
}\Vert_{L^r(\Omega)}&\le
\theta_m( \Vert{u}\Vert_{X_t})
\Vert{u(t)}\Vert_{L^r(\Omega)}
 \Vert{
\nabla(-\Delta)^{-1}u(t)
}\Vert_{L^\infty(\Omega)}\\
&\quad\le
C_p\theta_m( \Vert{u}\Vert_{X_t})
 \Vert{u(t)}\Vert_{L^r(\Omega)}
 \Vert{u(t)}\Vert_{X}.
\end{align*}
Adding the estimates for \(r=1\) and \(r=p\), we obtain
\begin{align*}
&\Vert{
\theta_m(\Vert{u}\Vert_{X_t})
u(t)\nabla(-\Delta)^{-1}u(t)
}\Vert_{X}\le
C_p\theta_m(\Vert{u}\Vert_{X_t})
 \Vert{u(t)}\Vert_X^2.
\end{align*}
If \(\theta_m(\Vert{u}\Vert_{X_t})=0\), then we have (\ref{eq:cutoff-global-bound}) holds.
If \(\theta_m(\Vert{u}\Vert_{X_t})\neq0\), then the definition of the cutoff
implies $\Vert {u}\Vert_{X_t}<2m.$
Since $\Vert{u(t)}\Vert_{X}
\le
\Vert {u}\Vert_{X_t},$
we conclude that
\begin{align*}
&\Vert{
\theta_m(\Vert {u}\Vert_{X_t})
u(t)\nabla(-\Delta)^{-1}u(t)
}\Vert_{X}\le
C_p\Vert {u(t)}\Vert_{X}^2\le
C_p(2m)^2=
4C_pm^2.
\end{align*}
This proves \eqref{eq:cutoff-global-bound}.

Next, due to the fact that  $|\Vert {u}\Vert_{X_t}-\Vert{v}\Vert_{X_t}|\le \sup_{s\le t} \Vert {u(s)-v(s)}\Vert_{X}$, the mean-value theorem gives us,
\begin{align}\label{lip_estimate}
|\theta_m(\Vert {u}\Vert_{X_t})-\theta_m(\Vert {v}\Vert_{X_t})|
\le
\frac Cm|\Vert {u}\Vert_{X_t}-\Vert {v}\Vert_{X_t}|
\le
\frac Cm\sup_{s\le t}\Vert {u(s)-v(s)}\Vert_{X}.
\end{align}
Without loss of generality,  we assume $\Vert{u}\Vert_{X_t}\le  \Vert {v}\Vert_{X_t}$ and use the following identity:
\begin{align}\label{right_hand_side_vanishing}
 \theta_m(\Vert{u}\Vert_{X_t})u\nabla(-\Delta)^{-1}u
 -\theta_m(\Vert {v}\Vert_{X_t})v\nabla(-\Delta)^{-1}v&=\theta_m(\Vert {v}\Vert_{X_t})[u\nabla(-\Delta)^{-1}u
 -v\nabla(-\Delta)^{-1}v]
 \nonumber\\
 &+[\theta_m(\Vert{u}\Vert_{X_t})-\theta_m(\Vert {v}\Vert_{X_t})]u\nabla(-\Delta)^{-1}u.
\end{align}
If \(\theta_m(\Vert {v}\Vert_{X_t})=0\), then the first term on the right-hand side of (\ref{right_hand_side_vanishing}) vanishes.
Otherwise, if \(\theta_m(\Vert {v}\Vert_{X_t})\neq0\) then the definition of the cutoff implies $\Vert {v}\Vert_{X_t}<2m.$  Since \(\Vert {u}\Vert_{X_t}\le \Vert {v}\Vert_{X_t}\), we also know that \(\Vert {u}\Vert_{X_t}<2m\). 
By using \eqref{eq:X-difference}, we thus obtain
\begin{align}\label{combine_1_estimate}
&\Vert {
\theta_m(\Vert {v}\Vert_{X_t}) 
\left[
u(t)\nabla(-\Delta)^{-1}u(t)
-v(t)\nabla(-\Delta)^{-1}v(t)
\right]
}\Vert_{X}\nonumber\\
&\quad\le
C_p
\left(
\Vert {u(t)}\Vert_{X}+\Vert {v(t)}\Vert_{X}
\right)
 \Vert {u(t)-v(t)}\Vert_{X}\nonumber\\
&\quad\le
C_p(\Vert {u}\Vert_{X_t}+\Vert {v}\Vert_{X_t})\sup_{s\le t} \Vert {u(s)-v(s)}\Vert_{X}\le4C_pm\sup_{s\le t} \Vert {u(s)-v(s)}\Vert_{X}.
\end{align}

Next, we treat the second term on the right hand side of \eqref{right_hand_side_vanishing}.
If \(\theta_m(\Vert {u}\Vert_{X_t})-\theta_m(\Vert {v}\Vert_{X_t})=0\), the second term on the right hand side of (\ref{right_hand_side_vanishing}) vanishes. Otherwise,
we must have \(\Vert {u}\Vert_{X_t}<2m\). 
Using the Lipschitz estimate for the cutoff shown in \eqref{lip_estimate}
along with the Poisson estimate \eqref{eq:X-quadratic}, we obtain, for some constant $C_p>0$, that 
\begin{align}\label{combine_2_estimate}
\Vert{
\left[
\theta_m(\Vert {u}\Vert_{X_t})-\theta_m(\Vert {v}\Vert_{X_t})
\right]
u(t)\nabla(-\Delta)^{-1}u(t)
}\Vert_{X}&\le
|\theta_m(\Vert {u}\Vert_{X_t})-\theta_m(\Vert {v}\Vert_{X_t})|
\Vert{
u(t)\nabla(-\Delta)^{-1}u(t)
}\Vert_{X}\nonumber\\
&\le
\frac {C_p}m \sup_{s\le t}\Vert {u(s)-v(s)}\Vert_{X}\,\Vert{u(t)}\Vert_{X}^2\nonumber\\
&\le
4C_pm\sup_{s\le t}\Vert {u(s)-v(s)}\Vert_{X}.
\end{align}
Hence, combining (\ref{combine_1_estimate}) and (\ref{combine_2_estimate}), we
obtain the desired estimate \eqref{eq:cutoff-global-lipschitz}.
\end{proof}
Now, we are ready to establish the existence of a global (mild) solution to the cutoff equation \eqref{eq:cutoff-mild-global}.  
\begin{lemma}\label{prop:global-cutoff}
For every integer $m\ge1$ and $T<\infty$, \eqref{eq:cutoff-mild-global} admits
a unique global adapted solution $n^m$ satisfying
$n^m\in L^q(\mathfrak O;C([0,T];X))$ for every $q\ge2$.  
\end{lemma}
\begin{proof}
For $q>2$ and $0\le a<T$, let $\mathcal S^q_{a,T}$ be the complete space of adapted, pathwise continuous $X$-valued processes with
\begin{equation}\label{eq:expected-sup-space}
 \Vert {u}\Vert_{\mathcal S^q_{a,T}}
 =\left(\mathbb E\sup_{a\le t\le T}\Vert{u(t)}\Vert_{X}^q\right)^{1/q}.
\end{equation}

Fix $q_0>2$ and define the mapping $ \Phi_m$ on $\mathcal S^{q_0}_{0,T}$ as
\[
 \Phi_m(u)(t)=S_\nu(t,0)n_{\rm in}
 -\nu\int_0^tS_\nu(t,s)\nabla\cdot[
 \theta_m(\Vert{u}\Vert_{X_s})u(s)\nabla(-\Delta)^{-1}u(s)]\,d s.
\]
For \(u\in\mathcal S^{q_0}_{0,T}\), define $F_u(s)
=
\theta_m(\Vert{u}\Vert_{X_s})
u(s)\nabla(-\Delta)^{-1}u(s).$
Since \(u\) is adapted and pathwise continuous, $\Vert{u}\Vert_{X_s}
=
\sup_{0\le r\le s}\Vert {u(r)}\Vert_{X}$
is \(\mathcal F_s\)-measurable. Hence \(F_u\) is $(\mathcal{F}_t)_{t\geq 0}-$adapted. We now aim to apply
Proposition~\ref{prop:expected-maximal}  to \(F_u\).

By the definition of \(\Phi_m\), we immediately have,
\begin{align}
\Vert {\Phi_m(u)}\Vert_{\mathcal S^{q_0}_{0,T}}
&\le
\left(
\mathbb E\sup_{0\le t\le T}
\Vert {S_\nu(t,0)n_{\rm in}}\Vert_{X}^{q_0}
\right)^{1/q_0}+
\left(
\mathbb E\sup_{0\le t\le T}
\Vert {
\nu\int_0^t
S_\nu(t,s)\nabla\cdot F_u(s)\,d s
}\Vert_{X}^{q_0}
\right)^{1/q_0}.
\label{eq:Phi-map-triangle}
\end{align}
Since, by assumption \(n_{\rm in}\) is deterministic, the  estimate
\eqref{eq:expected-homogeneous} yields the following bound for the first term on the right hand side of \eqref{eq:Phi-map-triangle}:
\begin{align}
\left(
\mathbb E\sup_{0\le t\le T}
\Vert {S_\nu(t,0)n_{\rm in}}\Vert_{X}^{q_0}
\right)^{1/q_0}
&\le
 \Vert{n_{\rm in}}\Vert_{L^{q_0}(\mathfrak O;X)}=\Vert {n_{\rm in}}\Vert_{X}.
\label{eq:Phi-homogeneous-bound}
\end{align}
For the second term, we apply \eqref{eq:expected-duhamel} with \(a=0\) to obtain
\begin{align}
&\left(
\mathbb E\sup_{0\le t\le T}
 \Vert {
\nu\int_0^t
S_\nu(t,s)\nabla\cdot F_u(s)\,d s
}\Vert_{X}^{q_0}
\right)^{1/q_0}\le
C_{p,q_0}(1+\sqrt T)\sqrt{\nu T}
\left(
\mathbb E\sup_{0\le s\le T}
\Vert {F_u(s)}\Vert_{X}^{q_0}
\right)^{1/q_0}.
\label{eq:Phi-Duhamel-bound}
\end{align}
Note, due to \eqref{eq:cutoff-global-bound}, that we have for every \(s\in[0,T]\) that $\Vert {F_u(s)}\Vert_{X}
\le
4C_pm^2$.
Consequently,
\begin{align}
\left(
\mathbb E\sup_{0\le s\le T}
\Vert {F_u(s)}\Vert_{X}^{q_0}
\right)^{1/q_0}
&\le
\left(
\mathbb E(4C_pm^2)^{q_0}
\right)^{1/q_0}=
4C_pm^2.
\label{eq:cutoff-flux-S-bound}
\end{align}
Substituting \eqref{eq:Phi-homogeneous-bound},
\eqref{eq:Phi-Duhamel-bound} and
\eqref{eq:cutoff-flux-S-bound} into
\eqref{eq:Phi-map-triangle}, we obtain for some constant $C_{p,q_0}>0$, that
\begin{equation}\label{eq:cutoff-map-bound}
\Vert {\Phi_m(u)}\Vert_{\mathcal S^{q_0}_{0,T}}
\le
\Vert {n_{\rm in}}\Vert_{X}
+
C_{p,q_0}m^2(1+\sqrt T)\sqrt{\nu T}.
\end{equation}
We next prove a contraction estimate for the mapping $\Phi_m$.  Note that,
\begin{align*}
\Phi_m(u)(t)-\Phi_m(v)(t)
&=
-\nu\int_0^t
S_\nu(t,s)\nabla\cdot
\bigl(F_u(s)-F_v(s)\bigr)\,d s.
\end{align*}
We again apply \eqref{eq:expected-duhamel} and obtain from the inequality \eqref{eq:Phi-Duhamel-bound} that 
\begin{align}
\Vert{
\Phi_m(u)-\Phi_m(v)
}\Vert_{\mathcal S^{q_0}_{0,T}}
&\le
C_{p,q_0}(1+\sqrt T)\sqrt{\nu T}
\left(
\mathbb E\sup_{0\le s\le T}
\Vert {F_u(s)-F_v(s)}\Vert_{X}^{q_0}
\right)^{1/q_0}.
\label{eq:Phi-difference-Duhamel}
\end{align}
Thanks to the Lipschitz estimate
\eqref{eq:cutoff-global-lipschitz},  we have for every \(s\in[0,T]\), that
\begin{align*}
\Vert{F_u(s)-F_v(s)}\Vert_{X}
&\le
C_pm
\sup_{0\le r\le s}
\Vert {u(r)-v(r)}\Vert_{X}\le
C_pm
\sup_{0\le r\le T}
\Vert {u(r)-v(r)}\Vert_{X}.
\end{align*}
Taking the supremum over \(s\in[0,T]\), raising to the \(q_0\)-th power,
taking expectations and then taking the \(q_0\)-th root, we obtain 
\begin{align}
&\left(
\mathbb E\sup_{0\le s\le T}
\Vert {F_u(s)-F_v(s)}\Vert_{X}^{q_0}
\right)^{1/q_0}\le
C_pm
\left(
\mathbb E\sup_{0\le r\le T}
\Vert{u(r)-v(r)}\Vert_{X}^{q_0}
\right)^{1/q_0}=
C_pm\Vert {u-v}\Vert_{\mathcal S^{q_0}_{0,T}}.
\label{eq:cutoff-flux-difference-S}
\end{align}
Substituting \eqref{eq:cutoff-flux-difference-S} into
\eqref{eq:Phi-difference-Duhamel} and absorbing \(C_p\) into
\(C_{p,q_0}\), we finally have that
\begin{equation}\label{eq:cutoff-contraction-estimate}
\Vert {
\Phi_m(u)-\Phi_m(v)
}\Vert_{\mathcal S^{q_0}_{0,T}}
\le
C_{p,q_0}m(1+\sqrt T)\sqrt{\nu T}
\Vert {u-v}\Vert_{\mathcal S^{q_0}_{0,T}}.
\end{equation}
Now we choose $T_m\le1$ so that $ 2C_{p,q_0}m\sqrt{\nu T_m}\le\tfrac12.$ Then, the Banach fixed point theorem gives us the existence of a unique solution to (\ref{eq:cutoff-mild-global}) on $[0,T_m]$.

Next, assume that $n^m(kT_m)$ for $k\geq 1$ is known, then we iteratively construct a global-in-time solution to the cutoff system (\ref{eq:cutoff-mild-global}).
  More precisely, we define,
\begin{align*}
 \Phi_m^{(k)}(u)(t)
 &=S_\nu(t,kT_m)n^m(kT_m)-\nu\int_{kT_m}^tS_\nu(t,s)\nabla\cdot[
 \theta_m(\Vert{u}\Vert_{X_s})u(s)\nabla(-\Delta)^{-1}u(s)]\,d s.
\end{align*}
Noting that 
$|\Vert {u}\Vert_{X_s}-\Vert{v}\Vert_{X_s}|
\le\sup_{kT_m\le r\le s}\Vert {u(r)-v(r)}\Vert_{X}$, we can similarly show that the mapping $\Phi_m^{(k)}$ is a contraction mapping and consequently, obtain a solution on $[kT_m,(k+1)T_m]$ by using the Banach fixed point theorem. By gluing these pieces together, we thus obtain the existence of a global solution to \eqref{eq:cutoff-mild-global}. 

Next, we will show that this global solution to (\ref{eq:cutoff-mild-global}) has finite moments of every order \(q\ge2\). For every \(q>2\), by using (\ref{eq:expected-homogeneous}) and \eqref{eq:expected-duhamel} we obtain
\begin{align*}
\Vert {n^m}\Vert_{\mathcal S^q_{0,T}}
&\le
\Vert {n_{\rm in}}\Vert_{X}
+
C_{p,q}(1+\sqrt T)\sqrt{\nu T}
\left(
\mathbb E\sup_{0\le s\le T}
\Vert {
\theta_m(\Vert {n^m}\Vert_{X_s})
n^m(s)\nabla(-\Delta)^{-1}n^m(s)
}\Vert_{X}^q
\right)^{1/q}.
\end{align*}
Additionally, thanks to  \eqref{eq:cutoff-global-bound}, we have for every $s$ that
\[
\Vert {
\theta_m(\Vert {n^m}\Vert_{X_s})
n^m(s)\nabla(-\Delta)^{-1}n^m(s)
}\Vert_{X}
\le
4C_pm^2.
\]
Hence, we conclude that
\[
\Vert {n^m}\Vert_{\mathcal S^q_{0,T}}
\le
\Vert {n_{\rm in}}\Vert_{X}
+
C_{p,q}m^2(1+\sqrt T)\sqrt{\nu T},
\qquad q>2.
\]
In other words, we have $n^m\in L^q(\mathfrak O;C([0,T];X))$
 \text{for every }$q\ge2$, as claimed.
\end{proof}

Next, we discuss the regularity of the global solution to the cutoff equation \eqref{eq:cutoff-mild-global} obtained in Lemma \ref{prop:global-cutoff}.
\begin{lemma}
\label{prop:pathwise-cutoff-regularity}
Suppose that $2<p<\infty$, $0<\nu\le1$ and
$n_{\rm in}\in H^1(\Omega)\cap X\cap L^\infty(\Omega)$ with $\|n_{\rm in}\|_X<m.$  Let $n^m$ be the global cutoff mild solution constructed  in Proposition \ref{prop:global-cutoff}.
Then, for every finite $T$ we have, almost surely, that
\begin{align}\label{eq:pathwise-cutoff-regularity}
 \begin{array}{ll}
& n^m\in C\bigl([0,T];H^1(\Omega)\cap L^1(\Omega)\bigr)
       \cap L^2\bigl(0,T;H^2(\Omega)\bigr),\qquad\sup_{0\le t\le T}\|n^m(t)\|_{L^\infty(\Omega)}<\infty.
 \end{array}
\end{align}
\end{lemma}

\begin{proof}
Fix a path on which the global cutoff mild solution constructed in Lemma \ref{prop:global-cutoff} is
continuous in $X$  We divide our proof into the following several steps.

\medskip\noindent

\textbf{Step 1: $L^\infty$- bound.}
By the definition of \(\theta_m\) and the \(X\)-contraction of
\(S_\nu(t,s)\) given in Lemma \ref{prop:pathwise-kernel}, we obtain that  $\sup_{t\ge0}\Vert {n^m(t)}\Vert_{X}\le2m.$
Moreover, by using interpolation inequalities and \eqref{eq:X-Poisson}, we further have 
\begin{equation*}
\Vert {n^m(t)}\Vert_{L^2(\Omega)}
\le 2m,
\qquad
\Vert {\nabla c^m(t)}\Vert_{L^\infty(\Omega)}
\le 2C_pm.
\end{equation*}
It then follows that
\begin{equation}\label{eq:pathwise-cutoff-flux-bound}
\left\|
\theta_m\bigl(\Vert {n^m}\Vert_{X_t}\bigr)
n^m(t)\nabla c^m(t)
\right\|_{L^p(\Omega)}
\le 4C_pm^2.
\end{equation}
Set $\gamma:=\frac12+\frac1p<1,$ then we apply  estimate
\eqref{eq:pathwise-Lp-Linfty-divergence} to (\ref{eq:cutoff-mild-global}), and then using
\eqref{eq:pathwise-cutoff-flux-bound} to get
\begin{align}
\Vert {n^m(t)}\Vert_{L^\infty(\Omega)}
&\le
\Vert {n_{\rm in}}\Vert_{L^\infty(\Omega)}
+
\nu\int_0^t
\left\|
S_\nu(t,s)\nabla\cdot
\left[
\theta_m\bigl(\Vert {n^m}\Vert_{X_s}\bigr)
n^m(s)\nabla c^m(s)
\right]
\right\|_{L^\infty(\Omega)}
\,d s
\notag\\
&\le
\Vert{n_{\rm in}}\Vert_{L^\infty(\Omega)}
+
C_{p,m,\Lambda_T}
\nu\int_0^t[\nu(t-s)]^{-\gamma}\,d s
\notag\\
&=
\Vert {n_{\rm in}}\Vert_{L^\infty(\Omega)}
+
\frac{C_{p,m,\Lambda_T}}{1-\gamma}
(\nu t)^{1-\gamma},
\label{eq:pathwise-cutoff-Linfty-bound}
\end{align}
where $\Lambda_T$ is as defined \eqref{eq:Gamma-definition}.
Therefore, $\sup_{0\le t\le T}\Vert {n^m(t)}\Vert_{L^\infty(\Omega)}<\infty$ almost surely.

\medskip\noindent
\textbf{Step 2: $L^2_tH_x^2$ estimate after transformation.}  Define  $g(t,x,y):=n^m(t,x+yW_t,y)$, $\widetilde c(t,x,y):=c^m(t,x+yW_t,y)$
and
\[
A_t:=
\begin{pmatrix}
1+W_t^2&-W_t\\
-W_t&1
\end{pmatrix}.
\]
Then, by using (\ref{eq:cutoff-mild-global}), we have for the fixed path, that $g$ satisfies
\begin{equation}\label{eq:pathwise-sheared-cutoff}
\partial_tg
=
\nu\nabla\cdot
\left[
A_t\left(
\nabla g-
\theta_m\bigl(\Vert {n^m}\Vert_{X_t}\bigr)
g\nabla\widetilde c
\right)
\right],
\qquad
-\nabla\cdot(A_t\nabla\widetilde c)=g,
\end{equation}
 in the mild sense.  On \([0,T]\),  as shown in Proposition \ref{prop:pathwise-kernel}, we have  $\Lambda_T^{-1}I\le A_t\le\Lambda_TI,$
where $\Lambda_T$ is defined in \eqref{eq:Gamma-definition}. Hence, similarly as shown in Step 1, we find
\begin{align}\label{gestimate}
\Vert {g(t)}\Vert_{L^2(\Omega)}\le2m,
\qquad\Vert {g(t)}\Vert_{L^\infty(\Omega)}
=
\Vert{n^m(t)}\Vert_{L^\infty(\Omega)}
\le C_{p,m,\nu,T,\Lambda_T}.
\end{align}
Invoking the Poisson estimate \eqref{eq:X-Poisson}, we  further  obtain 
\begin{equation}\label{eq:sheared-Poisson-gradient}
\Vert {\nabla\widetilde c(t)}\Vert_{L^\infty(\Omega)}
\le C_{p,m}\Lambda_T^{1/2}.
\end{equation}
Moreover, we use the \(L^2\) Fourier transform to   $-\nabla\cdot(A_t\nabla\widetilde c)=g$ and obtain
\begin{equation}\label{eq:sheared-Poisson-Hessian}
\Vert {\nabla^2\widetilde c(t)}\Vert_{L^2(\Omega)}
\le
\Lambda_T\Vert {g(t)}\Vert_{L^2(\Omega)}
\le2m\Lambda_T.
\end{equation}
Set $F:=
\theta_m\bigl(\Vert {n^m}\Vert_{X_t}\bigr)
A_tg\nabla\widetilde c.$
We collect  (\ref{gestimate}), (\ref{eq:sheared-Poisson-gradient}) and (\ref{eq:sheared-Poisson-Hessian}) to find that  $F\in L^2(0,T;L^2(\Omega;\mathbb R^2)).$  Next, from $\partial_tg=\nu\nabla\cdot(A_t\nabla g)-\nu\nabla\cdot F$  we see that
\begin{equation}\label{eq:pathwise-basic-energy}
\Vert {g(t)}\Vert_{L^2(\Omega)}^2
+
\frac{\nu}{\Lambda_T}
\int_0^t\Vert {\nabla g(s)}\Vert_{L^2(\Omega)}^2\,d s
\le
\Vert {n_{\rm in}}\Vert_{L^2(\Omega)}^2
+
\nu\Lambda_T
\int_0^t\Vert {F(s)}\Vert_{L^2(\Omega)}^2\,d s.
\end{equation}
This estimate implies that $g\in L^2(0,T;H^1(\Omega))$.

Next, for \(j=1,2\), we define
\[
D_h^jf(z):=\frac{f(z+he_j)-f(z)}{h},
\]
where $e_j$ is the unit vector at the $j$-th direction. We apply $D_h^j$S to the equation for $g$ which leads to
\[
\partial_tD_h^jg
=
\nu\nabla\cdot(A_t\nabla D_h^jg)
-\nu\nabla\cdot D_h^jF.
\]
Thus, we have
\[
D_h^jF
=
\theta_m\bigl(\Vert {n^m}\Vert_{X_t}\bigr)A_t
\left[
(D_h^jg)\nabla\widetilde c(\,\cdot+he_j)
+
gD_h^j\nabla\widetilde c
\right].
\]
Furthermore, by using Lemma \ref{lem:X-elliptic}, we have
\[
\Vert {D_h^j\nabla\widetilde c}\Vert_{L^2(\Omega)}
\le
\Vert {\partial_j\nabla\widetilde c}\Vert_{L^2(\Omega)}
\le
\Vert {\nabla^2\widetilde c}\Vert_{L^2(\Omega)}
\le2m\Lambda_T.
\]
Consequently,
\begin{equation}\label{eq:difference-flux-bound}
\Vert {D_h^jF(t)}\Vert_{L^2(\Omega)}
\le
C_{p,m,\nu,T,\Lambda_T}
\left(
1+\Vert {D_h^jg(t)}\Vert_{L^2(\Omega)}
\right),
\end{equation}
where the constant is independent of \(h\). Finally, we apply \eqref{eq:pathwise-basic-energy} to \(D_h^jg\), then use
\eqref{eq:difference-flux-bound} and $\Vert{D_h^jn_{\rm in}}\Vert_{L^2(\Omega)}
\le\Vert {\partial_jn_{\rm in}}\Vert_{L^2(\Omega)}$
to obtain
\begin{align}
\Vert {D_h^jg(t)}^2\Vert_{L^2(\Omega)}
&+
\frac{\nu}{\Lambda_T}
\int_0^t
\Vert {\nabla D_h^jg(s)}\Vert_{L^2(\Omega)}^2\,d s
\notag\\
&\le
\Vert {\partial_jn_{\rm in}}\Vert_{L^2(\Omega)}^2
+
C_{p,m,\nu,T,\Lambda_T}
\int_0^t
\left(
1+\Vert {D_h^jg(s)}\Vert_{L^2(\Omega)}^2
\right)\,d s.
\label{eq:pathwise-difference-energy}
\end{align}
Then an application of Gr\"onwall's inequality and difference-quotient approximation 
then gives that $g\in
L^\infty(0,T;H^1(\Omega))
\cap
L^2(0,T;H^2(\Omega)).$

\medskip\noindent
\textbf{Step 3: $L^2_tH_x^2$ estimate.}
Recall that  $\nabla\cdot F
=
\theta_m\bigl(\Vert {n^m}\Vert_{X_t}\bigr)
\left(
A_t\nabla g\cdot\nabla\widetilde c-g^2
\right).$
Steps 1 and 2 imply that  $\nabla\cdot F\in L^2(0,T;L^2).$ Combining this with \(g\in L^2(0,T;H^2)\) and
\eqref{eq:pathwise-sheared-cutoff} therefore yields $\partial_tg\in L^2(0,T;L^2).$
Hence $g\in
H^1(0,T;L^2)\cap L^2(0,T;H^2)
\subset C([0,T];H^1)$ by the interpolation inequality.  Finally, since  $n^m(t,x,y)=g(t,x-yW_t,y)$  and \(W\) is continuous, we have $n^m\in
C([0,T];H^1(\Omega))
\cap
L^2(0,T;H^2(\Omega)).$



\end{proof}

\subsection{Existence and pathwise uniqueness of local mild solution to the original system}\label{sec:local}

In  this section, we prove Theorem \ref{thm:local-complete} by constructing a maximal mild solution to \eqref{eq:KS-local} (see Definition \ref{def:maximal-local}). This is done in two steps: we first establish the existence of a local mild solution as defined in Definition \ref{def:local-mild} in Lemma \ref{prop:cutoff-to-local} below and then establish pathwise uniqueness of the local mild solution in Lemma \ref{thm:pathwise-uniqueness}. 

To that end, we recall that we have constructed a global mild solution $n^m$ to the approximation system \eqref{eq:cutoff-mild-global}. Now, we choose an integer $m_0>\Vert {n_{\rm in}}\Vert_{X}$ and let
\begin{equation}\label{eq:exit-time}
 \tau_m=\inf\{t>0:\Vert {n^m}\Vert_{X_t}\ge m\},\qquad m\ge m_0.
\end{equation}
\begin{lemma} \label{prop:cutoff-to-local}
Let $n^m$ be the global mild solution to \eqref{eq:cutoff-mild-global} constructed in Lemma \ref{prop:global-cutoff}. Then, $\tau_m$ given in  (\ref{eq:exit-time}) is an $\{\mathcal{F}_t\}_{t\geq0}-$stopping time satisfying $\mathbb P(\tau_m>0)=1$, and
$(n^m,\tau_m)$ is a local mild solution to \eqref{eq:KS-local} in the sense of Definition \ref{def:local-mild}.
\end{lemma}
\begin{proof}
We begin by noting that fact that the norm $\Vert\cdot \Vert_{X_t}$, defined in \eqref{eq:history-norm} is $\{\mathcal{F}_t\}_{t\geq0}-$adapted, continuous, and non-decreasing. This means that
$\{\tau_m\le t\}=\{\Vert {n^m}\Vert_{X_t}\ge m\}\in\mathcal F_t$. Since $\|n^m\|_{X_t}<m$, continuity yields $\tau_m>0$. By the definition of \(\tau_m\), the cutoff is identically equal to one
on \([0,\tau_m)\). Choose \(j\in\mathbb N\) sufficiently large so that $j^{-1}<m- \Vert {n_{\rm in}\Vert_{X}},$
and define
\[
 \rho_{m,j}=\inf\{t>0:\Vert {n^m}\Vert_{X_t}\ge m-j^{-1}\}\wedge j.
\]
Then we have $\rho_{m,j}\uparrow\tau_m$, with strict inequality on
$\{\tau_m<\infty\}$ and
\[
 \mathbb E\sup_{t\le T}\Vert {n^m(t\wedge\rho_{m,j})}\Vert_{X}^q
 \le\mathbb E\sup_{t\le T}\Vert {n^m(t)}\Vert_{X}^q<\infty.
\]
\end{proof}
Next, we establish the uniqueness of the local solution constructed in the previous lemma.
 \begin{lemma}
\label{thm:pathwise-uniqueness}
Let \((n^1,\tau_1)\) and \((n^2,\tau_2)\) be two local mild solutions to (\ref{eq:KS-local})
defined on the same probability space $(\mathfrak O,\mathcal F,(\mathcal F_t)_{t\ge0},\mathbb P)$
and satisfying $n^1(0)=n^2(0)=n_{\rm in}$ $\mathbb P\text{-almost surely}.$  Set $\tau:=\tau_1\wedge\tau_2.$
Then
\begin{equation}\label{eq:pathwise-uniqueness-statement}
\mathbb P\left(
n^1(t)=n^2(t)\ \text{in }X
\text{ for every }0\le t<\tau
\right)=1.
\end{equation}
Equivalently, there exists a set
\(\mathfrak O_0\in\mathcal F\) with
\(\mathbb P(\mathfrak O_0)=1\) such that for every
\(\omega\in\mathfrak O_0\),
\[
n^1(t,\omega)=n^2(t,\omega)
\qquad
\text{for every }
0\le t<\tau_1(\omega)\wedge\tau_2(\omega).
\]
\end{lemma}

\begin{proof}
Fix $q>2$. For $i=1,2$ and
$R>\Vert {n_{\rm in}}\Vert_{X}$, define 
\begin{equation}\label{eq:uniqueness-exit-times}
 \tau_i^R
 =
 \inf\left\{
 t\in[0,\tau_i):
 \Vert {n^i}\Vert_{X_t}\ge R
 \right\}
 \wedge\tau_i,
 \qquad
 \tau^R=\tau_1^R\wedge\tau_2^R,
\end{equation}
where we recall from \eqref{eq:history-norm} that $ \Vert {n^i}\Vert_{X_t}
 :=
 \sup_{0\le s\le t}\Vert {n^i(s)}\Vert_{X}.$
The continuity and adaptedness of the local mild solutions shown in Definition \ref{def:local-mild} imply that
$\tau_i^R$ and $\tau^R$ are stopping times.

Moreover, we claim
\begin{equation}\label{eq:uniqueness-exit-limit}
 \tau^R\uparrow\tau
 \qquad
 \mathbb P\text{-almost surely as }R\to\infty.
\end{equation}
Indeed, $\tau^R\le\tau$ for every $R$.   Since $n^1,n^2\in C([0,t];X),$
we have
\[
 \sup_{0\le s\le t}\Vert {n^1(s)}\Vert_{X}
 +
 \sup_{0\le s\le t}\Vert {n^2(s)}\Vert_{X}
 <\infty.
\]
Consequently, for every sufficiently large $R$, $\Vert {n^1}\Vert_{X_t}<R$ and  $\Vert {n^2}\Vert_{X_t}<R.$  Hence, $t<\tau^R$. Since this holds for every $t<\tau$,
\eqref{eq:uniqueness-exit-limit} follows.

Next, we discuss the uniqueness.  Set $u=n^1-n^2,$
 $c^i=(-\Delta)^{-1}n^i,$ for  $i=1,2$, then we  define
\begin{equation}\label{eq:uniqueness-flux}
 G
 =
 n^1\nabla c^1-n^2\nabla c^2.
\end{equation}
Using $G
 =
 (n^1-n^2)\nabla c^1
 +
 n^2\nabla(c^1-c^2),$
we obtain, for $r=1,p$,
\begin{align*}
\Vert {G}\Vert_{L^r(\Omega)}
 &\le
\Vert {n^1-n^2}\Vert_{L^r(\Omega)}
\Vert {\nabla c^1}\Vert_{L^\infty}
 +
\Vert {n^2}\Vert_{L^r(\Omega)}
 \Vert {\nabla(c^1-c^2)} \Vert_{L^\infty(\Omega)}                                     \\
 &\le
 C_p\Vert {n^1-n^2}\Vert_{L^r(\Omega)}\Vert {n^1}\Vert_{X}
 +
 C_p\Vert {n^2}\Vert_{L^r(\Omega)}\Vert {n^1-n^2}\Vert_{X}.
\end{align*}
We adding the estimates for $r=1$ and $r=p$ to get
\begin{equation}\label{eq:uniqueness-flux-difference}
\Vert {G}\Vert_{X}
 \le
 C_p\bigl(
\Vert {n^1}\Vert_{X}
 +
 \Vert {n^2}\Vert_{X}
 \bigr)
\Vert {u}\Vert_{X}.
\end{equation}
For $0\le s<\tau^R$, the definition of $\tau^R$ gives $\Vert {n^1(s)}\Vert_{X}<R$ and $
\Vert {n^2(s)}\Vert_{X}<R.$  Therefore, by using \eqref{eq:uniqueness-flux-difference}, one finds
\begin{equation}\label{eq:localized-flux-difference}
 \Vert {G(s)}\Vert_{X}
 \le
 2C_pR\Vert {u(s)}\Vert_{X},
 \qquad
 0\le s<\tau^R.
\end{equation}

Define
\begin{equation}\label{eq:stopped-difference-flux}
 G^R(s)
 =
 \begin{cases}
  G(s),&0\le s<\tau^R,\\
  0,&s\ge\tau^R.
 \end{cases}
\end{equation}
Equivalently,
 $G^R(s)=\mathbf 1_{\{s<\tau^R\}}G(s).$
Since $\tau^R$ is a stopping time and $G$ is adapted, we have $G^R$ is adapted.

For every $0\le t<\tau^R$, we  subtract the two mild equations to get
\begin{align}
 u(t)
 &=
 -\nu\int_0^t
 S_\nu(t,s)\nabla\cdot
 \left(
 n^1(s)\nabla c^1(s)
 -
 n^2(s)\nabla c^2(s)
 \right)\,d s                                      \notag\\
 &=
 -\nu\int_0^t
 S_\nu(t,s)\nabla\cdot G^R(s)\,d s.
 \label{eq:stopped-difference-mild}
\end{align}
Let $h\in(0,1]$, then by using Proposition \ref{prop:expected-maximal}, we have  from
\eqref{eq:stopped-difference-mild}
and \eqref{eq:localized-flux-difference} that
\begin{align}
 \left(
 \mathbb E\sup_{0\le t<h\wedge\tau^R}
\Vert {u(t)}\Vert_X^q
 \right)^{1/q}                                                  &\le
 \left(
 \mathbb E\sup_{0\le t\le h}
\Vert {
 \nu\int_0^t
 S_\nu(t,s)\nabla\cdot G^R(s)\,d s
 }\Vert_X^q
 \right)^{1/q}                                                   \notag\\
 &\le
 C_{p,q}(1+\sqrt h)\sqrt{\nu h}
 \left(
 \mathbb E\sup_{0\le s\le h}
\Vert {G^R(s)}\Vert_X^q
 \right)^{1/q}                                                   \notag\\
 &\le
 2C_pC_{p,q}R(1+\sqrt h)\sqrt{\nu h}
 \left(
 \mathbb E\sup_{0\le s<h\wedge\tau^R}
\Vert {u(s)}\Vert_X^q
 \right)^{1/q}.
 \label{eq:uniqueness-short-interval-estimate}
\end{align}
Since for $s<\tau^R$, $\Vert {u(s)}\Vert_X
 \le
 \Vert {n^1(s)}\Vert_X
 +
\Vert {n^2(s)}\Vert_X
 <2R.$
Choose a deterministic number $h_R\in(0,1]$ sufficiently small that
\begin{equation}\label{eq:uniqueness-time-choice}
 2C_pC_{p,q}R
 (1+\sqrt{h_R})\sqrt{\nu h_R}
 \le\frac12.
\end{equation}
Taking $h=h_R$ in
\eqref{eq:uniqueness-short-interval-estimate} gives
\begin{align*}
 \left(
 \mathbb E\sup_{0\le t<h_R\wedge\tau^R}
\Vert {u(t)}\Vert_X^q
 \right)^{1/q}
 &\le
 \frac12
 \left(
 \mathbb E\sup_{0\le t<h_R\wedge\tau^R}
 \Vert {u(t)}\Vert_X^q
 \right)^{1/q}.
\end{align*}
Since this quantity is finite and nonnegative, it must vanish. Hence $\mathbb E\sup_{0\le t<h_R\wedge\tau^R}
 \Vert {u(t)}\Vert_X^q
 =0.$
 It follows that
\[
 u(t)=0\quad\text{in }X
 \quad\text{for every }0\le t<h_R\wedge\tau^R,
 \qquad \mathbb P\text{-almost surely}.
\]

We now repeat the argument on consecutive deterministic intervals.
Suppose that, for some integer $k\ge1$,
\begin{equation}\label{eq:uniqueness-induction-hypothesis}
 u(t)=0\quad\text{in }X
 \quad\text{for every }0\le t<kh_R\wedge\tau^R,
 \qquad \mathbb P\text{-almost surely}.
\end{equation}
Set $a=kh_R$, $b=(k+1)h_R$, $E_k=\{a<\tau^R\}.$  Since $\tau^R$ is a stopping time, $E_k\in\mathcal F_a.$
On $E_k$, continuity of $u$ and
\eqref{eq:uniqueness-induction-hypothesis} imply $u(a)=0.$  Moreover, we have on $E_k$, for $a\le t<b\wedge\tau^R$,
\begin{align}
 u(t)
 &=
 S_\nu(t,a)u(a)
 -
 \nu\int_a^t
 S_\nu(t,s)\nabla\cdot G(s)\,d s                         =
 -\nu\int_a^t
 S_\nu(t,s)\nabla\cdot G(s)\,d s.
 \label{eq:uniqueness-restart-before-stopping}
\end{align}
Define $G_k^R(s)
 =
 \mathbf 1_{E_k}
 \mathbf 1_{\{s<\tau^R\}}
 G(s)$, $a\le s\le b.$
Since  $E_k\in\mathcal F_a\subseteq\mathcal F_s$
and $\{s<\tau^R\}\in\mathcal F_s,$
the random variable $G_k^R(s)$ is $\mathcal F_s$-measurable. Therefore,
$G_k^R$ is  adapted.  Multiplying \eqref{eq:uniqueness-restart-before-stopping} by
$\mathbf 1_{E_k}$ gives
\begin{equation}\label{eq:uniqueness-restarted-mild}
 \mathbf 1_{E_k}u(t)
 =
 -\nu\int_a^t
 S_\nu(t,s)\nabla\cdot G_k^R(s)\,d s,
 \qquad
 a\le t<b\wedge\tau^R.
\end{equation}
We  further invoke Proposition \ref{prop:expected-maximal} to obtain
\begin{align}
 &\left(
 \mathbb E\left[
 \mathbf 1_{E_k}
 \sup_{a\le t<b\wedge\tau^R}
\Vert {u(t)}\Vert_X^q
 \right]
 \right)^{1/q}                                                  \le
 C_{p,q}(1+\sqrt{b-a})\sqrt{\nu(b-a)}
 \left(
 \mathbb E\sup_{a\le s\le b}
 \Vert {G_k^R(s)}\Vert_X^q
 \right)^{1/q}.                                     \notag
\end{align}
For $a\le s<b\wedge\tau^R$, estimate
\eqref{eq:localized-flux-difference} yields $\Vert {G_k^R(s)}\Vert_X
 \le
 2C_pR\mathbf 1_{E_k}\Vert {u(s)}\Vert_X.$
Consequently,
\begin{align}
 \left(
 \mathbb E\left[
 \mathbf 1_{E_k}
 \sup_{a\le t<b\wedge\tau^R}
 \Vert {u(t)}\Vert_X^q
 \right]
 \right)^{1/q}                                                  &\le
 2C_pC_{p,q}R
 (1+\sqrt{b-a})\sqrt{\nu(b-a)}
 \left(
 \mathbb E\left[
 \mathbf 1_{E_k}
 \sup_{a\le s<b\wedge\tau^R}
\Vert {u(s)}\Vert_X^q
 \right]
 \right)^{1/q}                                                   \notag\\
 &\quad\le
 \frac12
 \left(
 \mathbb E\left[
 \mathbf 1_{E_k}
 \sup_{a\le s<b\wedge\tau^R}
\Vert {u(s)}\Vert_X^q
 \right]
 \right)^{1/q},
 \label{eq:uniqueness-induction-estimate}
\end{align}
where we used $b-a=h_R$ and \eqref{eq:uniqueness-time-choice}. Therefore,
\[
 \mathbb E\left[
 \mathbf 1_{E_k}
 \sup_{a\le t<b\wedge\tau^R}\Vert {u(t)}\Vert_X^q
 \right]=0.
\]
Thus, $u(t)=0$
\text{for every }
 $a\le t<b\wedge\tau^R$
\text{ on }$E_k$,
$\mathbb P$\text{-almost surely}.
On $E_k^c=\{\tau^R\le a\}$, the intersection $[a,b)\cap[0,\tau^R)$ is empty. We have therefore proved
\eqref{eq:uniqueness-induction-hypothesis} with $k+1$ in place of $k$.

By induction, for every deterministic $T<\infty$ and every
$R>\Vert {n_{\rm in}}\Vert_X$,
\begin{equation}\label{eq:uniqueness-before-R-exit}
 \mathbb P\left(
 u(t)=0\text{ in }X
 \text{ for every }0\le t<T\wedge\tau^R
 \right)=1.
\end{equation}
Fix $T<\infty$ and let $R\to\infty$ through the positive integers.
Since $\tau^R\uparrow\tau$
 $\mathbb P$\text{-almost surely},  we have,
\[
 \sup_{0\le t<T\wedge\tau^R}\Vert {u(t)}\Vert_X^q
 \rightarrow
 \sup_{0\le t<T\wedge\tau}\Vert {u(t)}\Vert_X^q,\qquad\text{almost surely}.
\]
Therefore, by the monotone convergence theorem and
\eqref{eq:uniqueness-before-R-exit}, we obtain
\begin{align*}
 \mathbb E\sup_{0\le t<T\wedge\tau}\Vert {u(t)}\Vert_X^q
 &=
 \lim_{R\to\infty}
 \mathbb E\sup_{0\le t<T\wedge\tau^R}\Vert{u(t)}\Vert_X^q =0.
\end{align*}
Consequently, $u(t)=0$\text{ in }$X$
\text{ for every }$0\le t<T\wedge\tau$, $\mathbb P$\text{-almost surely}.
Taking $T$ through the positive integers yields
\[
 \mathbb P\left(
 n^1(t)=n^2(t)\text{ in }X
 \text{ for every }0\le t<\tau
 \right)=1.
\]
This proves \eqref{eq:pathwise-uniqueness-statement}.
\end{proof}

We are now in position to show the existence of the maximal local mild solution to (\ref{eq:KS-local}) as defined in Definition \ref{def:maximal-local} and finish the proof of Theorem \ref{thm:local-complete}.  

\begin{proof}[Proof of Theorem \ref{thm:local-complete}]
Let $(n^m,\tau_m)$ be a local mild solution to \eqref{eq:KS-local} obtained in Lemma \ref{prop:cutoff-to-local}, where $\tau_m$ is the sequence of stopping times defined in
\eqref{eq:exit-time}. 
Define $\tau_*=\lim_{m\to\infty}\tau_m$
and, for $0\le t<\tau_*$, set
\begin{equation}\label{eq:maximal-solution-definition}
 n(t)=n^m(t\wedge\tau_m).
\end{equation}
Then we claim that $(n,\tau_*)$ is the desired unique maximal local mild solution of \eqref{eq:KS-local}.

First note that 
\begin{equation}\label{eq:maximal-cutoff-consistency}
 \tau_m\le\tau_{m+1}, \qquad
 \mathbb P\text{-almost surely}.
\end{equation}
Hence, by using  Lemma~\ref{thm:pathwise-uniqueness}, we have
\begin{equation}\label{eq:maximal-solution-consistency}
 n^{m+1}(t)=n^m(t),
 \qquad\text{ for } 0\le t<\tau_m, \qquad
 \mathbb P\text{-almost surely}.
\end{equation}

We next show that $(\tau_m)$ is an announcing sequence for $\tau_*$.
Since $\Vert {n_{\rm in}}\Vert_X<m$
and
 $t\mapsto\Vert {n^m}\Vert_{X_t}$
is continuous (see \eqref{eq:history-norm}), we have, on the event $\{\tau_m<\infty\}$, that
\begin{equation}\label{eq:maximal-exit-value}
\Vert {n^m}\Vert_{X_{\tau_m}}=m.
\end{equation}
We claim that
\begin{equation}\label{eq:maximal-strict-exit-order}
 \tau_m<\tau_{m+1}
 \qquad\text{on }\{\tau_m<\infty\}.
\end{equation}
Indeed, by continuity, the equality
\eqref{eq:maximal-solution-consistency} extends to $t=\tau_m$. Hence
 $\Vert {n^{m+1}}\Vert_{X_{\tau_m}}
 =
\Vert {n^m}\Vert_{X_{\tau_m}}
 =
 m
 <
 m+1.$
By the definition of $\tau_{m+1}$, we must have
\eqref{eq:maximal-strict-exit-order}.
On the event $\{\tau_*<\infty\}$, 
\eqref{eq:maximal-strict-exit-order} thus implies that
\begin{equation}\label{eq:maximal-announcing}
 \tau_m<\tau_*
 \quad\text{for every }m,
 \qquad
 \tau_m\uparrow\tau_*
 \quad\text{on }\{\tau_*<\infty\}.
\end{equation}
Therefore, $(\tau_m)_m$ announces $\tau_*$. Moreover, for every
$T<\infty$ and $q\ge2$,
\[
 \mathbb E\sup_{0\le t\le T}
\Vert {n(t\wedge\tau_m)}\Vert_X^q
 \le
 \mathbb E\sup_{0\le t\le T}
\Vert {n^m(t)}\Vert_{X}^q
 <\infty.
\]
It follows that $(n,\tau_*)$ is a local mild solution.
Furthermore, since $\tau_m\uparrow\tau_*$, we have
\begin{align}
 \limsup_{t\uparrow\tau_*}\Vert {n}\Vert_{X_t}
 &\ge
 \limsup_{m\to\infty}\Vert {n}\Vert_{X_{\tau_m}} 
 =\lim_{m\to\infty}m=\infty
 \qquad\text{on }\{\tau_*<\infty\}.
 \label{eq:maximal-blowup}
\end{align}
Hence, $(n,\tau_*)$ is in fact a maximal local mild solution (see
Definition~\ref{def:maximal-local}).

Finally, suppose that
$(\widetilde n,\widetilde\tau)$ is another maximal local mild solution with respect to the same stochastic basis and
with the same initial value. Lemma \ref{thm:pathwise-uniqueness} yields
\[
 n(t)=\widetilde n(t)
 \qquad
 \text{for every }
 0\le t<\tau_*\wedge\widetilde\tau,
 \quad \mathbb P\text{-almost surely}.
\]
Assume, without loss of generality, that $\tau_*<\widetilde\tau$ on a set of positive probability. Then
since $\widetilde n$ is continuous, it is bounded in $X$ on any compact
interval containing $[0,\tau_*]$. This contradicts
\eqref{eq:maximal-blowup}. 
Therefore,
 $\widetilde\tau=\tau_*$, $\mathbb P$\text{-almost surely}.
Thus, the maximal local mild solution is unique.

Finally, the required
regularity of this solution is obtained by using  Lemma  \ref{prop:pathwise-cutoff-regularity} and the local compactness. This finishes the proof of Theorem \ref{thm:local-complete}.
\end{proof}

\section{Global-in-time existence}\label{sec:global}
Having established the existence of a maximal local mild solution in Theorem \ref{thm:local-complete}, we next prove its global existence by choosing an $A$ in  \eqref{eq:KS-local} which is sufficiently large. To establish this result, we will appeal to dissipation enhancing mechanism of the stochastic shear flow. To that effect, we decompose the system \eqref{eq:KS-local} into its Fourier modes.

Recall the notation $f_0(y)=\frac1{2\pi}\int_{-\pi}^{\pi}f(x,y)\,d x.$
Define 
$$u=\neqm=n-n_0,\qquad\psi=\ceqm=c-c_0.$$  Averaging the first equation in \eqref{eq:KS-local} in $x$ gives us that the following equation for $n_0$, the zeroth mode of $n$,  holds in distributional sense:
\begin{equation}\label{eq:zero-equation}
 \partial_tn_0=\nu n_0''-\nu(n_0c_0')'
                       -\nu\partial_y(u\partial_y\psi)_0.
\end{equation}
Since $n_0$ is independent of $x$, the shear and thus the enhanced dissipation mechanism does not act on the zeroth-mode. By contrast, for $u=n_{\neq}$, we have $u_0=0$ and $
\psi_0=0.$ Hence, we consider the dynamics of $u$ described by the following equations that are obtained by subtracting (\ref{eq:zero-equation}) from \eqref{eq:KS-local},

\begin{equation}\label{eq:nonzero-equation}
du+y\partial_xu\circ dW_t
       =\nu\Delta u\, dt-(\Lin+\Non)\,\mathrm dt,
\end{equation}
where
\begin{align}
\Lin&:=\nu\left(c_0'\partial_yu+n_0'\partial_y\psi-2n_0u\right), \label{eq:L-definition}\\
 \Non&:=\nu P_{\neq}\left(\nabla u\cdot\nabla\psi-u^2\right),
          \label{eq:N-definition}
\end{align}
with $P_{\neq}$ defined in \eqref{nonzero-mode-projection}.

In the following two subsections, we first analyze the linear part of \eqref{eq:nonzero-equation} and establish the corresponding enhanced dissipation estimate, and then turn to the inhomogeneous equations and treat the nonlinear terms perturbatively.

\subsection{Blockwise contraction for the passive scalar equation}\label{sec:linear}
{In this section, we focus on 
 the passive scalar equation (\ref{eq:passive-local}).
 We divide the time interval into small blocks and prove that sufficiently many blocks produce extra dissipation which is needed to obtain a quantitative decay estimate for the nonzero Fourier modes and, subsequently, to control the nonlinear evolution. This argument of dividing the time interval into blocks is necessary to treat the stochastic term since, unlike the deterministic case, stochastic shear does not produce a fixed amount of mixing over any time interval.
 
First, we observe the following property of the solution operator $S_\nu$ to the linear equation \eqref{eq:passive-local} which is responsible for producing the necessary decay of solutions.}

We set time increment $\tau_\nu=\nu^{-1/2}$ and define $t_j=j\tau_\nu$.
\begin{proposition}\label{prop:decay}
    For every $h\in L^2(\Omega)$ with $h_0=0$, the following estimate holds pathwise, 
\begin{equation}\label{eq:pathwise-linear-block}
 \|S_\nu(t_{j+1},t_j)h\|_{L^2(\Omega)}\le e^{-Z_j^\nu}\|h\|_{L^2(\Omega)},
\end{equation}
where 
\begin{equation}\label{eq:block-variance}
 Z_j^\nu:=\nu\inf_{a\in\mathbb R}
       \int_{t_j}^{t_{j+1}}(W_r-W_{t_j}-a)^2\,d r.
\end{equation}
\end{proposition}
\begin{proof}
    We first consider $h\in C_c^\infty(\Omega)$ with $h_0=0$.
By using a density argument we can extend the following results to $h\in L^2(\Omega)$ with $h_0=0$; these details are given at the end.  

We fix $\nu>0$ and a block
index $j$. Recall that $\tau_\nu=\nu^{-1/2}$, $t_j=j\tau_\nu$ and $t_{j+1}=t_j+\tau_\nu.$  Define $B_r=W_{t_j+r}-W_{t_j}$ for any
$0\le r\le\tau_\nu$ and let $f(x,y,r)
=\bigl(S_\nu(t_j+r,t_j)h\bigr)(x,y).$ Then, by the definition of $S_\nu$ we have,
\[
d f+y\partial_xf\circ d B_r
=
\nu\Delta f\,d r,
\qquad
f(0)=h.
\]
Define  $H(r,x,y)=f(r,x+yB_r,y).$
Since $B_0=0$, we have $H(0)=h$.
By using the Stratonovich form of the It\^o formula, one obtains
\begin{align}
\begin{cases}\label{H-equation1}
\displaystyle
\partial_rH
=
\nu\left(
(1+B_r^2)\partial_x^2H
-2B_r\partial_{xy}H
+\partial_y^2H
\right),\\[1mm]
H(0)=h.
\end{cases}
\end{align}
As shown in \eqref{Fourier_mode_x_f} and \eqref{Fourier_transform_y}, we define 
\[
H_k(r,y)
=
\frac1{2\pi}\int_{\mathbb T} H(r,x,y)e^{-ikx}\,d x, \text{ and }
\widehat H_k(r,\eta)
=
\frac1{\sqrt{2\pi}}
\int_{\mathbb R }H_k(r,y)e^{-i\eta y}\,d y,
~~\eta\in\mathbb R,
\]
where $k\in\mathbb Z$,
then we apply the Fourier transform to (\ref{H-equation1}) and obtain
\begin{align*}
\partial_r\widehat H_k(r,\eta)
&=
\nu\left(
-(1+B_r^2)k^2+2B_rk\eta-\eta^2
\right)\widehat H_k(r,\eta)\\
&=
-\nu\left(
k^2+(\eta-kB_r)^2
\right)\widehat H_k(r,\eta),
\end{align*}
with $\widehat H_k(0,\eta)=\widehat h_k(\eta).$
Solving it, we have
\[
\widehat H_k(\tau_\nu,\eta)
=
\exp\left[
-\nu k^2\tau_\nu
-\nu\int_0^{\tau_\nu}(\eta-kB_r)^2\,d r
\right]\widehat h_k(\eta).
\]
Since $B_r=W_{t_j+r}-W_{t_j}$ and 
$s=t_j+r$, we further obtain 
\begin{align}\label{Fourier_representation}
\widehat H_k(\tau_\nu,\eta)
=
\exp\left[
-\nu k^2\tau_\nu
-\nu\int_{t_j}^{t_{j+1}}
\bigl(\eta-k(W_s-W_{t_j})\bigr)^2\, d s
\right]\widehat h_k(\eta).
\end{align}
By using change of variables, we obtain from (\ref{eq:block-variance}) that  
\begin{align*}
Z_j^\nu
&=
\nu\inf_{a\in\mathbb R}
\int_{t_j}^{t_{j+1}}
(W_s-W_{t_j}-a)^2\,d s=
\nu\inf_{a\in\mathbb R}
\int_0^{\tau_\nu}(B_r-a)^2\,d r.
\end{align*}
In addition, for every $k\in\mathbb Z\setminus\{0\}$ and $\eta\in\mathbb R$,
\begin{align*}
\nu\int_0^{\tau_\nu}(\eta-kB_r)^2\,d r
&=
\nu k^2
\int_0^{\tau_\nu}
\left(B_r-\frac{\eta}{k}\right)^2\,d r\\
&\ge
\nu k^2\inf_{a\in\mathbb R}
\int_0^{\tau_\nu}(B_r-a)^2\,d r=
k^2Z_j^\nu\ge Z_j^\nu.
\end{align*}

Substituting this into (\ref{Fourier_representation}) indicates that 
\begin{align}\label{preceding-multiplier-bound}
|\widehat H_k(\tau_\nu,\eta)|
&\le
e^{-\nu k^2\tau_\nu}
e^{-k^2Z_j^\nu}
|\widehat h_k(\eta)|\le
e^{-Z_j^\nu}|\widehat h_k(\eta)|,
\end{align}
where we used $e^{-\nu k^2\tau_\nu}\le1$.
 
Since $h_0=0$, we have $\widehat h_0=0$.
The Fourier representation therefore implies that
$\widehat H_0(r,\eta)=0$ for every $r$.
Using Plancherel's identity and (\ref{preceding-multiplier-bound}),
we obtain
\begin{align*}
\|H(\tau_\nu)\|_{L^2(\Omega)}^2
&=
2\pi\sum_{k\ne0}
\int_{\mathbb R}|\widehat H_k(\tau_\nu,\eta)|^2\,d\eta\le
2\pi\sum_{k\ne0}
\int_{\mathbb R}
e^{-2Z_j^\nu}|\widehat h_k(\eta)|^2\,d\eta\\
&=
e^{-2Z_j^\nu}
\left(
2\pi\sum_{k\ne0}
\int_{\mathbb R}|\widehat h_k(\eta)|^2\,d\eta
\right)=
e^{-2Z_j^\nu}\|h\|_{L^2(\Omega)}^2.
\end{align*}
We now transform to the original coordinates.  Noting that $f(r,x,y)=H(r,x-yB_r,y)$, thus $\bigl(S_\nu(t_{j+1},t_j)h\bigr)(x,y)
=
H(\tau_\nu,x-yB_{\tau_\nu},y).$
  Consequently,
\begin{align*}
\|S_\nu(t_{j+1},t_j)h\|_{L^2(\Omega)}^2
&=
\int_{\mathbb R}\int_{\mathbb T}
|H(\tau_\nu,x-yB_{\tau_\nu},y)|^2\,d xd y\\
&=
\int_{\mathbb R}\int_{\mathbb T}
|H(\tau_\nu,x,y)|^2\,d xd y\\
&=
\|H(\tau_\nu)\|_{L^2(\Omega)}^2\le
e^{-2Z_j^\nu}\|h\|_{L^2(\Omega)}^2.
\end{align*}
It follows that  $\|S_\nu(t_{j+1},t_j)h\|_{L^2(\Omega)}
\le e^{-Z_j^\nu}\|h\|_{L^2(\Omega)}.$ 

We now extend the argument to the case of  $h\in L^2(\Omega)$ with $h_0=0$.
Choose $\varphi_m\in C_c^\infty(\Omega)$ such that
$\varphi_m\to h$ in $L^2$, and set $h_m=\varphi_{m,\not=}.$
Then $h_m\in C_c^\infty(\Omega)$, $h_{m,0}=0$, and
\[
\|h_m-h\|_{L^2}
=
\|(\varphi_m-h)_{\not=}\|_{L^2(\Omega)}
\le\|\varphi_m-h\|_{L^2(\Omega)}
\rightarrow0.
\]
Using the $L^2$ contraction of $S_\nu$ shown in Proposition \ref{prop:pathwise-kernel},
together with the estimate for $h_m$ shown above, we obtain
\begin{align*}
\|S_\nu(t_{j+1},t_j)h\|_{L^2(\Omega)}
&\le
\|S_\nu(t_{j+1},t_j)(h-h_m)\|_{L^2(\Omega)}
+\|S_\nu(t_{j+1},t_j)h_m\|_{L^2(\Omega)}\\
&\le
\|h-h_m\|_{L^2(\Omega)}
+e^{-Z_j^\nu}\|h_m\|_{L^2(\Omega)}.
\end{align*}
Letting $m\to\infty$, we further obtain 
\[
\|S_\nu(t_{j+1},t_j)h\|_{L^2(\Omega)}
\le e^{-Z_j^\nu}\|h\|_{L^2(\Omega)}
\]
for every $h\in L^2(\Omega)$ with $h_0=0$.
This finishes the  proof of \eqref{eq:pathwise-linear-block}.
\end{proof}

Thus, Proposition \ref{prop:decay} tells us that 
\(Z_j^\nu\) provides dissipation generated by the time variation of the shear during the 
$j^{th}$ time block. Now, we prove the following lemma which tells us that a positive proportion of the time blocks provide a fixed amount of shear-induced enhanced dissipation.
\begin{proposition}\label{prop:blocks}
For each fixed \(\nu\), the random variables \(Z_j^\nu\) are independent and identically distributed, with common law equal to that of
\begin{align}\label{rescaled_brownian}
Z=\inf_{a\in\mathbb R}\int_0^1(B_r-a)^2\,dr,
\end{align}where \(B=(B_r)_{0\le r\le1}\) is a standard one-dimensional Brownian motion starting at zero.
Moreover, there are universal constants $z_*>0$, $p_*>0$, and $c,C>0$ such that, for
$I_j^\nu:=\mathbf1_{\{Z_j^\nu\ge z_*\}}$, we have for any $N\in\mathbb N$ 
\begin{equation}\label{eq:block-concentration}
 \mathbb P\left(\exists\ell\ge N:
          \sum_{j=0}^{\ell-1}I_j^\nu<\frac{p_*\ell}{2}\right)
 \le Ce^{ -cl} .
\end{equation}
\begin{equation}\label{eq:pathwise-linear-block1}
 \|S_\nu(t_{j+1},t_j)h\|_{L^2(\Omega)}\le e^{-Z_j^\nu}\|h\|_{L^2(\Omega)}.
\end{equation}
 
\end{proposition}
\begin{proof}
Since $Z_j^\nu$ is a measurable functional of the $j^{th}$-Brownian increment, we know that $\{Z_j^{\nu}\}$ are independent. 
Next, note for a continuous function $b$ on $[0,1]$, that
\begin{align}\label{complete_square}
 \inf_a\int_0^1(b(r)-a)^2\,d r
 =\int_0^1(b(r)-\bar b)^2\,d r\quad \text{ with }
 \bar b=\int_0^1b(r)\,d r.
\end{align}
Since $\tau_\nu=\nu^{-1/2}$ and $t_j=j\tau_\nu$, we define  $r=t_j+\tau_\nu s$.  Since $d r=\tau_\nu\,d s$, $0\le s\le1$,
we then obtain
\begin{align*}
Z_j^\nu
&=
\nu\inf_{a\in\mathbb R}
\int_{t_j}^{t_j+\tau_\nu}
(W_r-W_{t_j}-a)^2\,d r=
\nu\tau_\nu\inf_{a\in\mathbb R}
\int_0^1
(W_{t_j+\tau_\nu s}-W_{t_j}-a)^2\,d s.
\end{align*}
Let $B=(B_s)_{0\le s\le1}$ be a standard one-dimensional
Brownian motion, then we use rescaling to obtain the following equality in law:
\[
\bigl(W_{t_j+\tau_\nu s}-W_{t_j}\bigr)_{0\le s\le1}
\overset{\mathrm{d}}{=}
\bigl(\sqrt{\tau_\nu}\,B_s\bigr)_{0\le s\le1}
\quad\text{in }C([0,1];\mathbb R).
\]
Consequently,
\begin{align*}
Z_j^\nu
\overset{\mathrm{d}}{=}
\nu\tau_\nu\inf_{a\in\mathbb R}
\int_0^1
(\sqrt{\tau_\nu}\,B_s-a)^2\,d s=
\nu\tau_\nu^2\inf_{a\in\mathbb R}
\int_0^1
\left(B_s-\frac{a}{\sqrt{\tau_\nu}}\right)^2\,d s.
\end{align*}
Therefore, by using \eqref{complete_square}, we have for every $\omega\in\mathfrak{D}$ that
\[
\inf_{a\in\mathbb R}
\int_0^1
\left(B_s-\frac{a}{\sqrt{\tau_\nu}}\right)^2\,d s
=
\inf_{a\in\mathbb R}
\int_0^1(B_s-a)^2\,d s
=
Z.
\]
It then follows that $Z_j^\nu\overset{\mathrm{d}}{=}\nu\tau_\nu^2 Z.$  Finally, noting that  $\nu\tau_\nu^2
=
\nu(\nu^{-1/2})^2
=
1$, we have 
\begin{align}\label{equalinlaw}
    Z_j^\nu\overset{\mathrm{d}}{=}Z.
\end{align}

Next, note that the disjoint blocks $[t_j,t_{j+1})$ depend on the independent increments.   If $Z=0$, then the continuity of a Brownian motion implies that the process $B_t$ is constant. Since $B_0=0$, we then have $B_t=0$ for all $t\in[0,1]$, and thus ${\{Z=0\}}\subset \{B_1=0\}$ and  $\mathbb P(Z=0)=0.$  In other words we have,  
${\mathbb P(Z>0)=1.}$
Noting that $\{Z>0\}=\bigcup_{\ell\ge1}\{Z\ge1/\ell\}$, we can choose
$$z_*>0 \quad \text{of the form }z_*=\frac{1}{l_*}, l_*\in \mathbb{N}, \text{ such that } p_*:=\mathbb P(Z\ge z_*)>0.$$

Let $\theta>0$ be fixed. Since the function $x\mapsto e^{-\theta x}$
is strictly decreasing, we have for any $\theta>0$ that,
\[
\left\{
\sum_{j=0}^{\ell-1}I_j^\nu<\frac{p_*\ell}{2}
\right\}
=
\left\{
e^{-\theta\sum_{j=0}^{\ell-1}I_j^\nu}
>
e^{-\theta p_*\ell/2}
\right\}.
\]
 By applying  Chebyshev's inequality, we obtain
\begin{align}\label{substitution_preceding-inequality}
\mathbb P\left(
\sum_{j=0}^{\ell-1}I_j^\nu<\frac{p_*\ell}{2}
\right)
&=
\mathbb P\left(
e^{-\theta\sum_{j=0}^{\ell-1}I_j^\nu}
>
e^{-\theta p_*\ell/2}
\right)\le
\frac{
\mathbb E\left[e^{-\theta\sum_{j=0}^{\ell-1}I_j^\nu}\right]
}{
e^{-\theta p_*\ell/2}
}=
e^{\theta p_*\ell/2}
\mathbb E\left[e^{-\theta\sum_{j=0}^{\ell-1}I_j^\nu}\right].
\end{align}
By using the independence of variables $I_j^\nu$, we then have 
\begin{align}\label{multiplication_Ij}
\mathbb E\left[e^{-\theta\sum_{j=0}^{\ell-1}I_j^\nu}\right]
&=
\mathbb E\left[
\prod_{j=0}^{\ell-1}e^{-\theta I_j^\nu}
\right]=
\prod_{j=0}^{\ell-1}
\mathbb E\left[e^{-\theta I_j^\nu}\right].
\end{align}
Additionally, by the definition of $p_*$ and \eqref{equalinlaw}, we know that for any $j$ we have 
$\mathbb P(I_j^\nu=0)=1-p_*$ and $\mathbb P(I_j^\nu=1)=p_*$. Hence,
we find that $\mathbb E\left[e^{-\theta I_j^\nu}\right]
=
1-p_*+p_*e^{-\theta}$.  Therefore, it follows from (\ref{multiplication_Ij}) that $\mathbb E\left[e^{-\theta\sum_{j=0}^{\ell-1}I_j^\nu}\right]
=
(1-p_*+p_*e^{-\theta})^\ell$, which after
substitution into (\ref{substitution_preceding-inequality}) yields
\begin{align}\label{upperbound_good_event}
\mathbb P\left(
\sum_{j=0}^{\ell-1}I_j^\nu<\frac{p_*\ell}{2}
\right)\leq 
\left[
e^{\theta p_*/2}
(1-p_*+p_*e^{-\theta})
\right]^\ell.
\end{align}
Note that, that there exists
a (deterministic) $\theta_0>0$, depending only on $p_*$, such that $\frac{p_*}{2}
-
\frac{p_*e^{-\theta}}
     {1-p_*+p_*e^{-\theta}}
\le -\frac{p_*}{4},$ $0\le\theta\le\theta_0$.
Thus, applying logarithm to the right hand side term of (\ref{upperbound_good_event}), and then differentiating it at $\theta=0$, we find that
\begin{align*}
\log\left[
e^{\theta_0p_*/2}(1-p_*+p_*e^{-\theta_0})
\right]&=
\int_0^{\theta_0}
\left(
\frac{p_*}{2}
-
\frac{p_*e^{-\theta}}
     {1-p_*+p_*e^{-\theta}}
\right)d\theta\\
&\le
-\int_0^{\theta_0}\frac{p_*}{4}\,d\theta
=
-\frac{p_*\theta_0}{4}.
\end{align*}
Now set $c=\frac{p_*\theta_0}{4}>0.$
Taking $\theta=\theta_0$ in (\ref{upperbound_good_event}), we obtain, for every positive integer $\ell $, that
\begin{align}
\mathbb P\left(
\sum_{j=0}^{\ell-1}I_j^\nu<\frac{p_*\ell}{2}
\right)
&\le
\left[
e^{\theta_0p_*/2}(1-p_*+p_*e^{-\theta_0})
\right]^\ell\notag \\
&=
\exp\left(
\ell\log\left[
e^{\theta_0p_*/2}(1-p_*+p_*e^{-\theta_0})
\right]
\right)\le e^{-c\ell}. \label{lgood}
\end{align}
 Now fix a positive integer $N$. We find that,  
\begin{align*}
 \mathbb P\left({\exists\ell\ge N}:
\sum_{j=0}^{\ell-1}I_j^\nu<\frac{p_*\ell}{2}
\right) =&
\mathbb P\left(
\bigcup_{\ell=N}^{\infty}
\left\{
\sum_{j=0}^{\ell-1}I_j^\nu<\frac{p_*\ell}{2}
\right\}
\right) \\
&\le
\sum_{\ell=N}^{\infty}
\mathbb P\left(
\sum_{j=0}^{\ell-1}I_j^\nu<\frac{p_*\ell}{2}
\right)\le
\sum_{\ell=N}^{\infty}e^{-c\ell}\\
&=
e^{-cN}\sum_{k=0}^{\infty}(e^{-c})^k=
\frac{e^{-cN}}{1-e^{-c}}.
\end{align*}
Thus, with $C=\frac{1}{1-e^{-c}}$, we further conclude that 
\[
\mathbb P\left({\exists\ell\ge N}:
\sum_{j=0}^{\ell-1}I_j^\nu<\frac{p_*\ell}{2}
\right)
\le Ce^{-cN},
\]
where positive constants $c$ and $C$  are independent of $\nu$ and $N$.
This finishes the proof of \eqref{eq:block-concentration}.


\end{proof}
Proposition \ref{prop:blocks} shows that, on an event of high probability, the random solution operator \(S_\nu(t,s)\) associated with \eqref{eq:passive-local}, satisfies the stated contraction estimate.  In the next subsection, by using Proposition  \ref{prop:blocks}, we decompose \eqref{eq:KS-local} into its Fourier modes and establish the desired nonlinear estimate.
\subsection{Modal decomposition and {estimates for the }nonlinear part}\label{sec:nonlinear}

{Recall that $u=P_{\neq}n=n-n_0$ satisfies \eqref{eq:nonzero-equation}.
The purpose of this section is to develop estimates for the terms $\Lin$ and $\Non$ defined in \eqref{eq:L-definition}-\eqref{eq:N-definition}. These estimates set the stage for a bootstrap argument in the following section as we show that the nonlinear terms in \eqref{eq:nonzero-equation} are perturbative in appropriate spaces.}
\begin{lemma}\label{lem:source} For constant $K>0,$
if
\begin{equation}\label{eq:control-ball}
 \|n_0(t)\|_{H^1(\mathbb{R})}+\|n(t)\|_{L^4(\Omega)}\le K,
\end{equation}
then we have for a {deterministic constant $C_{K,M}>0$ }that,
\begin{equation}\label{eq:source-bound}
 \|\Lin(t)\|_{L^2(\Omega)}+\|\Non(t)\|_{L^2(\Omega)}\le C_{K,M}\nu(\|u(t)\|_{L^2(\Omega)}+\|\nabla u(t)\|_{L^2(\Omega)}),
\end{equation}
{where constant $C_{K,M}>0$ is a deterministic constant.}
\end{lemma}
\begin{proof}
To obtain the desired bounds \eqref{eq:source-bound}, we first make the following observations.

First recall that $n\ge0$ and hence $n_0\geq 0$. Moreover, $\int_\Omega n\,dx{ dy}=M$ and due to conservation of mass, we have
\[
\int_{\mathbb R} n_0(y)\,d y
=
\frac1{2\pi}\int_\Omega n(x,y)\,d xd y=
\frac{M}{2\pi}.
\]
In other words we have $\|n_0\|_{L^1(\Omega)} =M$ and thus,
\[
\|u\|_{L^1(\Omega)}
=
\|n-n_0\|_{L^1(\Omega)}
\le
\|n\|_{L^1(\Omega)}+\|n_0\|_{L^1(\Omega)}=2M.
\]

We next establish $L^4$ bounds for $u$. Thanks to Jensen's inequality, we have
\[
|n_0(y)|^4
=
\left|
\frac1{2\pi}\int_{\mathbb T }n(x,y)\,d x
\right|^4
\le
\frac1{2\pi}\int_{\mathbb T} |n(x,y)|^4\,d x.
\]
Consequently, 
\begin{align*}
\|n_0\|_{L^4(\Omega)}^4
&=
2\pi\int_{\mathbb R} |n_0(y)|^4\,d y\le
2\pi\int_{\mathbb R}
\frac1{2\pi}\int_{\mathbb T} |n(x,y)|^4\,d xd y=
\|n\|_{L^4(\Omega)}^4 .
\end{align*}
Moreover, the assumption \eqref{eq:control-ball} implies that we have
\begin{align}\label{u-L4-estimate}
\|u\|_{L^4(\Omega)}
\le
\|n\|_{L^4(\Omega)}+\|n_0\|_{L^4(\Omega)}
\le
2\|n\|_{L^4(\Omega)}
\le2K.
\end{align}
Similarly, we next estimate $u$ in $L^2$. First, note that
\begin{align*}
\int_\Omega n_0u\,d xd y
&=
\int_{\mathbb R}n_0(y)
\left(\int_{\mathbb T} u(x,y)\,d x\right)d y=0.
\end{align*}
Therefore, $\|n\|_{L^2(\Omega)}^2
=
\|n_0+u\|_{L^2(\Omega)}^2
=
2\pi\|n_0\|_{L^2(\mathbb{R})}^2+\|u\|_{L^2(\Omega)}^2.$
Then we apply H\"older's inequality to obtain
\begin{align*}
\|u\|^2_{L^2(\Omega)}\leq \|n\|_{L^2(\Omega)}^2
&=
\int_\Omega |n|^{2/3}|n|^{4/3}\,dxdy\\
&\le
\left(\int_\Omega |n|\,dxdy\right)^{2/3}
\left(\int_\Omega |n|^4\,dxdy\right)^{1/3}\\
&=
M^{2/3}\|n\|_{L^4(\Omega)}^{4/3}.
\end{align*}
It follows, again from \eqref{eq:control-ball}, that
$
\|u\|_{L^2(\Omega)}\le
M^{1/3}K^{2/3}.$ To summarize, we have
\begin{align}
    \|u\|_{L^1(\Omega)\cap L^2(\Omega)\cap L^4(\Omega)} \leq C_{K,M}.\label{est:u}
\end{align}
Next, we find bounds for $n_0$ in $L^\infty$. We use the
standard 1D Gagliardo-Nirenberg inequality on $\mathbb{R}$ which gives us,
\begin{align}\label{linf_bound_n0}
\|n_0\|_{L^\infty(\mathbb{R})}^2
&\le
2\|n_0\|_{L^2(\mathbb{R})}\|n_0'\|_{L^2(\mathbb{R})}\le
\|n_0\|_{L^2(\mathbb{R})}^2+\|n_0'\|_{L^2(\mathbb{R})}^2=\|n_0\|_{H^1(\mathbb{R})}^2\le K^2.
\end{align}

Now we turn to the set of terms $\mathcal{L}$ defined in \eqref{eq:L-definition} with the aim of bounding $\mathcal{L}$ in $L^2(\Omega)$. We begin with finding an estimate for the term  $c_0'u_y$. Recall that $c_0'(y)
=
\frac{M}{4\pi}
-
\int_{-\infty}^y n_0(z)\,d z.$
Since $n_0\ge0$ and $\int_{\mathbb R} n_0\,dy=M/(2\pi)$, we have  $0
\le
\int_{-\infty}^y n_0(z)\,d z
\le
\frac{M}{2\pi}.$
It follows that $\|c_0'\|_{L^\infty_y}\le\frac{M}{4\pi}.$
Consequently,
\begin{align*}
\|c_0'u_y\|_{L^2(\Omega)}^2
&=
\int_{\mathbb R}\int_{\mathbb T}
|c_0'(y)|^2|u_y(x,y)|^2\,d xd y\le
\left(\frac{M}{4\pi}\right)^2\|u_y\|_{L^2(\Omega)}^2.
\end{align*}
Thus,
\[
\|c_0'u_y\|_{L^2(\Omega)}
\le
\frac{M}{4\pi}\|\nabla u\|_{L^2(\Omega)}\le C_M\|\nabla u\|_{L^2(\Omega)}.
\]
Next we consider the term $n_0'\psi_y$ and collect all the necessary estimates for $\psi$. 
Since $\psi$ satisfies
$$-\Delta\psi=u,\quad\text{and}\quad P_0\psi=0,$$ 
invoking the Poisson estimate \eqref{eq:X-Poisson} for $p=4$ and the estimate \eqref{u-L4-estimate} we obtain  
$$\|\nabla\psi\|_{L^\infty(\Omega)}\le C_{K,M},$$
Next, we apply elliptic regularity estimate
\eqref{eq:elliptic-L2} to obtain
\[
\|\psi_y\|_{L^2(\Omega)}\le C\|u\|_{L^2(\Omega)},
\qquad
\|\psi_{yy}\|_{L^2(\Omega)}\le C\|u\|_{L^2(\Omega)}.
\]
Whereas, by applying Gagliardo-Nirenberg inequality to the
$L^2(\mathbb T)$-valued function $y\mapsto\psi_y(\cdot,y)$, we have
\begin{align*}
\|\psi_y\|_{L^\infty_yL^2_x}
&\le
2\|\psi_y\|^\frac12_{L^2(\Omega)}
\|\psi_{yy}\|^\frac12_{L^2(\Omega)}\le C\|u\|_{L^2(\Omega)}.
\end{align*}

Next we turn to the term $n_0'\psi_y.$   By using H\"older's inequality and the hypothesis \eqref{eq:control-ball}, we have
\begin{align}\label{square_root_before}
\|n_0'\psi_y\|_{L^2(\Omega)}
&=\left(
\int_{\mathbb R} |n_0'(y)|^2
\left(\int_{\mathbb T}|\psi_y(x,y)|^2\,d x\right)d y\right)^{\frac12}\nonumber\\
&\leq
\|\psi_y\|_{L^\infty_yL^2_x}
\|n_0'\|_{L^2_y}\le CK\| u\|_{L^2(\Omega)}.
\end{align}
To estimate $n_0u$, we use the $L^\infty_y$ bound for $n_0$ proved in \eqref{linf_bound_n0} as follows,
\begin{align*}
\|n_0u\|_{L^2(\Omega)}
&=
\left(\int_{\mathbb R}\int_{\mathbb T}
|n_0(y)|^2|u(x,y)|^2\,d xd y\right)^{\frac12}\le
\|n_0\|_{L^\infty_y}\|u\|_{L^2(\Omega)}\le K\|u\|_{L^2(\Omega)}.
\end{align*}

We collect all the estimates for the terms appearing in (\ref{eq:L-definition}) which tells us that $$\|\Lin\|_{L^2(\Omega)}\le C_{K,M}\nu(\| u\|_{L^2(\Omega)}+\|\nabla u\|_{L^2(\Omega)}).$$
Next, we will establish similar bounds for $\mathcal{N}$ defined in (\ref{eq:N-definition}). Since $\|f_{\not=}\|_{L^2(\Omega)}\le\|f\|_{L^2(\Omega)}$, we use H\"older's inequality to obtain 
\begin{align*}
\nu^{-1}\|\Non\|_{L^2(\Omega)}
&=
\left\|
\left(\nabla u\cdot\nabla\psi-u^2\right)_{\not=}
\right\|_{L^2(\Omega)}\\
&\le
\|\nabla u\cdot\nabla\psi-u^2\|_{L^2(\Omega)}\\
&\le
\|\nabla u\cdot\nabla\psi\|_{L^2(\Omega)}+\|u^2\|_{L^2(\Omega)}\\
&\le
\|\nabla\psi\|_{L^\infty(\Omega)}\|\nabla u\|_{L^2(\Omega)}+\|u\|_{L^4(\Omega)}^2\\
&\le C_{K,M}\|\nabla u\|_{L^2(\Omega)}+\|u\|_{L^4(\Omega)}^2.
\end{align*}

By using \eqref{est:u} we then obtain
\[
\|u\|_{L^4(\Omega)}^2
\le
C\|u\|_{L^2(\Omega)}\left(\|\nabla u\|_{L^2(\Omega)}+\|u\|_{L^2(\Omega)}\right)
=
C_{K,M}(\|\nabla u\|_{L^2(\Omega)}+\|u\|_{L^2(\Omega)}).
\]
Thus,
\[
\|\Non\|_{L^2(\Omega)}
\le C_{K,M}\nu(\|\nabla u\|_{L^2(\Omega)}+\|u\|_{L^2(\Omega)}).
\]
This concludes the proof of (\ref{eq:source-bound}). 
\end{proof}

\subsection{Bootstrap argument and enhanced dissipation estimate}\label{sec:bootstrap}
In this subsection, we will establish appropriate regularity of the maximal local-in-time solution given in Theorem \ref{thm:local-complete} which is required in the bootstrap argument given in the following section. 
\begin{lemma}\label{lem:Moser}
Let $(n,\tau_*)$ be the maximal local mild solution to (\ref{eq:KS-local}) in the sense of Definition \ref{def:maximal-local}. Then,
\begin{equation}\label{eq:Moser-bound}
 \sup_{t\in [0, \tau_*)}\|n(t)\|_{L^\infty(\Omega)}
 \le C\left(\sup_{t\in [0, \tau_*)}\|n(t)\|_{L^2(\Omega)},M,\|n_{\rm in}\|_{L^\infty(\Omega)}\right).
\end{equation}
\end{lemma}
\begin{proof}
 Fix {$\omega\in\mathfrak{D}$} and \(T<\tau_*\). Define {the stochastic shear flow} $\Phi(x,y):=(x+yW_t,y),$ and {pull back the unknown functions} by settings $g(t,x,y):=n(t,\Phi_t(x,y))$ and
$\widetilde c(t,x,y):=c(t,\Phi_t(x,y)).$ By using
Stratonovich chain rule, we have from (\ref{eq:KS-local}) that  \(g\) satisfies
\begin{equation}\label{eq:sheared-KS-Lp}
\partial_tg
=
\nu\nabla\cdot(A_t\nabla g)
-
\nu\nabla\cdot(A_tg\nabla\widetilde c),
\qquad
-\nabla\cdot(A_t\nabla\widetilde c)=g,
\end{equation}
where
\[
A_t
=
\begin{pmatrix}
1+W_t^2&-W_t\\
-W_t&1
\end{pmatrix}.
\]
By using Theorem \ref{thm:local-complete}, we then know that 
 $g\in L^2(0,T;H^2(\Omega)),$ $\partial_tg\in L^2(0,T;L^2(\Omega))$ and $g\in L^\infty((0,T)\times\Omega)$ almost surely.
Therefore, by using integration by parts, we have from \eqref{eq:sheared-KS-Lp} that
\begin{align}\label{shear_transform_before}
\frac{d}{dt}\int_\Omega g^p\,dxdy
&=
p\nu\int_\Omega
g^{p-1}\nabla\cdot(A_t\nabla g)\,dxdy
-
p\nu\int_\Omega
g^{p-1}\nabla\cdot(A_tg\nabla\widetilde c)\,dxdy
\nonumber\\
&=
-p\nu(p-1)
\int_\Omega
g^{p-2}\nabla g^{\mathsf T}A_t\nabla g\,dxdy
+p\nu(p-1)
\int_\Omega
g^{p-1}\nabla g^{\mathsf T}A_t\nabla\widetilde c\,dxdy.
\end{align}
On the other hand, since $\nabla(g^{p/2})
=
\frac p2g^{p/2-1}\nabla g$ and $p g^{p-1}\nabla g=\nabla(g^p),$
 we have
\begin{align*}
p\nu(p-1)
\int_\Omega
g^{p-1}\nabla g^{\mathsf T}A_t\nabla\widetilde c\,dxdy
&=
\nu(p-1)
\int_\Omega
\nabla(g^p)^{\mathsf T}A_t\nabla\widetilde c\,dxdy
\\
&=
-\nu(p-1)
\int_\Omega
g^p\nabla\cdot(A_t\nabla\widetilde c)\,dxdy
\\
&=
\nu(p-1)\int_\Omega g^{p+1}\,dxdy.
\end{align*}
Substituting this into (\ref{shear_transform_before}), we have 
\begin{align}\label{eq:sheared-Lp-energy}
&\frac{d}{dt}\int_\Omega g^p\,dxdy
+
\frac{4\nu(p-1)}p
\int_\Omega
\nabla(g^{p/2})^{\mathsf T}
A_t\nabla(g^{p/2})\,dxdy
=
\nu(p-1)\int_\Omega g^{p+1}\,dxdy.
\end{align}
Noting that the {transform $\Phi$} preserves all spatial \(L^q\)-norms, we have $\|g(t)\|_{L^2(\Omega)}=\|n(t)\|_{L^2(\Omega)}\le K_2$ for some constant $K_2>0.$ Invoking the Ladyzhenskaya inequality \eqref{eq:Ladyzhenskaya}, we have
\[
\|g^{\frac{p}2}\|_{L^4(\Omega)}^2
\le
C\|g^{\frac{p}2}\|_{L^2(\Omega)}
\left(\bigg(
\int_\Omega
\nabla(g^{p/2})^{\mathsf T}
A_t\nabla(g^{\frac{p}2})\,dxdy\bigg)^{\frac{1}{2}}+\|g^{\frac{p}2}\|_{L^2(\Omega)}\right).
\]
Therefore,
\begin{align*}
\int_\Omega g^{p+1}\,dxdy&\le
\|g\|_{L^2(\Omega)}\|g^{p/2}\|_{L^4(\Omega)}^2
\le
CK_2\left(\|g(t)\|_{L^p(\Omega)}^{\frac{p}2}\bigg(
\int_\Omega
\nabla(g^{\frac{p}2})^{\mathsf T}
A_t\nabla(g^{\frac{p}2})\,dxdy\bigg)^{\frac{1}{2}}+\|g(t)\|_{L^p(\Omega)}^p\right).
\end{align*}
By using (\ref{eq:sheared-Lp-energy}), Young's inequality then yields, for some constant $C_{E}>0$, that
\begin{align}\label{eq:sheared-Moser-energy}
&\frac{d}{dt}\|g(t)\|_{L^p(\Omega)}^p+\nu \bigg(
\int_\Omega
\nabla(g^{\frac{p}2})^{\mathsf T}
A_t\nabla(g^{\frac{p}2})\,dxdy\bigg)
\le
C_E\nu p^2(1+K_2^2)\|g(t)\|_{L^p(\Omega)}^p.
\end{align}
 For \(p\ge4\),  the  Nash inequality (\ref{eq:inhom-Nash}) gives, for some constant $C_N>0$, that
\begin{align}\label{sheared_nash_inequality}
\|g^{p/2}\|_{L^2(\Omega)}^2
\le&
C_N\sup_{0\le t\le T}\|g(t)\|^{\frac{p}2}_{L^{\frac{p}2}(\Omega)}\bigg(
\int_\Omega
\nabla(g^{p/2})^{\mathsf T}
A_t\nabla(g^{p/2})\,dxdy\bigg)^{1/2}+C_N\sup_{0\le t\le T}\|g(t)\|^p_{L^{\frac{p}{2}}(\Omega)}.
\end{align}
 Thanks to (\ref{sheared_nash_inequality}) and (\ref{eq:sheared-Moser-energy}), we obtain 
\begin{align}\label{eq:sheared-Moser-recursion}
&\sup_{0\le t\le T}\|g(t)\|_{L^p(\Omega)}
\le
\max\left\{
\|n_{\rm in}\|_{L^p(\Omega)},\,
[{C_0p^2(1+K_2^2)]^{1/p}}\sup_{0\le t\le T}\|g(t)\|_{L^{\frac{p}2}(\Omega)}
\right\},
\end{align}
where \(g(0)=n_{\rm in}\). Iterating \eqref{eq:sheared-Moser-recursion} with \(p=2^j\) gives
\[
\sup_{0\le t\le T}\|g(t)\|_{L^\infty(\Omega)}
\le
C(K_2,M,\|n_{\rm in}\|_{L^\infty(\Omega)}).
\]
Finally, since  $\|n(t)\|_{L^\infty(\Omega)}=\|g(t)\|_{L^\infty(\Omega)}$, we therefore obtain
\[
\sup_{0\le t\le T}\|n(t)\|_{L^\infty(\Omega)}
\le
C(K_2,M,\|n_{\rm in}\|_{L^\infty(\Omega)}).
\]
This finishes the proof of Lemma \ref{lem:Moser}.
\end{proof}

 We are now in position to employ a bootstrap argument for the solution to (\ref{eq:KS-local}).  
 For some fixed $0<\beta<1/4$, we define, 
\begin{align}\label{G_nu_def}
 N_\nu:=\lceil\nu^{-\beta}\rceil\quad\text{and}\quad G_\nu:=\left\{\sum_{j=0}^{\ell-1}I_j^\nu\ge \frac{p_*\ell}2
                      \text{ for every integer }\ell\ge N_\nu\right\}.
\end{align}
The motivation behind this definition comes from the blockwise contraction shown in Proposition \ref{prop:blocks} which states that the probability of $G_\nu$ is high for small $\nu$. In what follows, we will prove the existence of a global solution on this event $G_\nu$ of large probability. To this end, we apply a bootstrap argument based on the following stopping time:
For any (deterministic constant) $K\ge1$, we define 
\begin{equation}\label{eq:bootstrap-exit}
 \sigma_K:=\inf\left\{t: \|n_0(t\wedge \tau_*)\|_{H^1(\mathbb{R})}+\|n(t\wedge \tau_*)\|_{L^4(\Omega)}\ge4K\right\}\wedge
         \tau_*.
\end{equation}

\begin{lemma}\label{prop:bootstrap}
{Let $(n,\tau_*)$ be the maximal local solution to (\ref{eq:KS-local}), in the sense of Definition \ref{def:maximal-local}, and let $u=n-n_0$, where $n_0=P_0n$ (see \eqref{eq:zero-equation}-\eqref{eq:nonzero-equation})}.
Then, there exist deterministic constants $K$ and $\nu_0>0$, depending only on $\beta\in(0,\frac14)$
and the initial data, such that for $0<\nu\le\nu_0$, on the set $G_\nu$ given in (\ref{G_nu_def}), the following inequalities hold true:
\begin{align}
& \|u(t)\|_{L^2(\Omega)}\le C e^{-c(\sqrt\nu\,t-N_\nu)_+}\|u(0)\|_{L^2(\Omega)},
                    ~~0\le t<\sigma_K,\label{eq:bootstrap-decay}\\
 &\nu\int_0^{\sigma_K}\|\nabla u(t)\|_{L^2(\Omega)}^2\,d t\le2\|u(0)\|_{L^2(\Omega)}^2,\label{eq:bootstrap-gradient}\\
  &\sup_{t<\sigma_K}\|n(t)\|_{L^\infty(\Omega)}\le C.\label{eq:bootstrap-infty}\\
 \sup_{t<\sigma_K}
&\bigl(\|n_0(t)\|_{H^1(\mathbb{R})}+\|n(t)\|_{L^4(\Omega)}\bigr)\le2K.\label{eq:half-control}
\end{align}
Consequently, on the set $G_\nu$ we have $$\sigma_K=\tau_*=\infty,$$ where $\sigma_K$ is as defined in (\ref{eq:bootstrap-exit}).
\end{lemma}

\begin{proof}

Fix an $\omega\in G_\nu$ for which Theorem \ref{thm:local-complete} gives the local regularity
\eqref{eq:regularity-conclusion}.
We transform the SPDE (\ref{eq:nonzero-equation}) to a pathwise deterministic PDE with random coefficients. For that purpose, 
set
\[
\widetilde u(t,x,y)=u(t,x+yW_t,y),
\qquad
A_t=
\begin{pmatrix}
1+W_t^2 & -W_t\\
-W_t & 1
\end{pmatrix}.
\]
With this transformation applied to the nonzero equation (\ref{eq:nonzero-equation}) we obtain,
\begin{align}\label{transformed-eq}
\partial_t\widetilde u
=
\nu\nabla\cdot(A_t\nabla\widetilde u)
-(\Lin+\Non)(t,x+yW_t,y).
\end{align}
Since $W_t$ has continuous paths, local regularity for $u$ (see \eqref{eq:regularity-conclusion}) implies that we have $\widetilde u
\in C([0,T];H^1(\Omega))
\cap L^2(0,T;H^2(\Omega)).$
Hereon, for ease of notation we will use the following notation:
$$U(t):=\|u(t)\|_{L^2(\Omega)},\quad\text{ and }\quad Y(t):=\|\nabla u(t)\|_{L^2(\Omega)}.$$
Then, by invoking Lemma~\ref{lem:source} we know that $\|\Lin+\Non\|_{L^2(\Omega)}
\le C_K\nu(U+Y),$ which implies
$ \Lin+\Non\in L^2((0,T);L^2(\Omega))$.
It then follows from \eqref{transformed-eq}  that   $\partial_t\widetilde u\in L^2(0,T;L^2(\Omega)).$  Noting that
 $\partial_x\widetilde u
=(\partial_xu)(t,x+yW_t,y)$, $\partial_y\widetilde u-W_t\partial_x\widetilde u
=(\partial_yu)(t,x+yW_t,y)$,
 we have $\|\widetilde u\|_{L^2(\Omega)}^2
=U^2,$ $\int_\Omega
\nabla\widetilde u^{\mathsf T}A_t\nabla\widetilde u\,dxdy
=
\|\partial_x\widetilde u\|_{L^2(\Omega)}^2
+\|\partial_y\widetilde u-W_t\partial_x\widetilde u
  \|_{L^2(\Omega)}^2
=Y^2.$  Moreover, since $U^2$ is absolutely continuous on $[0,T]$, we test the nonzero equation (\ref{transformed-eq}) with $\tilde u$ and use the previous estimates to obtain
  \[
 \frac12(U^2)'+\nu Y^2=-\langle\Lin+\Non,u\rangle,
 \qquad
 |\langle\Lin+\Non,u\rangle|
 \le C_K\nu(U+Y)U
 \le\frac\nu2Y^2+C_K\nu U^2,
\]
where   $\Lin$ and $\Non$ are given by (\ref{eq:L-definition}) and (\ref{eq:N-definition}), respectively.
Thus, we have
\begin{equation}\label{eq:nonzero-energy}
 \frac{\mathrm d}{\mathrm dt}U^2+\nu Y^2
                  \le C_K\nu U^2,
\end{equation}
where we emphasize that $C_K>0$ is a deterministic constant. 

{For the fixed $\omega \in G_\nu$ we divide the time interval $[0,\sigma_K(\omega))$ in $N_K(\omega)$ many blocks of width $\tau_\nu:=\nu^{-1/2}$. We let
$t_j:=j\tau_\nu$ and consider the blocks $[t_j,t_{j+1})\subset[0,\sigma_K)$}.  After applying the Gr\"onwall inequality to \eqref{eq:nonzero-energy} shows that 
\begin{align}\label{eqn:recursive}
    \sup_{t_j\le t\le t_{j+1}}U(t)^2
       \le e^{C_K\sqrt\nu}U_j^2,\qquad
 \nu\int_{t_j}^{t_{j+1}}Y^2\,d t
       \le(1+C_K\sqrt\nu e^{C_K\sqrt\nu})U_j^2,
 \quad\text{ where } U_j:=U(t_j).
\end{align}
{Since $C_K$ is deterministic}, for the fixed parameter $K$ we can choose $\nu$ sufficiently small so that
\begin{equation}\label{eq:one-block-energy}
 \sup_{t_j\le t\le t_{j+1}}U(t)\le2U_j,
 \qquad \nu\int_{t_j}^{t_{j+1}}Y^2\,d t\le2U_j^2.
\end{equation}
We invoke
the mild representation (\ref{eq:nonzero-equation}) with the initial condition chosen at \(t_j\) and write 
\begin{equation}\label{eq:restarted-nonzero-mild}
\begin{aligned}
u(t_{j+1})
&=
S_\nu(t_{j+1},t_j)u(t_j)-
\int_{t_j}^{t_{j+1}}
S_\nu(t_{j+1},t)
\bigl(\Lin(t)+\Non(t)\bigr)\,d t.
\end{aligned}
\end{equation}
 Thanks to (\ref{eq:pathwise-linear-block1}), we then have 
\begin{equation}\label{eq:linear-part-on-block}
\Vert {
S_\nu(t_{j+1},t_j)u(t_j)}\Vert_{L^2(\Omega)}
\le
e^{-Z_j^\nu}\Vert {u(t_j)}\Vert_{L^2(\Omega)}
=
e^{-Z_j^\nu}U_j.
\end{equation}
We apply Minkowski's inequality to the second term on the right hand side of (\ref{eq:restarted-nonzero-mild}), and use
\(L^2\)-contraction of $S_{\nu}$ to obtain
\begin{align}
\left\|
\int_{t_j}^{t_{j+1}}
S_\nu(t_{j+1},t)
\bigl(\Lin(t)+\Non(t)\bigr)\,d t
\right\|_{L^2(\Omega)}
&\le
\int_{t_j}^{t_{j+1}}
\left\|
S_\nu(t_{j+1},t)
\bigl(\Lin(t)+\Non(t)\bigr)
\right\|_{L^2(\Omega)}
\,dt
\notag\\
&\qquad\le
\int_{t_j}^{t_{j+1}}
\left(
\Vert{\Lin(t)}_{L^2(\Omega)}
+
\Vert{\Non(t)}\Vert_{L^2(\Omega)}
\right)\,d t.
\label{eq:Duhamel-L2-contraction}
\end{align}
Applying
\eqref{eq:control-ball} and
\eqref{eq:source-bound}, to the inequality above we arrive at 
\begin{equation}\label{eq:block-source-integral}
\left\|
\int_{t_j}^{t_{j+1}}
S_\nu(t_{j+1},t)
\bigl(\Lin(t)+\Non(t)\bigr)\,d t
\right\|_{L^2(\Omega)}
\le
C_K\nu
\int_{t_j}^{t_{j+1}}
\bigl(U(t)+Y(t)\bigr)\,d t.
\end{equation}
Combining 
\eqref{eq:linear-part-on-block} and
\eqref{eq:block-source-integral} with \eqref{eq:restarted-nonzero-mild}, we obtain
\begin{equation}\label{eq:block-recurrence-before-error}
U_{j+1}
\le
e^{-Z_j^\nu}U_j
+
C_K\nu
\int_{t_j}^{t_{j+1}}
\bigl(U(t)+Y(t)\bigr)\,d t.
\end{equation}
Next, we consider the right side integral terms. We use \eqref{eq:one-block-energy} and the fact that \(t_{j+1}-t_j=\tau_\nu\), to obtain
\begin{align}
\nu\int_{t_j}^{t_{j+1}}U(t)\,dt
&\le
\nu(t_{j+1}-t_j)
\sup_{t_j\le t\le t_{j+1}}U(t)
\le
2\nu\tau_\nu U_j
=2\nu^{1/2}U_j.
\label{eq:block-X-integral}
\end{align}
Similarly, the Cauchy--Schwarz inequality gives
\begin{align}
\nu\int_{t_j}^{t_{j+1}}Y(t)\,d t
&\le
\nu(t_{j+1}-t_j)^{1/2}
\left(
\int_{t_j}^{t_{j+1}}Y(t)^2\,d t
\right)^{1/2}
=
\nu\tau_\nu^{1/2}
\left(
\int_{t_j}^{t_{j+1}}Y(t)^2\,d t
\right)^{1/2}.
\label{eq:block-Y-CS}
\end{align}
The second estimate in \eqref{eq:one-block-energy} indicates that
\begin{equation}\label{eq:block-Y2-integral}
\left(
\int_{t_j}^{t_{j+1}}Y(t)^2\,d t
\right)^{1/2}
\le
\sqrt{\frac{2}{\nu}}\,U_j.
\end{equation}
Substituting \eqref{eq:block-Y2-integral} into
\eqref{eq:block-Y-CS} and using
\(\tau_\nu=\nu^{-1/2}\), we obtain
\begin{align}
\nu\int_{t_j}^{t_{j+1}}Y(t)\,d t
&\le
\nu\tau_\nu^{1/2}
\sqrt{\frac{2}{\nu}}\,U_j
=
\sqrt{2}\,\nu^{1/2}\tau_\nu^{1/2}U_j
=
\sqrt{2}\,\nu^{1/4}U_j.
\label{eq:block-Y-integral}
\end{align}
It follows from \eqref{eq:block-X-integral} and
\eqref{eq:block-Y-integral} that
\begin{align}
C_K\nu
\int_{t_j}^{t_{j+1}}(U(t)+Y(t))\,d t
&\le
C_K\left(
\nu^{1/2}+\nu^{1/4}
\right)U_j,
\label{eq:block-total-error}
\end{align}
Define $\delta_\nu
:=
C_K\left(\nu^{1/4}+\nu^{1/2}\right),$
then \eqref{eq:block-recurrence-before-error} becomes
\begin{equation}\label{eq:nonlinear-recurrence}
U_{j+1}
\le
\left(e^{-Z_j^\nu}+\delta_\nu\right)U_j.
\end{equation}

We will now prove \eqref{eq:bootstrap-decay}.

  Recall that $I_j^\nu=\mathbf1_{\{Z_j^\nu\ge z_*\}}$, then we have  if $I_j^\nu=1,$ $Z_j^\nu\ge z_*$ for which we have $e^{-Z_j^\nu}+\delta_\nu\leq r_{\nu}$.  On the other hand, if $I_j^\nu=0,$ then $e^{-Z_j^\nu}+\delta_\nu\leq 1+\delta_\nu$.  For sufficiently small $\nu$, we have for some constant $\gamma>0$ {independent of $\nu$,} that
\[
 (e^{-z_*}+\delta_\nu)<1,\qquad
 \frac{p_*}{2}\log (e^{-z_*}+\delta_\nu)+
 \left(1-\frac{p_*}{2}\right)\log(1+\delta_\nu)\le-\gamma
.\] 
Consider the case where $\tau_\nu N_\nu<\sigma_K$. For our fixed $\omega\in G_\nu$, if for any $l\in\mathbb{N}$ satisfying $l>N_\nu$ if we also have $l\tau_\nu =t_l < \sigma_K$,
then iterating \eqref{eq:nonlinear-recurrence} gives us
\begin{equation}\label{eq:block-product}
 U_\ell\le
 (e^{-z_*}+\delta_\nu)^{\frac{p_*\ell}2}
 (1+\delta_\nu)^{(1-\frac{p_*}{2})\ell}U_0
 \le e^{-\gamma\ell}U_0,
\end{equation}
since among any first $l$-many time blocks we know that there are at least $p_*l/2$ many blocks on which $I_j^\nu=1$.

More generally, if \( N_\nu\tau_\nu<t<\sigma_K\) then by choosing \(l\ge N_\nu\)
such that \(t\in[t_l,t_{l+1})\) and by
using \eqref{eq:one-block-energy} and \eqref{eq:block-product}, we then obtain the following desired decay estimate for  large times
\begin{align}\label{combine_regime2}
U(t)
\le2U_l
\le2e^{-\gamma l}U(0)
\le2e^\gamma e^{-\gamma\sqrt{\nu}t}U(0).
\end{align}
Whereas, for all other cases, if $t\le \min\{N_\nu\tau_\nu, \sigma_K\}$, we have from \eqref{eq:nonzero-energy} for small enough $\nu$ (since $\sqrt\nu N_\nu\to0$ as $\nu\to0$), that
\begin{align}\label{regime1_short}
 U(t)\le e^{C_K\nu t/2}U_0
 \le e^{C_K\sqrt\nu N_\nu/2}U_0\le2U_0 \leq C e^{-c(\sqrt{\nu}t-N_\nu)_+}U(0),
\end{align}
since for \(0\le t\le N_\nu\tau_\nu\) we have $\bigl(\sqrt{\nu}\,t-N_\nu\bigr)_+=0$.  Combining \eqref{regime1_short} and \eqref{combine_regime2}, we thus conclude that \eqref{eq:bootstrap-decay} holds.

Next we will prove (\ref{eq:bootstrap-gradient}). Recall that $t_{N_\nu}=N_\nu\tau_\nu$. 
Note, by applying the bounds \eqref{regime1_short} on the interval $(0,t_{N_\nu})$ and the decaying estimates \eqref{eq:block-product} on the interval $(t_{N_\nu},\sigma_K)$, we obtain,
\begin{align}\label{eq:integrated-X}
\nu\int_0^{\sigma_K}U(t)^2\,dt
&\le
4\nu t_{N_\nu}U_0^2
+
\sum_{\substack{j\ge N_\nu\\ t_j<\sigma_K}}
4\nu\tau_\nu U_j^2\le
4\sqrt{\nu}\,N_\nu U_0^2
+
4\sqrt{\nu}\,U_0^2
\sum_{j=N_\nu}^{\infty}e^{-2\gamma j}
\nonumber\\
&\le
4\sqrt{\nu}
\left(
N_\nu+\frac{1}{1-e^{-2\gamma}}
\right)U_0^2
\nonumber\\
&\le
C\sqrt{\nu}(N_\nu+1)U_0^2. 
\end{align}
We integrate \eqref{eq:nonzero-energy}   from $0$ to $T<\sigma_K$. Then, by using (\ref{eq:integrated-X}), we obtain
\[
 \nu\int_0^T Y^2(t)\,d t\le U_0^2+C_K\nu\int_0^T U^2(t)\,d t.
\]
We again choose $\nu$ sufficiently small such that it ensures $C_KC\sqrt\nu(N_\nu+1)\le1$. Then letting $T\uparrow\sigma_K$, we show that 
 \eqref{eq:bootstrap-gradient} holds.
 
Next, we will establish a uniform bound on the zero mode $n_0$ in $L^2$. For that purpose, let $F_0=(u\psi_y)_0$. Then, thanks to an application of the Cauchy-Schwarz inequality in $x$ and
\eqref{eq:elliptic-L2} we obtain
\begin{equation*}
 \nu\|F_0\|_{L^2(\mathbb{R})}
 \le C\|u\|_{L^2(\Omega)}\|\partial_y\psi\|_{L^\infty_yL^2_x}
 \le CU^2,
\end{equation*}
which thanks to \eqref{eq:integrated-X} and \eqref{eq:bootstrap-decay} implies that
\begin{align}\label{eq:zero-flux}
    \int_0^{\sigma_K} \|F_0\|_{L^2(\mathbb{R})} \leq \sup_{t\leq \sigma_K}U^2(t)\int_0^{\sigma_K}U^2(t)dt \leq C.
\end{align}
We test \eqref{eq:zero-equation} with $n_0$ which gives us
\begin{align*}
 &\frac12\frac{  d}{  dt}\|n_0\|_{L^2(\mathbb{R})}^2
       +\nu\|n_0'\|_{L^2(\mathbb{R})}^2=\nu\int_{\mathbb R} n_0'n_0c_0'\,d y+\nu\int_{\mathbb R} n_0'F_0\,d y
 \le\frac\nu2\|n_0'\|_{L^2(\mathbb{R})}^2
       +C\nu M^2\|n_0\|_{L^2(\mathbb{R})}^2+C\nu\|F_0\|_{L^2(\mathbb{R})}^2.
\end{align*}
Since $n_0\ge0$ and $\|n_0\|_{L^1(\Omega)}=M/(2\pi)$,
the interpolation inequality \eqref{eq:one-D-Nash} gives us
$\|n_0'\|_{L^2(\mathbb{R})}^2\ge cM^{-4}\|n_0\|_{L^2(\mathbb{R})}^6$.  By integrating the equation above and applying \eqref{eq:zero-flux}, we then obtain 
\begin{equation}\label{eq:zero-L2-final}
 \sup_{t<\sigma_K}\|n_0(t)\|_{L^2(\mathbb{R})}^2\le R_0,
\end{equation}
where $R_0>0$ is a constant independent of $K$.

For establishing estimates for $n$, we first use orthogonality and (\ref{eq:bootstrap-decay}) to obtain
\[
 \|n(t)\|_{L^2(\Omega)}^2=2\pi\|n_0(t)\|_{L^2_y}^2+U(t)^2
            \le2\pi R_0+16U_0^2.
\]
Moreover, Lemma~\ref{lem:Moser} then gives us our desired inequality \eqref{eq:bootstrap-infty}:
\begin{equation}\label{eq:infty-final}
 \sup_{t<\sigma_K}\|n(t)\|_{L^\infty(\Omega)}\le C,
\end{equation}

 It remains to prove \eqref{eq:half-control}. By using interpolation inequality, we immediately obtain,
\begin{equation}\label{eq:L4-final}
 \sup_{t<\sigma_K}\|n(t)\|_{L^4(\Omega)}
       \le M^{1/4}C^{3/4}.
\end{equation}
 To finish the proof of \eqref{eq:half-control} we will estimate the derivative of the zero mode $n_0$ in $L^2.$
For that purpose, we test
\eqref{eq:zero-equation} against $-n_0''$ and obtain
\begin{align}
 \frac12\frac{ d}{ dt}\|n_0'\|_{L^2(\mathbb{R})}^2+\nu\|n_0''\|_{L^2(\mathbb{R})}^2
 &=\nu\int_{\mathbb R} n_0''(n_0'c_0'-n_0^2+(u_y\psi_y+u\psi_{yy})_0)\,d y\notag\\
 &\le\frac\nu2\|n_0''\|_{L^2_y}^2+C\nu(1+\|n_0'\|_{L^2(\mathbb{R})}^2+\|(u_y\psi_y+u\psi_{yy})_0\|_{L^2_y}^2),
 \label{eq:zero-H1-energy}
\end{align}
where we used $\|n_0^2\|_{L^2(\mathbb{R})}\le C\|n_0\|_{L^2(\mathbb{R})}$ and \eqref{eq:infty-final}.  To estimate $(u_y\psi_y)_0$, we use Cauchy--Schwarz in \(x\) which gives us the following estimate:
\begin{align*}
\|(u_y\psi_y)_0\|_{L^2(\mathbb{R})}
&\le
\frac1{2\pi}
\left\|
\|u_y(\cdot,y)\|_{L^2_x}
\|\psi_y(\cdot,y)\|_{L^2_x}
\right\|_{L^2_y}
\le
\frac1{2\pi}
\|\psi_y\|_{L^\infty_yL^2_x}
\|u_y\|_{L^2(\Omega)}
\le
CUY,
\end{align*}
where, due to Lemma \ref{lem:interpolation}, in the last step we applied $\|\psi_y\|_{L^\infty_yL^2_x}\le CU$
 and $\|u_y\|_{L^2(\Omega)}\le Y$.
For the term $(u\psi_{yy})_0,$ we invoke Jensen's inequality to show that
 \begin{align*}
\|(u\psi_{yy})_0\|_{L^2(\mathbb{R})}
&\le
C\|u\psi_{yy}\|_{L^2(\Omega)}
\le
C\|u\|_{L^\infty(\Omega)}
\|\psi_{yy}\|_{L^2(\Omega)}
\le
CU,
\end{align*}
where we used (\ref{eq:infty-final})
and $\|\psi_{yy}\|_{L^2(\Omega)}
\le
\|\nabla^2\psi\|_{L^2(\Omega)}
\le CU$ (see Lemma \ref{lem:interpolation}).
Hence, we use  \eqref{eq:integrated-X} with  $\sqrt\nu(N_\nu+1)\le1$ and \eqref{eq:bootstrap-gradient} to further obtain
\begin{align}\label{Fprime_estimate}
 \|F_0'\|_{L^2(\Omega)}^2\le C(U^2+Y^2),\qquad
 \nu\int_0^{\sigma_K}\|F_0'\|_{L^2(\Omega)}^2\,d t\le C.
\end{align}
 Thanks to \eqref{eq:zero-L2-final}, we have $\|n_0''\|_{L^2(\mathbb{R})}^{{2}}\ge \Vert n_0'\Vert_{L^2(\mathbb{R})}^{{4}}/R_0$. 
Then, combining this with (\ref{Fprime_estimate}), we have for \eqref{eq:zero-H1-energy}  that 
\begin{equation}\label{eq:zero-H1-final}
 \sup_{t<\sigma_K}\|n_0'(t)\|_{L^2(\mathbb{R})}\le C.
\end{equation}
Now we collect \eqref{eq:zero-L2-final}, \eqref{eq:zero-H1-final} and \eqref{eq:L4-final} to get
\[
\sup_{t<\sigma_K}
\left(
\|n_0(t)\|_{H^1(\mathbb{R})}+\|n(t)\|_{L^4(\Omega)}
\right)
\le
C,
\]
where $C$ is independent of \(K\). By fixing \(K\) so that
\[
K>
\max\left\{
1,\,
\|n_0(0)\|_{H^1(\mathbb{R})}+\|n_{\rm in}\|_{L^4(\Omega)},\,
\frac{C}{2}
\right\}.
\]
Then, for a sufficiently small
\(\nu_0=\nu_0(K)>0\) we obtain for $0<\nu<\nu_0$ on \(G_\nu\)
\begin{equation*}
\sup_{t<\sigma_K}
\left(
\|n_0(t)\|_{H^1(\mathbb{R})}+\|n(t)\|_{L^4(\Omega)}
\right)
\le C<2K.
\end{equation*}
This finishes the proof of the desired inequality (\ref{eq:half-control}).

We finally show that $\sigma_K=\tau_*=\infty$.  If \(\sigma_K<\tau_*\), then by continuity of $n$ and the definition of \(\sigma_K\)
we would have $\|n_0(\sigma_K)\|_{H^1(\mathbb{R})}+\|n(\sigma_K)\|_{L^4(\Omega)}\geq 4K.$
On the other hand, letting \(t\uparrow\sigma_K\) in
\eqref{eq:half-control} gives $\|n_0(\sigma_K)\|_{H^1_y}+\|n(\sigma_K)\|_{L^4(\Omega)}
\le2K,$
which is impossible. Therefore \(\sigma_K=\tau_*\).  In addition, if \(\tau_*<\infty\), then \eqref{eq:infty-final} gives $\sup_{t<\tau_*}\|n(t)\|_{L^\infty(\Omega)}<\infty,$
contradicting (\ref{eq:Linfty-blowup-local}). Hence $\sigma_K=\tau_*=\infty.$

This concludes the proof of Lemma \ref{prop:bootstrap}.

\end{proof}

Finally, we are in position to prove Theorem~\ref{thm:global-original}. 
\begin{proof}[Proof of Theorem~\ref{thm:global-original}]

Let $0<\nu_0\le1$ be as given by
Lemma~\ref{prop:bootstrap}, and set $A_0:=\lceil\nu_0^{-1}\rceil.$
For any $0<\nu_A=\frac1A\le\nu_0$, invoking Theorem~\ref{thm:local-complete}, we know that exists  
a unique maximal local solution
$(\widetilde n^\nu, \widetilde\tau^\nu_*)$ to \eqref{eq:KS-local}. We denote by $(n^A,\tau^A_*)$, the maximal local mild solution to \eqref{eq1}.

For each integer $A\ge A_0$, let
$\mathcal D_A\in\mathcal F$ be a full-probability event
on which all the pathwise conclusions of
Theorem~\ref{thm:local-complete} hold for
$\widetilde n$.
Define $\mathcal D:=\bigcap_{A=A_0}^{\infty}\mathcal D_A,$ then we have
$\mathbb P(\mathcal D)=1$.


Thanks to Lemma \ref{prop:bootstrap}, we also know that, for any $\nu_A\leq \nu_0$, on the set $G_{\nu_A}=G_{A^{-1}}$, defined in (\ref{G_nu_def}), we have $\tau^A_*=\infty$; which in other words means that the maximal local-in-time solution $(n^A,\tau^A_*)$ constructed in Theorem \ref{thm:local-complete} 
in fact exists globally in time on the set $G_{A^{-1}}$. Moreover,
 Proposition~\ref{prop:blocks} gives us control over the probability of non-global-existence. In particular, recall we have for any $\nu_A>0$ that $N_{\nu_A}\ge A^\beta$ and hence that
\begin{equation}\label{eq:exceptional-event}
 \mathbb P(G_{\nu_A}^c)\le Ce^{-c{A}^{\beta}}.
\end{equation}

Therefore, we have
\[
 \sum_{A=A_0}^{\infty}\mathbb P(G_{A^{-1}}^c)
       \le C\sum_{A=A_0}^{\infty}e^{-cA^\beta}<\infty.
\]
Define $\Omega_{\rm BC}
:=
\left\{
\omega:
\omega\notin G_{A^{-1}}
\text{ for only finitely many integers }A\ge A_0
\right\}\cap \mathcal{D}.$
By the Borel--Cantelli lemma, we know that $$\mathbb P(\Omega_{\rm BC})=1.$$  
Next, for every \(\omega\in\Omega_{\rm BC}\), define
\[
A_*(\omega)
=
1+\max\left(
\{A\ge A_0, A\in\mathbb N:\omega\notin G_{A^{-1}}\}\cup\{A_0-1\}
\right).
\]
For every \(m\ge A_0\), $\{A_*\le m\}
=\mathcal{D}\cap
\bigcap_{A=m}^{\infty}G_{A^{-1}}\in\mathcal F,$
which means that the set \(A_*\) is measurable. Moreover, $\{A_*<\infty\}
=\mathcal{D}\cap
\bigcup_{m=A_0}^{\infty}
\bigcap_{A=m}^{\infty}G_{A^{-1}}
=
\Omega_{\rm BC}.$
Hence, we decude that $$\mathbb P(A_*<\infty)=1.$$ 
For every $\omega\in\Omega_*$ and every integer
$A\ge A_*(\omega)$, we have
$\omega\in\mathcal D\cap G_{A^{-1}}$, and therefore $\tau_*^{A}(\omega)=\infty.$

This concludes the proof of Theorem \ref{thm:global-original}.
\end{proof}

\section*{Acknowledgments}
K. Tawri gratefully acknowledges the support by the National Science Foundation grant DMS-2553666 (transferred from DMS-2407197).

\appendix

\section{Positive probability of finite-time blow-up}
\label{app:blowup}
\begin{proposition}
\label{prop:positive-blowup-probability}
Let $\Omega=\mathbb T\times\mathbb R$.
There exists a deterministic nonnegative initial datum $n_{\mathrm{in}}\in C_c^\infty(\Omega)$ satisfying $\int_\Omega n_{\mathrm{in}}\,dz=16\pi$
with the following property:
For each finite deterministic $A\ge1$, there exists a  unique
nonnegative maximal local mild solution $(n^A,\tau^A_*)$ to \eqref{eq1} in the sense of Definition \ref{def:maximal-local} with
initial datum $n_{\mathrm{in}}$.   Moreover, there exist deterministic constants $T_0,c_0>0$,
independent of $A$, such that
\begin{equation}\label{eq:positive-blowup-probability}
\mathbb P(\tau_*^A\le T_0)\ge e^{-c_0A}>0
\qquad\text{for every finite deterministic }A\ge1.
\end{equation}
\end{proposition}
\begin{proof}
The philosophy behind this proof is that there is a positive-probability set of paths along which the random shear remains sufficiently small, so that it does not provide enough mixing to prevent blow-up. We first transform the SPDE into a PDE with random coefficients, as in the proof of the existence theorem. Then on the aforementioned positive probability event we apply classical ideas that show that, if the solution exists for all times then a suitable localized second moment decreases strictly, which is not possible indefinitely. This proof is
inspired by the localized second-moment argument presented
in Theorem~3.1 and Lemma~3.1 of \cite{Nagai2001}
and the second-moment identity given in
Lemma~2.1 of \cite{dolbeault2004optimal}
for deterministic Keller--Segel systems. 
We give a condensed proof below.

We start by fixing an $A\ge1$. Theorem \ref{thm:local-complete}, then states that there exists a maximal local mild solution $(n^A,\tau^A_{*})$ to \eqref{eq1}.  Define $g(t,x,y):=n^A(t,x+\sqrt A\,yW_t,y),$ $\widetilde c(t,x,y):=c^A(t,x+\sqrt A\,yW_t,y)$
and
\[
A_t:=
\begin{pmatrix}
1+AW_t^2&-\sqrt A\,W_t\\
-\sqrt A\,W_t&1
\end{pmatrix},
\]
then by using the Stratonovich chain rule, we find that
\begin{equation}\label{eq:counterexample-transformed}
\begin{cases}
\partial_t g
=
\nabla\cdot(A_t\nabla g)
-\nabla\cdot(gA_t\nabla\widetilde c),&(x,y) \in\Omega, t>0,
\\
-\nabla\cdot(A_t\nabla\widetilde c)=g, \qquad (x,y) \in\Omega, t>0.
\end{cases}
\end{equation}
Another consequence of Theorem \ref{thm:local-complete}, is that  $g\ge0$ for $t\in(0,\tau^*)$  and 
$\int_\Omega g(z,t)\,dz=M:=16\pi.$ 
Now we recall, for $c$ solving $-\Delta c =n$ that we have,
\[
\nabla c=K*n,
\qquad G(x,y)
=
-\frac{1}{4\pi}
\log\bigl(2(\cosh y-\cos x)\bigr),\quad 
K=\nabla G= 
-\frac{1}{4\pi}
\frac{(\sin x,\sinh y)}{\cosh y-\cos x}.
\]
We know that $K$ is odd and satisfies
\begin{equation}\label{eq:counterexample-kernel-expansion}
K(z)
=
-\frac{z}{2\pi|z|^2}+O(|z|)
\qquad\text{as }z\to0.
\end{equation}
Since the shear transformation preserves volume,  we further have $A_t\nabla\widetilde c=K_t*g,$
where
\[
K_t(x,y)
:=
\begin{pmatrix}
1&-\sqrt A\,W_t\\
0&1
\end{pmatrix}
K(x+\sqrt A\,yW_t,y).
\]
Consequently, $K_t$ is odd and, for ${\bf z}=(x,y)$, it satisfies
\begin{equation}\label{eq:counterexample-sheared-kernel}
K_t({\bf z})
=
-\frac{{\bf z}}
{2\pi\bigl((x+\sqrt A\,yW_t)^2+y^2\bigr)}
+O(|{\bf z}|),
\end{equation}
uniformly whenever $\sqrt A\,|W_t|\le1$ and $|{\bf z}|$ is
sufficiently small. Moreover, there is a constant $C>0$ such that 
\begin{equation}\label{eq:counterexample-global-kernel}
|K_t({\bf z})|
\le C\bigl(1+d_\Omega({\bf z},0)^{-1}\bigr),
\qquad \sqrt A\,|W_t|\le1,
\end{equation}
where $d_\Omega$ denotes the distance on the cylinder.

Next, for $R>0$, chosen sufficiently small, we
let $\varphi_R\in C^\infty(\Omega)$ be a nonnegative function
satisfying
\begin{equation}\label{eq:counterexample-cutoff}
\begin{gathered}
\varphi_R({\bf z})=|{\bf z}|^2
\quad\text{for }{\bf z}\in B_R,
\\
\varphi_R({\bf z})\ge R^2
\quad\text{for }{\bf z}\in\Omega\setminus B_R,
\\
\varphi_R\text{ is constant on }\Omega\setminus B_{2R},
\\
\|\nabla\varphi_R\|_{L^\infty}\le CR,
\qquad
\|\nabla^2\varphi_R\|_{L^\infty}\le C.
\end{gathered}
\end{equation}
  Next, we define the localized second moment and mass outside $B_R$ by
\[
J(t):=\int_\Omega\varphi_R({\bf z})g({\bf z},t)\,d{\bf z},
\qquad
m_{\mathrm{out}}(t)
:=
\int_{\Omega\setminus B_R}g({\bf z},t)\,d{\bf z}.
\]
Then, we know 
\begin{equation}\label{eq:counterexample-moment-mass}
J(t)\ge0,
\qquad
m_{\mathrm{out}}(t)\le\frac{J(t)}{R^2}.
\end{equation}
Observe that by testing \eqref{eq:counterexample-transformed} with $\varphi_R$
and using the oddness of $K_t$, we obtain for $t\in(0,\tau_*)$ that
\begin{align}
J'(t)
={}&
\int_\Omega
g({\bf z})\operatorname{tr}\bigl(A_t\nabla^2\varphi_R({\bf z})\bigr)\,d{\bf z}
+
\frac12\iint_{\Omega\times\Omega}
\bigl(\nabla\varphi_R({\bf z})-\nabla\varphi_R({\bf z}')\bigr)
\cdot K_t({\bf z}-{\bf z}')\,g({\bf z})g({\bf z}')\,d{\bf z}\,d{\bf z}'.
\label{eq:counterexample-virial-identity}
\end{align}
We next estimate $J'$ term by term.  We fix some constant $\delta>0$ and define 
{
\begin{align}\label{sigma}
    \sigma_\delta^A
=
\inf\{t\ge0:\sqrt A\,|W_t|=\delta\}.
\end{align}
Then on $[0,\sigma^A_\delta)$}, since $\nabla^2\varphi_R=2I$ in $B_R$,
we have for some constant $C>0$ that
\begin{equation}\label{eq:counterexample-diffusion}
\int_\Omega
g\,\operatorname{tr}(A_t\nabla^2\varphi_R)\,d{\bf z}
\le
(4+2\delta^2)M+C\,m_{\mathrm{out}}.
\end{equation}
In light of 
\eqref{eq:counterexample-sheared-kernel} and the fact that
$\nabla\varphi_R({\bf z})-\nabla\varphi_R({\bf z}')=2({\bf z}-{\bf z}')$, for distinct ${\bf z},{\bf z}'\in B_R$, we obtain for some constant $C>0,$ that
\begin{align}
\frac12
\bigl(\nabla\varphi_R({\bf z})-\nabla\varphi_R({\bf z}')\bigr)
\cdot K_t({\bf z}-{\bf z}')
&=
-\frac{|{\bf  z}-{\bf z}'|^2}
{2\pi\bigl((z_1-z_1'+\sqrt A\,W_t(z_2-z_2'))^2
+(z_2-z_2')^2\bigr)}
+O(R^2)
\nonumber\\
&\le
-\frac{1}{2\pi(1+\delta)^2}+CR^2,
\label{eq:counterexample-interior-pointwise}
\end{align}
where the last inequality follows from
\[
\begin{aligned}
&\bigl((z_1-z_1'+\sqrt A\,W_t(z_2-z_2'))^2
+(z_2-z_2')^2\bigr)^{1/2}\le
(1+\sqrt A\,|W_t|)|{\bf z}-{\bf z}'|
\le(1+\delta)|{\bf z}-{\bf z}'|.
\end{aligned}
\]
Consequently, for some constant $C>0,$ we have that
\begin{align}
&\frac12\iint_{B_R\times B_R}
\bigl(\nabla\varphi_R(\mathbf{z})
-\nabla\varphi_R(\mathbf{z}')\bigr)
\cdot K_t(\mathbf{z}-\mathbf{z}')
\,g(\mathbf{z},t)g(\mathbf{z}',t)
\,d\mathbf{z}\,d\mathbf{z}'
\nonumber\\
&\quad\le
\left(-\frac{1}{2\pi(1+\delta)^2}+CR^2\right)
\left(\int_{B_R}g(\mathbf{z},t)\,d\mathbf{z}\right)^2
\nonumber\\
&\quad\le
-\frac{\bigl(M-m_{\mathrm{out}}(t)\bigr)^2}
{2\pi(1+\delta)^2}
+CR^2M^2,
\label{eq:counterexample-interior}
\end{align}
where we used the fact that $\int_{B_R}g({\bf z},t)\,d{\bf z}=M-m_{\mathrm{out}}(t)$ with 
$0\le M-m_{\mathrm{out}}(t)\le M.$  In addition, by using  $|\nabla\varphi_R({\bf z})-\nabla\varphi_R({\bf z}')|
\le C\min\{d_\Omega({\bf z},{\bf z}'),R\}$
together with \eqref{eq:counterexample-global-kernel}, we obtain, for some constant $C>0$, that $\left|
\bigl(\nabla\varphi_R({\bf z})-\nabla\varphi_R({\bf z}')\bigr)
\cdot K_t({\bf z}-{\bf z}')
\right|
\le C$. 
Then we conclude that, 
\begin{align}
&\frac12\left|
\iint_{(\Omega\times\Omega)\setminus(B_R\times B_R)}
\bigl(\nabla\varphi_R(\mathbf{z})
-\nabla\varphi_R(\mathbf{z}')\bigr)
\cdot K_t(\mathbf{z}-\mathbf{z}')
\,g(\mathbf{z},t)g(\mathbf{z}',t)
\,d\mathbf{z}\,d\mathbf{z}'
\right|
\nonumber\\
&\quad\le
\frac C2
\iint_{(\Omega\times\Omega)\setminus(B_R\times B_R)}
g(\mathbf{z},t)g(\mathbf{z}',t)
\,d\mathbf{z}\,d\mathbf{z}'
\nonumber\\
&\quad=
\frac C2\left[M^2-\bigl(M-m_{\mathrm{out}}(t)\bigr)^2\right]
\nonumber\\
&\quad\le CM\,m_{\mathrm{out}}(t).
\label{eq:counterexample-exterior}
\end{align}
 Upon combining
\eqref{eq:counterexample-moment-mass}--%
\eqref{eq:counterexample-exterior}, we obtain from \eqref{eq:counterexample-virial-identity} that
\begin{equation}\label{eq:counterexample-moment-inequality}
J'
\le
(4+2\delta^2)M
-\frac{M^2}{2\pi(1+\delta)^2}
+CR^2M^2
+\frac{C(1+M)}{R^2}J,
\end{equation}
for some appropriately large constant $C>0$.
 
Now we take $M=16\pi$
 and $\delta=\frac14.$
and, for simplicity, set $\alpha
:=
\frac{M^2}{2\pi(1+\delta)^2}
-(4+2\delta^2)M
=
\frac{398\pi}{25}>0.$
We choose $R>0$ sufficiently small so that $CR^2M^2\le\frac{\alpha}{4} $. Then for $H:=\frac{\alpha R^2}{4C(1+M)}$,
it then follows from \eqref{eq:counterexample-moment-inequality} that $J(t)\le H$. { In light of \eqref{sigma}, for the times we have $\sqrt A\,|W_t|\le\delta$}, this further implies that
\begin{equation}\label{eq:counterexample-negative-derivative}
J'(t)\le-\frac{\alpha}{2}.
\end{equation}
{ Now we will show that for some $T_0>0,$ we have $J(t)<H$ for all $t<\tau_*^A\wedge T_0\wedge \sigma_\delta^A$.}

To that end, let
\[
\rho(z)=
\begin{cases}
c_\rho\exp\!\left(-\dfrac{1}{1-|z|^2}\right),
& |z|<1,\\
0,& |z|\ge1,
\end{cases}
\]
where $c_\rho>0$ is chosen so that $\int_{\mathbb R^2}\rho(z)\,dz=1.$ We next define $m_\rho
=
\int_{\mathbb R^2}|z|^2\rho(z)\,dz>0.$
Choose $0<\varepsilon<
\min\left\{
\frac R2,\,
\sqrt{\frac{H}{2Mm_\rho}}
\right\}$
and set $n_{\mathrm{in}}(z)
=
M\varepsilon^{-2}\rho(z/\varepsilon),$
extended by zero outside its support.  We then have,
\begin{equation}\label{eq:counterexample-initial-moment}
J(0)=Mm_\rho\varepsilon^2<\frac H2.
\end{equation}
Set $T_0:=\frac{4J(0)}{\alpha}$
and consider the event $E_A
= \{ \sigma_\delta^A >T_0\}=
\left\{
\sup_{0\le t\le T_0}\sqrt A\,|W_t|<\delta
\right\}$ where $\sigma^A_\delta$ is defined in \eqref{sigma}.
Then, we claim that, on $E_A$, we have $J(t)<H$
\text{for all }$0\le t<\tau_*^A\wedge T_0.$  Indeed, suppose otherwise i.e.  { $\tau_H:=\inf\{t:J(t)=H\} <\tau_*^A\wedge T_0\wedge \sigma_\delta^A$. Then by integrating
\eqref{eq:counterexample-negative-derivative} on $[0,\tau_H)$ we find that $J(\tau_H)\leq J(0) <\frac{H}2$
which contradicts the fact that $J(\tau_H)=H$.}
Consequently,
\begin{equation}\label{eq:counterexample-moment-decrease}
J(t)\le J(0)-\frac{\alpha}{2}t, \qquad 0\le t<\tau_*^A,\quad t\le T_0\text{ on } E_A.
\end{equation}

Now if $\tau_*^A>T_0$, then $0\le J(T_0)
\le J(0)-\frac{\alpha T_0}{2}
=-J(0)<0,$
{which contradicts the fact that $J(T_0)<0$ obtained by \eqref{eq:counterexample-moment-decrease} and the definition of $T_0$}. Therefore,
\begin{equation}\label{eq:counterexample-event-inclusion}
E_A\subseteq\{\tau_*^A\le T_0\}
\qquad\text{up to a null set}.
\end{equation}
By the continuation criterion \eqref{eq:Linfty-blowup-local}, { on $E_A$}, we find that the solution blows up at { time $\tau^A_*\leq T_0$}.

Next, we estimate the probability of $E_A$.
For that purpose, we define
\[
F(x)=\cos\!\left(\frac{\pi x}{2\delta}\right),
\qquad
\lambda_A=\frac{\pi^2A}{8\delta^2}.
\]
Note that $\frac A2F''=-\lambda_AF$ with $F(\pm\delta)=0,$ $F(0)=1$ and 
$0\le F(x)\le1$ for $|x|\le\delta.$  Then, applying It\^o's product rule to
$e^{\lambda_At}F(\sqrt A\,W_t)$ yields
\[
\begin{aligned}
d\left(e^{\lambda_At}F(\sqrt A\,W_t)\right)
&=
e^{\lambda_At}
\left(\lambda_AF+\frac A2F''\right)
(\sqrt A\,W_t)\,dt+
\sqrt A\,e^{\lambda_At}
F'(\sqrt A\,W_t)\,dW_t\\
&=
\sqrt A\,e^{\lambda_At}
F'(\sqrt A\,W_t)\,dW_t.
\end{aligned}
\]
Until the stopping time $\sigma_\delta^A$, defined in \eqref{sigma}, by using the fact that $W_0=0$, we obtain
\begin{align}\label{taking-expectation-appendix}
e^{\lambda_A(t\wedge\sigma_\delta^A)}
F\left(\sqrt A\,W_{t\wedge\sigma_\delta^A}\right)
=
1+\sqrt A\int_0^{t\wedge\sigma_\delta^A}
e^{\lambda_As}F'(\sqrt A\,W_s)\,dW_s.
\end{align}
The stochastic integral has zero expectation on $[0,T_0]$,
since
\[
\begin{aligned}
\mathbb E\int_0^{T_0\wedge\sigma_\delta^A}
A e^{2\lambda_As}
\left|F'(\sqrt A\,W_s)\right|^2\,ds
&\le
\frac{A\pi^2}{4\delta^2}
\int_0^{T_0}e^{2\lambda_As}\,ds<\infty.
\end{aligned}
\]
Consequently, taking $t=T_0$ and then expectations in \eqref{taking-expectation-appendix} gives us
\begin{align}\label{expectation-one-process}
\mathbb E\left[
e^{\lambda_A(T_0\wedge\sigma_\delta^A)}
F\left(\sqrt A\,W_{T_0\wedge\sigma_\delta^A}\right)
\right]=1.
\end{align}
By continuity of $W$ and the definition of
$\sigma_\delta^A$, on $\{\sigma_\delta^A\le T_0\}$ we have $\sqrt A\,W_{\sigma_\delta^A}\in\{-\delta,\delta\}$ and  $F(\sqrt A\,W_{\sigma_\delta^A})=F(\pm\delta)=0.$
Therefore,
\[
\begin{aligned}
&e^{\lambda_A(T_0\wedge\sigma_\delta^A)}
F\left(\sqrt A\,W_{T_0\wedge\sigma_\delta^A}\right)=
e^{\lambda_AT_0}F(\sqrt A\,W_{T_0})
\mathbf 1_{\{\sigma_\delta^A>T_0\}}.
\end{aligned}
\]
Taking expectations on both sides of the identity above and using \eqref{expectation-one-process}, we obtain
\[
\begin{aligned}
1
&=
\mathbb E\left[
e^{\lambda_A(T_0\wedge\sigma_\delta^A)}
F\left(\sqrt A\,W_{T_0\wedge\sigma_\delta^A}\right)
\right]=
e^{\lambda_AT_0}
\mathbb E\left[
F(\sqrt A\,W_{T_0})
\mathbf 1_{\{\sigma_\delta^A>T_0\}}
\right].
\end{aligned}
\]
Finally, since $0\le F\le1$ on $[-\delta,\delta]$, we have
\[
\begin{aligned}
e^{-\lambda_AT_0}
&=
\mathbb E\left[
F(\sqrt A\,W_{T_0})
\mathbf 1_{\{\sigma_\delta^A>T_0\}}
\right]\le
\mathbb P(\sigma_\delta^A>T_0)
=
\mathbb P(E_A).
\end{aligned}
\]
Finally \eqref{eq:counterexample-event-inclusion} gives us
\[
\mathbb P(\tau_*^A\le T_0)
\ge
\mathbb P(E_A)
\ge
e^{-\lambda_AT_0}
=
e^{-c_0A}>0,
\qquad A\ge1,
\]
where $c_0:=\frac{\pi^2T_0}{8\delta^2}$. The initial datum
$n_{\mathrm{in}}$ and the constants $T_0,c_0$ are deterministic
and independent of $A$. 
\end{proof}

\bibliographystyle{abbrv}
\bibliography{ref}

@misc{CotiZelatiHairerVillringer2025,
  title         = {A Stochastic {RAGE} Theorem and Enhanced Dissipation
                   for Transport Noise},
  author        = {Coti Zelati, Michele and Hairer, Martin and
                   Villringer, David},
  year          = {2025},
  eprint        = {2507.11422},
  archivePrefix = {arXiv},
  primaryClass  = {math.AP},
  doi           = {10.48550/arXiv.2507.11422},
  url           = {https://arxiv.org/abs/2507.11422}
}

@article{SimingHe,
author = {Y. Gong and S. He},
title = {On the $8\pi$-Critical-Mass Threshold of a {P}atlak--{K}eller--{S}egel--{N}avier--{S}tokes System},
journal = {SIAM J. Math. Anal.},
volume = {53},
number = {3},
pages = {2925-2956},
year = {2021},
}

@article{CastroLear2023,
  author  = {Castro, {\'A}ngel and Lear, Daniel},
  title   = {Traveling Waves Near {Couette} Flow for the {2D Euler} Equation},
  journal = {Comm. Math. Phys.},
  year    = {2023},
  volume  = {400},
  pages   = {2005--2079},
  doi     = {10.1007/s00220-023-04636-6}
}

@article{FlandoliGaleatiLuo,
  author  = {F. Flandoli and L. Galeati and D. Luo},
  title   = {Delayed blow-up by transport noise},
  journal = {Comm. Partial Differential Equations},
  volume  = {46},
  number  = {9},
  pages   = {1757--1788},
  year    = {2021},
  url     = {https://arxiv.org/abs/2009.13005}
}

@article{Keller1970,
  title={Initiation of slime mold aggregation viewed as an instability},
  author={E. Keller and L. Segel},
  journal={J. Theor. Biol.},
  volume={26},
  number={3},
  pages={399--415},
  year={1970},
  publisher={Elsevier}
}

@article{Keller1971,
  title={Model for chemotaxis},
  author={E. Keller and L. Segel},
  journal={J. Theoret. Biol.},
  volume={30},
  number={2},
  pages={225--234},
  year={1971},
  publisher={Elsevier}
}

@article{patlak1953random,
  title={Random walk with persistence and external bias},
  author={C. Patlak},
  journal={Bull. Math. Biophys.},
  volume={15},
  number={3},
  pages={311--338},
  year={1953},
  publisher={Springer}
}

@article{kong2024global,
  title={Global existence and aggregation of chemotaxis--fluid systems in dimension two},
  author={F. Kong  and C. Lai and J. Wei},
  journal={J. Differential Equations},
  volume={400},
  pages={1--89},
  year={2024},
  publisher={Elsevier}
}

@article{Horstmann2003,
  title={From 1970 until present: the {K}eller-{S}egel model in chemotaxis and its consequences I},
  author={D. Horstmann},
  journal={Jahresber Deutsch. Math.-Verein.},
  volume={105},
  pages={103--165},
  year={2003}
}

@article{horstmann2004,
  title={From 1970 until present: the {K}eller-{S}egel model in chemotaxis and its consequences II},
  author={D. Horstmann},
  journal={Jahresber Deutsch. Math.-Verein.},
  volume={106},
  pages={51--69},
  year={2004}
}

@article{childress1981,
  title={Nonlinear aspects of chemotaxis},
  author={S. Childress and K. Percus},
  journal={Math. Biosci.},
  volume={56},
  number={3-4},
  pages={217--237},
  year={1981},
  publisher={Elsevier}
}

@article{nanjundiah1973,
  title={Chemotaxis, signal relaying and aggregation morphology},
  author={V. Nanjundiah},
  journal={J. Theor. Biol.},
  volume={42},
  number={1},
  pages={63--105},
  year={1973},
  publisher={Elsevier}
}

@article{herrero1996,
  title={Chemotactic collapse for the {K}eller-{S}egel model},
  author={M. Herrero and J. Vel{\'a}zquez},
  journal={J. Math. Biol.},
  volume={35},
  number={2},
  pages={177--194},
  year={1996},
  publisher={Springer}
}

@article {hillen2009user,
    AUTHOR = {Hillen, T. and Painter, K. J.},
     TITLE = {A user's guide to {PDE} models for chemotaxis},
   JOURNAL = {J. Math. Biol.},
  FJOURNAL = {Journal of Mathematical Biology},
    VOLUME = {58},
      YEAR = {2009},
    NUMBER = {1-2},
     PAGES = {183--217},
      ISSN = {0303-6812,1432-1416},
   MRCLASS = {92C17 (35K57)},
  MRNUMBER = {2448428},
       DOI = {10.1007/s00285-008-0201-3},
       URL = {https://doi.org/10.1007/s00285-008-0201-3},
}

@article{senba2000some,
  title={Some structures of the solution set for a stationary system of chemotaxis},
  author={T. Senba and T. Suzuki},
  journal={Adv. Math. Sci. Appl.},
  volume={10},
  number={1},
  pages={191--224},
  year={2000}
}

@article{wang2002steady,
  title={{Steady state solutions of a reaction-diffusion system modeling chemotaxis}},
  author={G. Wang and J. Wei},
  journal={Math. Nachr.},
  volume={233},
  number={1},
  pages={221--236},
  year={2002},
  publisher={Wiley Online Library}
}

@article{winkler2014stabilization,
  title={{S}tabilization in a two-dimensional chemotaxis-{N}avier-{S}tokes system},
  author={M. Winkler},
  journal={Arch. Ration. Mech. Anal.},
  volume={211},
  pages={455--487},
  year={2014},
  publisher={Springer}
}

@article{zhai20202d,
  title={{2D} stochastic chemotaxis-{N}avier-{S}tokes system},
  author={J. Zhai and T. Zhang},
  journal={J. Math. Pures Appl.},
  volume={138},
  pages={307--355},
  year={2020},
  publisher={Elsevier}
}

@article{zhang2025keller,
  title={On the {K}eller-{S}egel models interacting with a stochastically forced incompressible viscous flow in {$R^2$}},
  author={L. Zhang and B. Liu},
  journal={J. Differential Equations},
  volume={414},
  pages={487--554},
  year={2025},
  publisher={Elsevier}
}

@article{WongZakai1965Convergence,
  author  = {E. Wong and M. Zakai},
  title   = {On the convergence of ordinary integrals to stochastic integrals},
  journal = {Ann. Math. Statist.},
  year    = {1965},
  volume  = {36},
  number  = {5},
  pages   = {1560--1564},
  doi     = {10.1214/aoms/1177699916}
}

@article{CuiWangWang2026,
  author  = {S. Cui and L. Wang and W. Wang},
  title   = {Suppression of blow-up for the 3{D}
             {Patlak--Keller--Segel--Navier--Stokes} system
             via the {Couette} flow},
  journal = {J. Math. Pures Appl.},
  volume  = {210},
  pages   = {103874},
  year    = {2026},
  doi     = {10.1016/j.matpur.2026.103874}
}

@article{ChenZhaiZhang2025,
  author  = {Y. Chen and J. Zhai and T. Zhang},
  title   = {Well-posedness of stochastic chemotaxis system},
  journal = {J. Differential Equations},
  volume  = {442},
  pages   = {113531},
  year    = {2025},
  doi     = {10.1016/j.jde.2025.113531}
}

@article{WongZakai1965Relation,
  author  = {E. Wong and M. Zakai},
  title   = {On the relation between ordinary and stochastic differential equations},
  journal = {Int. J. Eng. Sci.},
  year    = {1965},
  volume  = {3},
  number  = {2},
  pages   = {213--229},
  doi     = {10.1016/0020-7225(65)90045-5}
}

@article{kong2026global,
  title={Global Well-posedness of the {2D} Stochastic Self-consistent {K}eller-{S}egel-{N}avier-{S}tokes System with Subcritical Cellular Mass},
  author={F. Kong and C. Lai and K. Tawri},
  journal={arXiv preprint arXiv:2605.17114},
  year={2026}
}

@article{FlandoliGaleatiLuo2021,
  author  = {F. Flandoli and L. Galeati and D. Luo},
  title   = {Delayed blow-up by transport noise},
  journal = {Comm. Partial Differential Equations},
  volume  = {46},
  number  = {9},
  pages   = {1757--1788},
  year    = {2021},
  doi     = {10.1080/03605302.2021.1893748}
}

@article{FengShiWang2022,
  author  = {Y. Feng  and B. Shi and W. Wang},
  title   = {Dissipation enhancement of planar helical flows and
             applications to three-dimensional {Kuramoto--Sivashinsky}
             and {Keller--Segel} equations},
  journal = {J. Differential Equations},
  volume  = {313},
  pages   = {420--449},
  year    = {2022},
  doi     = {10.1016/j.jde.2021.12.029}
}

@article{ShiWang2024,
  author  = {B. Shi and W. Wang},
  title   = {Enhanced dissipation and blow-up suppression for the
             three dimensional {Keller--Segel} equation with the
             plane {Couette--Poiseuille} flow},
  journal = {J. Differential Equations},
  volume  = {403},
  pages   = {368--405},
  year    = {2024},
  doi     = {10.1016/j.jde.2024.05.026}
}

@article{GallaySlijepcevic2015,
  author  = {T. Gallay and S. Slijep{\v{c}}evi{\'c} },
  title   = {Uniform boundedness and long-time asymptotics for the
             two-dimensional {Navier--Stokes} equations in an
             infinite cylinder},
  journal = {J. Math. Fluid Mech.},
  volume  = {17},
  number  = {1},
  pages   = {23--46},
  year    = {2015},
  doi     = {10.1007/s00021-014-0188-z}
}

@article{TerraciniTzvetkovVisciglia2014,
  author  = {S. Terracini and N. Tzvetkov and N. Visciglia},
  title   = {The nonlinear {Schr{\"o}dinger} equation ground states on product spaces},
  journal = {Anal. PDE},
  volume  = {7},
  number  = {1},
  pages   = {73--96},
  year    = {2014},
  doi     = {10.2140/apde.2014.7.73}
}

@article{LiXiangXu2025,
  author  = {H. Li and Z. Xiang  and X. Xu},
  title   = {Suppression of blow-up in
             {Patlak--Keller--Segel--Navier--Stokes} system
             via the {Poiseuille} flow},
  journal = {J. Differential Equations},
  volume  = {434},
  pages   = {113301},
  year    = {2025},
  doi     = {10.1016/j.jde.2025.113301}
}

@article{FlandoliGaleatiLuo2024,
  author  = {F. Flandoli and L. Galeati and D. Luo},
  title   = {Quantitative convergence rates for scaling limit of
             {SPDEs} with transport noise},
  journal = {J. Differential Equations},
  volume  = {394},
  pages   = {237--277},
  year    = {2024},
  doi     = {10.1016/j.jde.2024.02.053}
}

@article{ZengZhangZi2021,
  author  = {L. Zeng and Z. Zhang and R. Zi},
  title   = {Suppression of blow-up in
             {Patlak--Keller--Segel--Navier--Stokes} system
             via the {Couette} flow},
  journal = {J. Funct. Anal.},
  volume  = {280},
  number  = {10},
  pages   = {108967},
  year    = {2021},
  doi     = {10.1016/j.jfa.2021.108967}
}

@article{duan2010global,
  title={Global solutions to the coupled chemotaxis-fluid equations},
  author={R. Duan and A. Lorz and P. Markowich},
  journal={Commun. Partial Differ. Equ.},
  volume={35},
  number={9},
  pages={1635--1673},
  year={2010},
  publisher={Taylor \& Francis}
}

@article{BedrossianHe2017,
  title={Suppression of blow-up in {P}atlak--{K}eller--{S}egel via shear flows},
  author={J. Bedrossian and S. He},
  journal={SIAM J. Math. Anal.},
  volume={49},
  number={6},
  pages={4722--4766},
  year={2017},
  publisher={SIAM}
}

@article{dolbeault2004optimal,
  title={Optimal critical mass in the two dimensional {K}eller--{S}egel model in ${R^2}$},
  author={J. Dolbeault and B. Perthame },
  journal={C. R. Math. Acad. Sci. Paris},
  volume={339},
  number={9},
  pages={611--616},
  year={2004},
  publisher={Elsevier}
}

@article {blanchet2006two,
    AUTHOR = {Blanchet, A. and Dolbeault, J. and Perthame, B.},
     TITLE = {Two-dimensional {K}eller-{S}egel model: optimal critical mass
              and qualitative properties of the solutions},
   JOURNAL = {Electron. J. Differential Equations},
  FJOURNAL = {Electronic Journal of Differential Equations},
      YEAR = {2006},
     PAGES = {No. 44, 32},
      ISSN = {1072-6691},
   MRCLASS = {35Q80 (35B45 35D05 92C17)},
  MRNUMBER = {2226917},
MRREVIEWER = {Roberto\ Natalini},
}

@article{biler20068pi,
  title={{The 8$\pi$-problem for radially symmetric solutions of a chemotaxis model in the plane}},
  author={P. Biler and G. Karch and P. Lauren{\c{c}}ot and T. Nadzieja},
  journal={Math. Models Methods Appl. Sci.},
  volume={29},
  number={13},
  pages={1563--1583},
  year={2006},
  publisher={Wiley Online Library}
}

@article{velazquez2004point,
  title={{Point dynamics in a singular limit of the {K}eller--{S}egel model I: Motion of the concentration regions}},
  author={J. Vel{\'a}zquez},
  journal={SIAM J. Appl. Math.},
  volume={64},
  number={4},
  pages={1198--1223},
  year={2004},
  publisher={SIAM}
}

@article{blanchet2008infinite,
  title={{Infinite time aggregation for the critical {P}atlak-{K}eller-{S}egel model in {$\mathbb R^2$}}},
  author={A. Blanchet and J. Carrillo and N. Masmoudi},
  journal={Commun. Pure Appl. Math.},
  volume={61},
  number={10},
  pages={1449--1481},
  year={2008},
  publisher={Wiley Online Library}
}

@article{davila2020existence,
  title={{Existence and stability of infinite time blow-up in the {K}eller-{S}egel system}},
  author={J. Davila and M. del Pino and J. Dolbeault and M. Musso and J. Wei},
  journal={Arch. Ration. Mech. Anal.},
  volume={248},
  number={4},
  pages={61},
  year={2024},
  publisher={Springer}
}

@article{KiselevXu2016,
  title={Suppression of chemotactic explosion by mixing},
  author={A. Kiselev  and  X. Xu },
  journal={Arch. Ration. Mech. Anal.},
  volume={222},
  number={2},
  pages={1077--1112},
  year={2016},
  publisher={Springer}
}

@article{he2023enhanced,
  title={Enhanced dissipation and blow-up suppression in a chemotaxis-fluid system},
  author={S. He},
  journal={SIAM J. Math. Anal.},
  volume={55},
  number={4},
  pages={2615--2643},
  year={2023},
  publisher={SIAM}
}

@article{Nagai2001,
  author  = {T. Nagai},
  title   = {Blowup of nonradial solutions to parabolic-elliptic
             systems modeling chemotaxis in two-dimensional domains},
  journal = {J. Inequal. Appl.},
  volume  = {6},
  number  = {1},
  pages   = {37--55},
  year    = {2001},
  doi     = {10.1155/S1025583401000042}
}

\end{document}